\documentclass[english]{myarticle}

\usepackage{listings}
\usepackage{amsfonts}
\usepackage{graphicx}
\usepackage{amsmath,amsopn}
\usepackage{physics}
\usepackage{mathtools,thmtools}

\DeclareMathOperator{\NBV}{NBV}
\DeclareMathOperator{\Lip}{Lip}
\DeclareMathOperator{\loc}{loc}
\DeclareMathOperator{\inv}{INV}

\usepackage{enumitem}
\setlist[enumerate]{leftmargin=.5in}
\setlist[itemize]{leftmargin=.5in}

\newcommand{\TheTitle}{Numerical Periodic Normalization at Codim 1 Bifurcations of Limit Cycles in DDEs}

\title{\TheTitle}

\author{Maikel M. Bosschaert\thanks{Data Science Institute, Hasselt University, Diepenbeek Campus, 3590 Diepenbeek, Belgium \email{(maikel.bosschaert@uhasselt.be)}.}
	\and Bram Lentjes\thanks{Department of Mathematics, Hasselt University, Diepenbeek Campus, 3590 Diepenbeek, Belgium \email{(bram.lentjes@uhasselt.be)}.}
	\and Len Spek\thanks{Department of Applied Mathematics, University of Twente, 7500 AE Enschede, The Netherlands \email{(l.spek@utwente.nl)}.}
	\and Yuri A. Kuznetsov\thanks{Department of Mathematics, Utrecht University, 3508 TA Utrecht, The Netherlands and Department of Applied Mathematics, University of Twente, 7500 AE Enschede, The Netherlands \email{(i.a.kouznetsov@uu.nl)}.}
}

\begin{document}
\maketitle

\begin{abstract}
Recent work in \cite{Article1,Article2} by the authors on periodic center manifolds and normal forms for bifurcations of limit cycles in delay differential equations (DDEs) motivates the derivation of explicit computational formulas for the critical normal form coefficients of all codimension one bifurcations of limit cycles. In this paper, we derive such formulas via an application of the periodic normalization method in combination with the functional analytic perturbation framework for dual
semigroups (sun-star calculus). The explicit formulas allow us to distinguish between nondegenerate, sub- and supercritical bifurcations. To efficiently apply these formulas, we introduce the characteristic operator as this enables us to use robust numerical boundary-value algorithms based on orthogonal collocation. Although our theoretical results are proven in a more general setting, the software implementation and examples focus on discrete DDEs. The actual implementation is described in detail and its effectiveness is demonstrated on various models.
\end{abstract}


\begin{sloppypar}

\markboth{M.M. Bosschaert, B. Lentjes, L. Spek and Yu.A. Kuznetsov}{Numerical Periodic Normalization at Codim 1 Bifurcations of Limit Cycles in DDEs}

\section{Introduction} \label{sec:introduction}
Isolated periodic orbits (limit cycles) in smooth delay differential equations (DDEs) play an important role in applications. In the simplest case, often encountered in such applications, DDEs have the form
\begin{equation} \label{eq:discreteDDEintro}
    \dot{x}(t) = f(x(t),x(t-\tau_1),\dots,x(t - \tau_m),\alpha), \quad t \geq 0,
\end{equation}
where $x(t) \in \mathbb{R}^n$, $ \alpha \in \mathbb{R}^p$, $f: \mathbb{R}^{n \times (m+1)} \times \mathbb{R}^p \to \mathbb{R}^n$ is sufficiently smooth and the delays $0 < \tau_1 < \dots < \tau_m$ are constant. In generic systems of the form \eqref{eq:discreteDDEintro} depending on one
control parameter ($p = 1$), a hyperbolic limit cycle exists for an open interval of
parameter values $\alpha$. At a boundary of such an interval, the limit cycle may still exist but become nonhyperbolic so that either a \emph{fold} (limit point), or \emph{period-doubling} (flip), or \emph{Neimark-Sacker} (torus) bifurcation occurs. To determine the nature (nondegenerate, sub- or supercritical) of such a bifurcation, one needs to analyze the critical coefficients of an associated normal form. Up to date, a robust and efficient method to compute such normal form coefficients of a general system of the form \eqref{eq:discreteDDEintro} has not been available. Therefore, we present in this paper rather compact explicit computational formulas for the critical normal form coefficients of all codimension one bifurcations of limit cycles in DDEs. Our results are based on the periodic center manifold \cite{Article1} and normal forms \cite{Article2} for DDEs, meaning that we completely avoid Poincar\'e maps. Furthermore, our obtained formulas are implemented in the freely available software \verb|Julia| package \verb|PeriodicNormalizationDDEs| \cite{Bosschaert2024c}. In addition, the effectiveness, robustness and correctness of these formulas will be verified on various models.

\subsection{Background, aims and challenges} \label{subsec:background}
Consider a {\it classical delay differential equation}
\begin{equation} \label{eq:DDE} \tag{DDE} 
    \dot{x}(t) = F(x_t), \quad t \geq 0,
\end{equation}
where $x(t) \in \mathbb{R}^n$ and $x_t \in X \coloneqq C([-h,0],\mathbb{R}^n)$ represents the \emph{history} of the unknown $x$ at time $t$ defined by $x_t(\theta) \coloneqq x(t+\theta)$ for all $\theta \in [-h,0]$. Here, $0 < h < \infty$ denotes the upper bound of (finite) delays and the state space $X$ becomes a Banach space when equipped with the supremum norm. Moreover, we assume that $F : X \to \mathbb{R}^n$ is a $C^{k+1}$-smooth (nonlinear) operator for some integer $k \geq 1$ and that \eqref{eq:DDE} admits a $T$-periodic solution $\gamma : \mathbb{R} \to \mathbb{R}^n$ such that the periodic orbit (cycle) $\Gamma \coloneqq \{ \gamma_t \in X : t \in \mathbb{R} \}$ is nonhyperbolic. Using the perturbation framework of dual semigroups (sun-star calculus), developed in \cite{Clement1988,Clement1987,Clement1989,Clement1989a,Diekmann1991}, the existence of a $T$-periodic $C^k$-smooth finite-dimensional center manifold $\mathcal{W}_{\loc}^c(\Gamma)$ near $\Gamma$ for \eqref{eq:DDE} has been recently rigorously established in \cite{Article1} by the authors. To study the local dynamics of \eqref{eq:DDE} on $\mathcal{W}_{\loc}^c(\Gamma)$, the authors proved in \cite{Article2} the existence of a special coordinate system on $\mathcal{W}_{\loc}^c(\Gamma)$ such that any solution of \eqref{eq:DDE} on this invariant manifold can be locally parametrized in terms of periodic normal forms. This coordinate system and normal forms are based on the work of Iooss (and Adelmeyer) \cite{Iooss1988,Iooss1999}, who originally developed them to analyze bifurcations of limit cycles in finite-dimensional ODEs. Of course, to study these bifurcations one can alternatively use Poincar\'e maps and their corresponding normal forms \cite{Diekmann1995,Guckenheimer1983,Kuznetsov2005,Szuecs2012}. However, applications of these results to the analysis of concrete finite-dimensional ODEs are exceptional, since higher-order derivatives of the Poincar\'e map are required to compute the critical normal forms coefficients \cite{Kuznetsov2023a,Kuznetsov2005}. These higher-order derivatives are hardly available numerically, see \cite{Witte2014,Witte2013} and the references therein for more details.

To overcome this problem in finite-dimensional ODEs, one would rather like to compute the critical normal form coefficients of the associated periodic normal forms. This task was accomplished by Kuznetsov et al. in \cite{Kuznetsov2005,Witte2014,Witte2013} for all generic codimension one and two bifurcations by using the \emph{periodic normalization method}, see also \cite[Section 5.6]{Kuznetsov2023a} for a gentle introduction on this topic. One of the advantages of this normalization method is that the center manifold reduction and the calculation of the normal form coefficients are performed simultaneously by solving the so-called \emph{homological equation}. The method gives compact explicit
expressions for the critical normal form coefficients in terms of the solutions of certain boundary-value problems. Such formulation enables the use of robust numerical boundary-value solvers based on orthogonal collocation. Furthermore, the computational formulas for all codimension one \cite{Kuznetsov2005} and two \cite{Witte2014,Witte2013} bifurcations of limit cycles have been implemented in the \verb|MATLAB| continuation package \verb|MatCont| \cite{Matcont} with several examples to illustrate their effectiveness.

Going back to the setting of classical DDEs, comparable computational formulas for the critical normal form coefficients of all codimension one and two bifurcations of limit cycles were up-to-date, not available. There are, however, some approaches in the literature \cite[Section 3.2]{Qesmi2018}. For example, Szalai and St\'ep\'an \cite{Szalai2010} derived for an explicit model with one delay the critical normal form coefficients of the Poincar\'e map for the period-doubling bifurcation, but under the very restrictive condition that the delay must be equal to the period of the bifurcating cycle. Similar results were obtained by R\"ost \cite{Roest2005,Roestproceedings,Roest2006} for the Neimark-Sacker bifurcation when the period of the bifurcating cycle is an integer multiple of the delay. These restrictive conditions are not suited for studying bifurcations of limit cycles in general systems of the form \eqref{eq:DDE}, where the period of the cycle usually varies with parameters. We also refer to the work from Church and Liu \cite{Church2018,Church2019,Church2021,Church2020} on some analyses of codimension one bifurcations for periodic impulsive DDEs by employing Poincar\'e maps. They could obtain computational formulas in simple models when the periodic solution and the (adjoint) eigenfunctions are known explicitly. Lastly, we highlight recent works of Z\'athureck\'y \cite{Zathurecky2023,Zathurecky2023a} on computational formulas for critical normal form coefficients of the period-doubling bifurcation in DDEs (and ODEs) using the Lyapunov-Schmidt reduction \cite{Golubitsky1985}. While his formulas are nice and compact, implementing them in software remains challenging. Moreover, Z\'athureck\'y's approach has theoretical limitations, particularly in studying the Neimark-Sacker bifurcation and codimension two bifurcations involving resonances. Therefore, we derive in this paper explicit computational formulas for the critical normal form coefficients of all codimension one bifurcations of limit cycles for classical DDEs by combining periodic normalization with sun-star calculus. Notably, our formulas closely resemble their finite-dimensional ODE counterparts \cite{Kuznetsov2005,Witte2014,Witte2013,Kuznetsov2023a}, which can be obtained by taking the limit as the delay $h \downarrow 0$ in our expressions.

Due to the infinite-dimensional nature of \eqref{eq:DDE}, one would expect that the implementation of our formulas into a software package would be difficult and computationally expensive. However, if one takes a look at the explicit computational formulas for the (critical) normal coefficients of all codimension one and two bifurcations of equilibria in classical DDEs \cite{Bosschaert2020,Bosschaert2024a,Janssens2010}, one notices immediately that these formulas only rely on the computation of finite-dimensional objects. Indeed, the existence of a characteristic matrix \cite{Kaashoek1992,Diekmann1995,Kaashoek2022} to study spectral properties of linear autonomous DDEs allows one to reduce infinite-dimensional spectral problems in $X$ towards linear problems in $\mathbb{C}^n$. Using this reduction, such explicit computational formulas can be derived in terms of Jordan chains of the (adjoint) characteristic matrix.

The natural question arises of whether such a characteristic matrix approach is also suited to study the spectral properties of periodic linear DDEs. Over the years, several characteristic matrices have been developed, though often under strict conditions. For example, Hale and Verduyn Lunel established in \cite{Hale1993} a characteristic matrix when the delay is a multiple of the period. By considering various time-delayed models, it was observed in 2000 by Just \cite{Just2000}, in 2001 by Verduyn Lunel \cite{Lunel2001}, and in 2005 by Skubachevskii and Walter \cite{Skubachevskii2006}, that analyzing the spectral data is more challenging when the period becomes more irrational relative to the delay. It lasted until 2006 that Szalai et al. \cite{Szalai2006} provided a general construction of a characteristic matrix $\Delta_n(z) \in \mathbb{C}^{n \times n}$ for periodic linear DDEs. However, Sieber and Szalai \cite{Sieber2011} founded that $\Delta_n$ typically has poles in the complex plane, which could coincide with a Floquet multiplier of interest. Therefore, they modified $\Delta_n(z)$ to $\Delta_{nk}(z) \in \mathbb{C}^{nk \times nk}$ such that $\Delta_{nk}$ has no poles inside a disk of finite radius (dependent on $k$) in the complex plane, but may still have poles outside this disk. We also mention a recent work by Kaashoek and Verduyn Lunel \cite[Chapter 10]{Kaashoek2022} on characteristic matrices for periodic linear DDEs, where the monodromy operator is written as a sum of a Volterra and a finite-rank operator. This decomposition ensures the existence of a characteristic matrix \cite[Chapter 6]{Kaashoek2022}, which works effectively when the period is an integer multiple of the delay. Their framework was also recently applied in \cite{Wolff2022a} to study symmetric periodic linear DDEs, relevant to equivariant Pyragas control \cite{Pyragas1992}. While they assert that their results hold for more irrational periods relative to the delay \cite[Section 11.1]{Kaashoek2022}, the construction becomes more complex and will be detailed in a future publication.

These approaches show that a simple characteristic matrix construction does not naturally fit within the framework of periodic linear DDEs. To address this issue, we introduce the \emph{characteristic operator} $\Delta(z)$, which is more suitable in this context. Although $\Delta(z)$ acts on an infinite-dimensional function space, it reduces our spectral problems dramatically. Specifically, it induces linear (inhomogeneous) DDEs on the interval $[0,T]$ with $T$-(anti)periodic boundary conditions, allowing the use of the well-known orthogonal collocation method \cite{Engelborghs2001,Engelborghs2002a} to approximate the solution. Through this reduction, our explicit computational formulas for the critical normal form coefficients of all codimension one bifurcations of limit cycles can be expressed in terms of the Jordan chains of the (adjoint) characteristic operator, mirroring the computational reduction used for bifurcations of equilibria in classical DDEs.

\subsection{Overview}
The paper is organized as follows. In \cref{sec:periodic center manifolds} we review the most important results from \cite{Article1,Article2} on periodic smooth finite-dimensional center manifolds and normal forms for all codimension one bifurcations of limit cycles in classical DDEs.

In \cref{sec:periodic spectral computations} we derive explicit computational formulas for the (adjoint) (generalized) eigenfunctions in terms of the periodic Jordan chains of the (adjoint) characteristic operator. When we derive the critical normal form coefficients in the upcoming section, it turns out that we will encounter periodic linear operator equations of a certain type. Therefore, we study in addition the solvability of such equations and invoke a suited Fredholm alternative.

In \cref{sec:periodic normalization} we derive the explicit computational formulas for the critical normal form coefficients of the fold \eqref{eq:normalformcoeff fold}, period-doubling \eqref{eq:normalformcoeff pd} and Neimark-Sacker bifurcation \eqref{eq:normalformcoeff NS} in classical DDEs.

In \cref{sec:implementation} we describe the implementation of the critical normal form coefficients for discrete DDEs in detail. Along with our proposed direct implementation scheme, we also present an efficient method to numerically evaluate our formulas. 

In \cref{sec:examples} we demonstrate the correctness, effectiveness and accuracy of our implemented numerical methods through various models and examples. Additionally, we show that our approach is more efficient and reliable than the pseudospectral discretization method for DDEs \cite{Breda2016,Breda2022,Breda2022b,Diekmann2020,Ando2022,Wolff2021}. A detailed explanation of our efficient numerical implementation is provided in the (online) \hyperlink{mysupplement}{Supplement}.

\section{Periodic center manifolds and normal forms for DDEs} \label{sec:periodic center manifolds}
In this section, we primarily summarize the results from \cite{Article1,Article2} on periodic center manifolds and normal forms for bifurcations of limit cycles in classical DDEs. All unreferenced claims related to basic properties of (time-dependent) perturbations of delay equations can be found in the article \cite{Clement1988} and the book \cite{Diekmann1995}. For an introduction to the general theory on (adjoint) semigroups, we refer to the well-known books \cite{Engel2000,Neerven1992}.

Consider a $C^{k+1}$-smooth operator $\hat{F} : X \times \mathbb{R} \to \mathbb{R}^n$ together with the (nonlinear) DDE
\begin{equation} \label{eq:DDEparameter}
    \dot{x}(t) = \hat{F}(x_t,\alpha), \quad t \geq 0,
\end{equation}
where $\alpha \in \mathbb{R}$ is a \emph{parameter}. Let $\alpha_0 \in \mathbb{R}$ denote a \emph{critical parameter value} of interest and consider \eqref{eq:DDE} with $F = \hat{F}(\cdot,\alpha_0)$. To study the dynamics near the periodic solution $\gamma$, it is convenient to translate $\gamma$ towards the origin. More specifically, if $x$ is a solution of \eqref{eq:DDE}, then $y \coloneqq x - \gamma$ satisfies the periodic (nonlinear) DDE 
\begin{equation} \label{eq:T-DDEphi1}
\dot{y}(t) = L(t)y_t + G(t,y_t), \quad t \geq 0, \\\
\end{equation}
where the $C^k$-smooth bounded linear operator $L(t) \coloneqq DF(\gamma_t) \in \mathcal{L}(X,\mathbb{R}^n)$ denotes the Fr\'echet derivative of $F$ evaluated at $\gamma_t$ and the $C^k$-smooth operator $G(t,\cdot) \coloneqq F(\gamma_t + \cdot) - F(\gamma_t) - L(t)$ consists solely of nonlinear terms. Regarding the linear part in \eqref{eq:T-DDEphi1}, it is traditional to apply a vector-valued version of the Riesz representation theorem \cite[Theorem 1.1]{Diekmann1995} as
\begin{equation} \label{eq:L(t)varphi}
    L(t)\varphi = \int_0^h d_2 \zeta(t,\theta) \varphi(-\theta) \eqqcolon \langle \zeta(t,\cdot),\varphi \rangle, \quad \forall t \in \mathbb{R}, \ \varphi \in X.
\end{equation}
The \emph{kernel} $\zeta : \mathbb{R} \times [0,h] \to \mathbb{R}^{n \times n}$ is a matrix-valued function, $\zeta(t,\cdot)$ is of bounded variation, right-continuous on the open interval $(0,h)$, $T$-periodic in the first variable and normalized by the requirement that $\zeta(\cdot,0) = 0$. The integral appearing in \eqref{eq:L(t)varphi} is of Riemann-Stieltjes type and the subscript in $d$ reflects on the fact that we integrate over the second variable of $\zeta$.

Let $U \coloneqq \{U(t,s)\}_{t \geq s} \subseteq \mathcal{L}(X)$ denote the strongly continuous forward evolutionary system from \cite[Theorem XII.2.7 and XII.3.1]{Diekmann1995} corresponding to the linearization of \eqref{eq:DDE} around $\Gamma$, that is \eqref{eq:T-DDEphi1} with $G = 0$. The associated (generalized) generator of $U$ (at time $s$) reads
\begin{equation} \label{eq:D(As)}
    \mathcal{D}(A(s)) = \{ \varphi \in C^1([-h,0],\mathbb{R}^n) : \varphi'(0) = L(s) \varphi \}, \quad A(s)\varphi = \varphi'.
\end{equation}
Due to compactness, the spectrum $\sigma(U(s+T,s))$ of the \emph{monodromy operator} (at time $s$) $U(s+T,s)$ is a countable set in $\mathbb{C}$, independent of the starting time $s$, consisting of $0$ and isolated eigenvalues of finite type, called \emph{Floquet multipliers}, that can possibly accumulate to $0$. It is known that $1$ is always a Floquet multiplier, called the \emph{trivial Floquet multiplier}, with associated eigenfunction $\dot{\gamma}_s$. The number $\sigma \in \mathbb{C}$ satisfying $\lambda = e^{\sigma T}$ is called the \emph{Floquet exponent} (of $\lambda$) and note that this number is uniquely determined up to an additive multiple of $\frac{2 \pi i }{T}$. The \emph{(generalized) eigenspace} (at time $s$) associated to $\lambda$, denoted by $E_\lambda(s) \coloneqq \mathcal{N} ((\lambda I - U(s+T,s))^{k_\lambda})$, is finite-dimensional, where $1 \leq k_\lambda < \infty$ is the order of a pole $\lambda = \mu$ of the map $\mu \mapsto (\mu I - U(s+T,s))^{-1}$. The dimension $m_\lambda$ of $E_\lambda(s)$ is called the \emph{algebraic multiplicity} (of $\lambda$) while the \emph{geometric multiplicity} (of $\lambda$) is the dimension of the eigenspace $\mathcal{N} (\lambda I - U(s+T,s))$. Hence, the set of Floquet multipliers on the unit circle $\Lambda_0 \coloneqq \{ \lambda \in \sigma(U(s+T,s)) : |\lambda| = 1\}$ is finite, say it has cardinality $1 \leq n_0 + 1 < \infty$ when counted with algebraic multiplicity. Define the $(n_0 +1)$-dimensional \emph{center subspace} (at time $s$) by $X_0(s) \coloneqq \oplus_{\lambda \in \Lambda_0} E_{\lambda}(s)$ and let $X_0 \coloneqq \{ (t,\varphi) \in \mathbb{R} \times X : \varphi \in X_0(t) \}$ denote the \emph{center bundle}. In this setting, the periodic local center manifold theorem \cite[Corollary 15]{Article1} for \eqref{eq:DDE} applies. Hence, there exists a $T$-periodic $C^k$-smooth $(n_0+1)$-dimensional locally positively invariant manifold $\mathcal{W}_{\loc}^{c}(\Gamma)$ around $\Gamma$, called the \emph{(local) center manifold around $\Gamma$}, whose tangent bundle is $X_0$.

Our next objective is to revisit how linear periodic classical DDEs naturally align with the framework of sun-star calculus. Let us start with a convention. For $\mathbb{K} \in \{ \mathbb{R}, \mathbb{C} \}$ let $\mathbb{K}^n$ be the linear space of column vectors, while $\mathbb{K}^{n \star}$ denotes the linear space of row vectors, all with components in $\mathbb{K}$. A representation theorem by Riesz \cite{Riesz1914} enables us to identify $X^\star = C([-h,0],\mathbb{R}^n)^\star$ with the Banach space $\NBV([0,h],\mathbb{R}^{n\star})$ consisting of functions $\zeta : [0,h] \to \mathbb{R}^{n \star}$ that are normalized by $\zeta(0) = 0$, are continuous from the right on $(0,h)$ and have bounded variation. Because $X$ is not reflexive, the dual backward evolutionary system $U^\star \coloneqq \{U^\star(s,t)\}_{s \leq t} \subseteq \mathcal{L}(X^\star)$, where $U^\star(s,t) \coloneqq U(t,s)^\star$ for all $s \leq t$ is only weak$^\star$ continuous. This is also visible on the generator level, as the adjoint $A^\star(s) \coloneqq [A(s)]^\star$ of $A(s)$ is only the weak$^\star$ (generalized) generator of $U^\star$ and there holds
\begin{align} \label{eq:D(Astars)}
    \mathcal{D}(A^\star(s)) = \bigg\{  f \in & \NBV([0,h],\mathbb{R}^{n \star}) : \  f(\theta) =  f(0^+) + \int_0^\theta g(\sigma) d\sigma \mbox{ for all }  \theta \in (0,h], \nonumber \\
     & \ g \in \NBV([0,h],\mathbb{R}^{n\star}) \mbox{ and }  g(h) = 0 \bigg \}, \quad A^\star(s) f = f' + f(0^+) \zeta(s,\cdot),
\end{align}
where $f(0^+) \coloneqq \lim_{t \downarrow 0} f(t)$ and the function $g$ is called the \emph{derivative} of $f$. The $s$-independent \emph{sun dual} $X^\odot$ is the maximal subspace of $X^\star$ where $t \mapsto U^\star(s,t)$ is strongly continuous, and it follows from \cite[Lemma 1]{Article2} that $\mathcal{D}(A^\star(s))$ is norm dense in $X^\odot$. Since the $\NBV$-functions are dense in $L^1$, $X^\odot$ has the same description as $\mathcal{D}(A^\star(s))$, but $g \in L^1([0,h],\mathbb{R}^{n\star})$. As an element $f \in X^\odot$ is completely specified by $f(0^+) \in \mathbb{R}^{n \star}$ and $g \in L^1([0,h],\mathbb{R}^{n\star})$, we obtain the isomorphism $X^{\odot} \cong \mathbb{R}^{n \star} \times L^1([0,h],\mathbb{R}^{n\star})$. The natural dual pairing $\langle \cdot , \cdot \rangle : X^\odot \times X \to \mathbb{R}$ between $\varphi^\odot = (c,g) \in X^\odot$ and $\varphi \in X$ reads
\begin{equation} \label{eq:Xsunpairing} 
    \langle \varphi^\odot, \varphi \rangle = c \varphi(0) + \int_0^h g(\theta) \varphi(-\theta) d\theta.
\end{equation}
Restricting $U^\star(s,t)$ to $X^\odot$ for all $s \leq t$ yields a strongly continuous backward evolutionary system $U^\odot \coloneqq \{U^\odot(s,t)\}_{s \leq t}$ on the Banach space $X^\odot$ with (generalized) generator $A^\odot(s) \coloneqq [A(s)]^\odot$ given by the part of $A^\star(s)$ in $X^\odot$, i.e.,
\begin{equation} \label{eq:D(Asuns)}
    \mathcal{D}(A^\odot(s)) = \{(c,g) \in \mathcal{D}(A^\star(s)) : g+c\zeta(s,\cdot) \in X^{\odot} \}, \quad A^\odot(s)(c,g) = g + c\zeta(s,\cdot).
\end{equation}
Note that $A^\odot(s)$ is densely defined closed linear operator on $X^\odot$. Repeating the construction once more yields the weak$^\star$ (generalized) generator
\begin{equation*}
    \mathcal{D}(A^{\odot \star}(s)) = \{ (\alpha,\varphi) \in X^{\odot \star} : \varphi \in \Lip([-h,0],\mathbb{R}^n) \mbox{ and } \varphi(0) = \alpha \}, \quad A^{\odot \star}(s)j\varphi = (L(s)\varphi,\varphi'),
\end{equation*}
where the dual space $X^{\odot \star}$ of $X^{\odot}$, and its restriction to the maximal subspace of strong continuity $X^{\odot \odot}$, have the representations $X^{\odot \star} \cong \mathbb{R}^{n} \times L^\infty([0,h], \mathbb{R}^{n})$ and $X^{\odot \odot} \cong \mathbb{R}^{n} \times C([0,h], \mathbb{R}^{n})$. The linear canonical embedding $j: X \to X^{\odot \star}$ has action $j\varphi = (\varphi(0),\varphi)$, mapping $X$ onto $X^{\odot \odot}$, meaning that $X$ is $\odot$-\emph{reflexive} with respect to $U$. The natural dual pairing $\langle \cdot, \cdot \rangle : X^{\odot \star} \times X^\odot \to \mathbb{R}$ between $\varphi^{\odot \star} = (a,\psi) \in X^{\odot \star}$ and $\varphi^\odot = (c,g) \in X^{\odot}$ is given by
\begin{equation} \label{eq:Xsunstarpairing}
    \langle \varphi^{\odot \star}, \varphi^\odot \rangle = ca + \int_0^h g(\theta) \psi(-\theta) d\theta.
\end{equation}

\begin{remark} \label{remark:DDEperturbation}
$A^{\odot \star}(s)$ can also be obtained as a bounded perturbation of the closed linear operator $A_0^{\odot \star}$ since $\mathcal{D}(A^{\odot \star}(s)) = \mathcal{D}(A_0^{\odot \star})$ and $A^{\odot \star}(s)j = A_0^{\odot \star}j + B(s)$, see \cite[Chapter II and III]{Diekmann1995} and \cite[Section 2]{Article2}. Here, the $T$-periodic bounded linear perturbation $B \in C^k(\mathbb{R},\mathcal{L}(X,X^{\odot \star}))$ is defined by
\begin{equation} \label{eq:perturbationB}
    B(t)\varphi \coloneqq [L(t)\varphi]r^{\odot \star}, \quad \forall t \in \mathbb{R}, \ \varphi \in X,
\end{equation}
and we used the conventional shorthand notation $wr^{\odot \star} \coloneqq \sum_{i = 1}^{n} w_ir_i^{\odot \star}$ for all $w = (w_1,\dots,w_n) \in \mathbb{R}^{n}$, where $r_i^{\odot \star} \coloneqq (e_i,0)$ and $e_i$ denotes the $i$th standard basic vector of $\mathbb{R}^n$ for all $i = 1,\dots,n$. \hfill $\lozenge$
\end{remark}

\subsection{Periodic normal forms for codim 1 bifurcations} \label{subsec:periodic normal forms}
Our next aim is to describe the local dynamics on the center manifold $\mathcal{W}_{\loc}^{c}(\Gamma)$ in terms of normal forms. This task has been accomplished by the authors in \cite{Article2} via lifting the results from \cite{Iooss1988,Iooss1999} on periodic normal forms from finite-dimensional ODEs towards the setting of classical DDEs. It turns out that any solution of \eqref{eq:DDE} on the invariant manifold $\mathcal{W}_{\loc}^{c}(\Gamma)$ can be parametrized locally near the cycle $\Gamma$ in $(\tau,\xi)$-coordinates, where $\tau \in [0,lT]$ for some $l \in \mathbb{N}$ is a phase coordinate along $\Gamma$, and $\xi$ is a real or complex coordinate along this phase that is transversal to $\Gamma$, see \cite[Figure 1]{Article2} or \cite[Figure 5.18 and 5.20]{Kuznetsov2023a} for a visualization. The number $l$ depends on the location and multiplicities of the Floquet multipliers on the unit circle. Moreover, the periodic normal forms for the classical DDEs turn out to be exactly the same as for the finite-dimensional ODEs, see in particular \cite[Theorem 24, 25 and 26]{Article2} compared to \cite[Theorem III.7, III.10 and III.13]{Iooss1999}. As we are interested in all generic codimension one bifurcations of limit cycles, we will restrict ourselves to the fold ($l=1$), period-doubling ($l=2$) and Neimark-Sacker ($l=1$) 
bifurcation. Explicit computations of the periodic normal forms for these bifurcations in ODEs, and thus classical DDEs, can be found in \cite[Section III.2]{Iooss1999} and \cite[Appendix A]{DellaRossa2011}, which is an extended arXiv version of \cite{Witte2013}. To discuss the local dynamics of \eqref{eq:DDEparameter} in the vicinity of $\Gamma$ for parameter values $\alpha$ near the critical parameter value $\alpha_0$, we will assume that \eqref{eq:DDEparameter} is \emph{generic}, meaning that the \emph{nondegenarcy} and \emph{transversality conditions} are satisfied, see \cite[Section 2]{Kuznetsov2023a} for more information. The nondegenarcy conditions are specified for each bifurcation individually, while a discussion on the transversality conditions is provided \cref{remark:parameterdependentCMT}.

\subsubsection{Fold bifurcation}
The critical cycle $\Gamma$ undergoes a \emph{fold} (limit point) bifurcation if $\Lambda_0 = \{1\}$, and the trivial Floquet multiplier has algebraic multiplicity two and geometric multiplicity one. The corresponding $2$-dimensional center manifold $\mathcal{W}_{\loc}^{c}(\Gamma)$ can be parametrized locally near $\Gamma$ in $(\tau,\xi)$-coordinates, where $\tau \in [0,T]$ and $\xi$ is a real coordinate on $\mathcal{W}_{\loc}^{c}(\Gamma)$ transverse to $\Gamma$. It follows from \cite[Theorem 25]{Article2} that the critical periodic normal form at the fold bifurcation is 
\begin{equation} \label{eq:normal form fold}
    \begin{dcases}
    \dot{\tau} = 1 + \xi + a\xi^2 + \mathcal{O}(\xi^3),\\
    \dot{\xi} = b\xi^2 + \mathcal{O}(\xi^3),
    \end{dcases}
\end{equation}
where $a,b \in \mathbb{R}$ and the $\mathcal{O}(\xi^3)$-terms are $T$-periodic in $\tau$. The dot above $\tau$ and $\xi$ in \eqref{eq:normal form fold} refers to a derivative with respect to the original time variable $t$. This convention will also be used in the upcoming periodic normal forms, see \eqref{eq:normal form pd} and \eqref{eq:normal form NS}. The fold bifurcation is nondegenerate if $b\neq 0$ and then two periodic solutions of \eqref{eq:DDEparameter} collide at $\Gamma$ and disappear when the parameter $\alpha$ passes the critical value $\alpha_0$, see \cite[Figure 5.19]{Kuznetsov2023a} for a visualization. This normal form coefficient serves as a test function to detect a codimension two bifurcation of limit cycles, namely a \emph{cusp} (CPC) bifurcation occurs at $b=0$. 

\subsubsection{Period-doubling bifurcation}
The critical cycle $\Gamma$ undergoes a \emph{period-doubling} (flip) bifurcation if $\Lambda_0 = \{\pm 1\}$, and both Floquet multipliers are simple. The corresponding $2$-dimensional center manifold $\mathcal{W}_{\loc}^{c}(\Gamma)$ can be parametrized locally near $\Gamma$ in $(\tau,\xi)$-coordinates, where $\tau \in [0,2T]$ and $\xi$ is a real coordinate on $\mathcal{W}_{\loc}^{c}(\Gamma)$ transverse to $\Gamma$. It follows from \cite[Theorem 26]{Article2} that the critical periodic normal form at the period-doubling bifurcation is 
\begin{equation} \label{eq:normal form pd}
    \begin{dcases}
    \dot{\tau} = 1 + a\xi^2 + \mathcal{O}(\xi^4),\\
    \dot{\xi} = c\xi^3 + \mathcal{O}(\xi^4),
    \end{dcases}
\end{equation}
where $a,c \in \mathbb{R}$ and the $\mathcal{O}(\xi^4)$-terms are $2T$-periodic in $\tau$. The critical cycle is stable within $\mathcal{W}_{\loc}^c(\Gamma)$ if $c < 0$ (supercritical) and unstable if $c > 0$ (subcritical). In the former case, a stable cycle of approximately double period spawns off from $\Gamma$ in \eqref{eq:DDEparameter} when the parameter $\alpha$ passes the critical value $\alpha_0$, see \cite[Figure 5.21]{Kuznetsov2023a} for a visualization. In the latter case, an unstable cycle of approximately double period spawns off from $\Gamma$ in \eqref{eq:DDEparameter} when the parameter $\alpha$ passes the critical value $\alpha_0$. This normal form coefficient serves as a test function to detect a codimension two bifurcation of limit cycles, namely a \emph{generalized period-doubling} (GPD) bifurcation occurs at $c=0$.

\subsubsection{Neimark-Sacker bifurcation}
The critical cycle $\Gamma$ undergoes a \emph{Neimark-Sacker} (torus) bifurcation if $\Lambda_0 = \{1,e^{\pm i \omega T}\}$, where all Floquet multipliers on the unit circle are simple and do not satisfy any strong resonance conditions, i.e.,
\begin{equation} \label{eq:resonances}
    e^{iq \omega T} \neq 1, \quad \forall q \in \{1,2,3,4\}.
\end{equation}
The corresponding $3$-dimensional center manifold $\mathcal{W}_{\loc}^{c}(\Gamma)$ can be parametrized locally near $\Gamma$ in $(\tau,\xi)$-coordinates, where $\tau \in [0,T]$ and $\xi$ is a complex coordinate on $\mathcal{W}_{\loc}^{c}(\Gamma)$ transverse to $\Gamma$. It follows from \cite[Theorem 24]{Article2} that the critical periodic normal form at the Neimark-Sacker bifurcation is 
\begin{equation} \label{eq:normal form NS}
    \begin{dcases}
    \dot{\tau} = 1 + a|\xi|^2 + \mathcal{O}(|\xi|^4),\\
    \dot{\xi} =  i \omega \xi + d\xi|\xi|^2 + \mathcal{O}(|\xi|^4),
    \end{dcases}
\end{equation}
where $a \in \mathbb{R}, d \in \mathbb{C}$ and the $\mathcal{O}(|\xi|^4)$-terms are $T$-periodic in $\tau$. The critical cycle is stable within $\mathcal{W}_{\loc}^{c}(\Gamma)$ if $\Re d < 0$ (supercritical) and unstable if $\Re d > 0$ (subcritical). In the former case, there exists a unique stable invariant torus that spawns off from $\Gamma$ in \eqref{eq:DDEparameter} when $\alpha$ passes the critical value $\alpha_0$, see \cite[Figure 5.22]{Kuznetsov2023a} for a visualization. In the latter case, there exists a unique unstable invariant torus that spawns off from $\Gamma$ in \eqref{eq:DDEparameter} when $\alpha$ passes the critical value $\alpha_0$. This normal form coefficient serves as a test function to detect a codimension two bifurcation of limit cycles, namely a \emph{Chenciner} (CH) bifurcation occurs at $ \Re d = 0$.

\begin{remark} \label{remark:parameterdependentCMT}
While the nondegeneracy conditions can be specified from the critical periodic normal form, the transversality conditions can only be derived from the parameter-dependent periodic normal form. One can prove along the lines of \cite{Article2} that these normal forms for finite-dimensional ODEs \cite[Theorem III.19]{Iooss1999} coincide with that of classical DDEs. Hence, the transversality conditions for all generic codimension one bifurcations in finite-dimensional ODEs, and thus classical DDEs, can be found in \cite[Section 5.6.2]{Kuznetsov2023a}. However, as the transversality conditions are in general satisfied for systems of the form \eqref{eq:DDEparameter}, we will not discuss these here. \hfill $\lozenge$
\end{remark}

\section{Periodic spectral computations} \label{sec:periodic spectral computations}
As we are interested in computing critical normal form coefficients via periodic normalization, it is clear from the finite-dimensional periodic normalization ODE-articles \cite{Kuznetsov2005,Witte2013,Witte2014} that we have to compute (adjoint) (generalized) eigenfunctions of the monodromy operator. Suppose that $\lambda$ is a Floquet multiplier of algebraic multiplicity $m_\lambda$. If $\{\phi_s^0,\dots,\phi_s^{m_\lambda - 1} \}$ is an ordered basis of $E_\lambda(s)$, it follows from \cite[Theorem XIII.3.3]{Diekmann1995} that $\{U(\tau,s)\phi_s^0,\dots,U(\tau,s)\phi_s^{m_\lambda - 1} \}$ is an ordered basis of $E_\lambda(\tau)$ for all $\tau \in \mathbb{R}$. However, this basis is not necessarily $T$-periodic, and is therefore not suited for applying periodic normalization. Furthermore, we also need a $T$-periodic dual basis of $E_\lambda(s)$, that is a $T$-periodic basis of the $m_\lambda$-dimensional subspace $E_\lambda^\odot(s) \coloneqq \mathcal{N}((\lambda I - U^\odot(s-T,s))^{k_\lambda})$ of $X^\odot$, where $k_\lambda$ was already defined in \cref{sec:periodic center manifolds}. The functions spanning a basis of $E_\lambda^\odot(s)$ are called \emph{adjoint (generalized) eigenfunctions}. 

The existence of $T$-periodic and sufficiently smooth (adjoint) (generalized) eigenfunctions for classical DDEs has already been proven by the authors in \cite{Article2}, but an explicit representation of these bases was not yet available. In \cref{subsec: Jordan chains via characteristic operators} we show that the (adjoint) (generalized) eigenfunctions can be computed in terms of Jordan chains of the so-called \emph{(adjoint) characteristic operator}, a generalization of the well-known (adjoint) characteristic matrix \cite{Diekmann1995,Kaashoek1992,Kaashoek2022}. Moreover, when applying periodic normalization in \cref{sec:periodic normalization} to classical DDEs, it turns out that we have to solve periodic linear operator equations on an infinite-dimensional space consisting of periodic functions. The solvability of these linear operator equations will be discussed in \cref{subsec:solvability}. Let us make the following remark regarding spectral theory for DDEs.

\begin{remark} \label{remark:complexification} 
For using spectral theory on the real Banach space $X$, we have to \emph{complexify} $X$ and all discussed operators on $X$. This is not entirely trivial and is discussed in \cite[Section III.7 and Section IV.2]{Diekmann1995}. To clarify, by the spectrum of a real (unbounded) linear operator $L$ defined on (a subspace of) $X$, we mean the spectrum of its complexification $L_\mathbb{C}$ on (a subspace of) the complexified Banach space $X_\mathbb{C}$. For the ease of notation, we omit the additional symbols. \hfill $\lozenge$
\end{remark}

\subsection{Periodic smooth Jordan chains via characteristic operators} \label{subsec: Jordan chains via characteristic operators}
For a (complex) Banach space $E$ and integer $l \geq 0$, we define $C_T^l(\mathbb{R},E)$ as the (complex) Banach space consisting of $T$-periodic $C^l$-smooth $E$-valued functions defined on $\mathbb{R}$ with the standard $C^l$-norm. 

For the construction of the $T$-periodic (generalized) eigenfunctions, it is clear that we have to look for functions $\varphi_i \in C_T(\mathbb{R},X)$, where $X = C([-h,0],\mathbb{R}^n)$. It looks at first sight not easy to find a formula for such functions that is at a later stage easy to implement in a software package. This is precisely where the \emph{characteristic operator} $\Delta$ will help us a lot. The main idea of the characteristic operator is to reduce spectral problems in $C_T(\mathbb{R},X)$ towards linear problems in $C_T(\mathbb{R},\mathbb{C}^n)$. This reduction is also visible in the original characteristic matrix approach since in that setting, the spectral problems in $X$ are reduced to linear problems in $\mathbb{C}^n$, see \cite[Section IV.4]{Diekmann1995} and \cite[Section II.1]{Kaashoek1992} for more information. The idea of the construction of the characteristic operator is not to work directly with the monodromy operator, but rather with a Floquet eigenfunction. Indeed, recall from \cite[Lemma 8.1.2]{Hale1993} that $\sigma \in \mathbb{C}$ is a Floquet exponent if and only if there exists a nonzero function $v$ of the form $v(\tau) = e^{\sigma \tau}q(\tau)$ with $q \in C_T^1(\mathbb{R},\mathbb{R}^n)$ that solves the periodic linear DDE \eqref{eq:T-DDEphi1} with $G = 0$. Plugging everything into this equation while using \eqref{eq:L(t)varphi} yields
\begin{equation} \label{eq:motivationcharac}
    \dot{q}(\tau) + \sigma q(\tau) - \int_0^h d_2 \zeta (\tau,\theta) e^{-\sigma \theta} q(\tau-\theta) = 0, \quad \forall \tau \in \mathbb{R}.
\end{equation}
To cast this equation into a functional analytic setting, introduce for any $z \in \mathbb{C}$ the densely defined linear operator $\Delta(z) : \mathcal{D}(\Delta(z)) \coloneqq C_T^1(\mathbb{R},\mathbb{C}^n) \subseteq C_T(\mathbb{R},\mathbb{C}^n) \to C_T(\mathbb{R},\mathbb{C}^n)$ by
\begin{equation} \label{eq:Delta(z)q}
    (\Delta(z)q)(\tau) \coloneqq \dot{q}(\tau) + z q(\tau) - \int_0^h d_2 \zeta (\tau,\theta) e^{-z \theta} q(\tau-\theta), \quad \forall q \in C_T^1(\mathbb{R},\mathbb{C}^n),
\end{equation}
Hence, $\sigma$ is a Floquet exponent if and only if there exists a nonzero $q \in C_T^1(\mathbb{R},\mathbb{C}^n)$ such that $\Delta(\sigma)q = 0$, i.e. $\sigma$ is a \emph{characteristic value} of $\Delta$. In the following lemma, we will study some properties of the \emph{characteristic operator} $\Delta(z)$. To do so, let us first recall some notions from spectral theory for (unbounded) linear operators on a Banach space $E$, see \cite{Engel2000,Hille1957,Taylor1986} for more information. For a closable linear operator $A: \mathcal{D}(A) \subseteq E \to E$, we say that a complex number $z$ belongs to the \emph{resolvent set} $\rho(A)$ of $A$ if the operator $zI - A$ is injective, has dense range $\mathcal{R}(zI-A) \supseteq \mathcal{D}(A)$ and the \emph{resolvent} of $A$ at $z$ defined by $(z I - A)^{-1}$ is a bounded linear operator from $\mathcal{R}(zI-A)$ to $\mathcal{D}(A)$. The \emph{spectrum} $\sigma(A)$ of $A$ is defined to be the complement of $\rho(A)$ in $\mathbb{C}$, and the \emph{point spectrum} $\sigma_p(A)$ of $A$ is the set of those $\mu \in \mathbb{C}$ such that $\mu I - A$ is not injective, i.e. $A \varphi = \mu \varphi$ for some nonzero \emph{eigenvector} $\varphi \in \mathcal{D}(A)$ corresponding to the \emph{eigenvalue} $\mu$. When $A$ is a closed linear operator, recall from the closed graph theorem that $z \in \rho(A)$ if and only if $zI - A$ is a linear bijection.

\begin{lemma} \label{lemma:Delta(z)closed}
The operator $\Delta(z)$ is closed and $\rho(\Delta(z)) \neq \emptyset$ for all $z \in \mathbb{C}$, and there holds
\begin{align} 
\begin{split} \label{eq:not inv is floquet}
    \{ \sigma \in \mathbb{C} : \Delta(\sigma) \mbox{ is not invertible} \} &= \{ \sigma \in \mathbb{C} : \sigma \mbox{ is a characteristic value of } \Delta \} \\
    & = \{\sigma \in \mathbb{C} : \sigma \mbox{ is a Floquet exponent} \}.    
\end{split}
\end{align}
\end{lemma}
\begin{proof}
Clearly, $\Delta(z) = D + B(z)$, where $D : \mathcal{D}(D) \subseteq C_T(\mathbb{R},\mathbb{C}^n) \to C_T(\mathbb{R},\mathbb{C}^n)$ with domain $\mathcal{D}(D) \coloneqq C_T^1(\mathbb{R},\mathbb{C}^n)$ is defined by $Dq \coloneqq \dot{q}$, and $B(z) \in \mathcal{L}( C_T(\mathbb{R},\mathbb{C}^n))$ is defined by $(B(z)q)(\tau) \coloneqq zq(\tau) - L(\tau)[\theta \mapsto e^{-z \theta}q(\tau-\theta)]$. Clearly, $D$ is closed and the fact that $B(z)$ is bounded follows from $\|B(z)q\|_{\infty} \leq ( |z| + \mbox{TV}_\zeta \max\{1,e^{-h \Re (z)} \}) \| q \|_{\infty}$. This inequality follows from a result on Riemann-Stieltjes integrals, see for example \cite[page 14]{Diekmann1995}. Here, $\mbox{TV}_\zeta \coloneqq \sup_{t \in \mathbb{R}} \mbox{TV}(\zeta(t,\cdot))$, where $\mbox{TV}(\zeta(t,\cdot))$ denotes the total variation of $\zeta(t,\cdot)$. Note that $\mbox{TV}_\zeta$ is finite since $\theta \mapsto \zeta(\cdot,\theta)$ is bounded, and $t \mapsto \zeta(t,\cdot)$ is continuous and $T$-periodic. As $\Delta(z)$ is the sum of the closed operator $D$ and bounded linear operator $B(z)$, it follows from \cite[Lemma 2.4]{Engel2000} that $\Delta(z)$ is closed.

Let us now prove the second claim. The resolvent of $D$ at $\mu$ is given by
\begin{equation*}
    ((\mu I - D)^{-1}q)(\tau) = \frac{e^{\mu \tau}}{1-e^{-\mu T}} \bigg( \int_\tau^T e^{-\mu s}q(s) ds + \int_0^\tau e^{-\mu (s+T)}q(s) ds\bigg), \quad \forall q \in C_T(\mathbb{R},\mathbb{C}^n),
\end{equation*}
whenever $\mu \in \rho(D) = \mathbb{C} \setminus \hspace{-2pt} \frac{2 \pi i}{T} \mathbb{Z}$, and thus $\|(\mu I - D)^{-1}\| \leq 1/|\mu| \to 0$ as $\mu \to \infty$ for $\mu \in \mathbb{R}$. Hence, for every fixed $z \in \mathbb{C}$, there exists a $\mu_z \in \rho(D)$ such that $\|(\mu_z I - D)^{-1}\| < \|B(z)\|^{-1}$. Thus $I - B(z)(\mu_z I -D)^{-1}$ is an invertible operator in $\mathcal{L}(C_T(\mathbb{R},\mathbb{C}^n))$ by the Neumann series. Since $\mu_z \in \rho(D)$, $[I - B(z)(\mu_z I -D)^{-1}](\mu_z I - D) = \mu_z I - \Delta(z)$ is invertible and thus $\mu_z \in \rho(\Delta(z))$, i.e. $\rho(\Delta(z)) \neq \emptyset$.

Due to the Arzel\'a-Ascoli theorem, the canonical injection $(C_T^1(\mathbb{R},\mathbb{C}^n),\|\cdot\|_{C^1}) \hookrightarrow C_T(\mathbb{R},\mathbb{C}^n)$ is compact, and so $D$ has compact resolvent by \cite[Proposition II.4.25]{Engel2000}. Since $\rho(\Delta(z)) \neq \emptyset$ and $B(z) \in \mathcal{L}(C_T(\mathbb{R},\mathbb{C}^n))$, it follows from \cite[Exercise II.4.30]{Engel2000} that $\Delta(z) = D + B(z)$ has compact resolvent. Let $\sigma \in \mathbb{C}$ such that $\Delta(\sigma)$ is not invertible, then $0 \in \sigma(\Delta(\sigma)) = \sigma_p(\Delta(\sigma))$ due to \cite[Corollary IV.1.19]{Engel2000}, which proves that $\Delta(\sigma)$ is not injective meaning that $\sigma$ is a characteristic value of $\Delta$. But recall from the discussion above \eqref{eq:motivationcharac} that characteristic values of $\Delta$ coincide with Floquet exponents.
\end{proof}

If $w \in C_T(\mathbb{R},\mathbb{C}^n)$ is given, the equation $\Delta(z)u = w$ is a periodic linear (inhomogeneous) DDE where $u \in C_T^1(\mathbb{R},\mathbb{C}^n)$ is the unknown. If $u$ is a unique solution of this equation, we denote $u$ by $\Delta(z)^{-1} w$. 
To derive the explicit formulas for the (generalized) eigenfunctions, we proceed from \cite[Section 3.1]{Article2}. Therefore, let us introduce the linear operator $\mathcal{A} : \mathcal{D}(\mathcal{A}) \subseteq C_T(\mathbb{R},X) \to C_T(\mathbb{R},X)$ by
\begin{equation*}
    \mathcal{D}(\mathcal{A}) \coloneqq \{ \varphi \in C_T^1(\mathbb{R},X) : \varphi(\tau) \in \mathcal{D}(A(\tau)) \mbox{ for all $\tau \in \mathbb{R}$} \}, \quad 
    (\mathcal{A}\varphi)(\tau) \coloneqq A(\tau)\varphi(\tau) - \dot{\varphi}(\tau),
\end{equation*}
and recall from \cite[Proposition 7]{Article1} that $\mathcal{A}$ is not closed but closable. Due to the density of $C_T^1(\mathbb{R},X)$ in $C_T(\mathbb{R},X)$ and $\mathcal{D}(A(\tau))$ in $X$, \cite[Lemma 17]{Article1} tells us that $\mathcal{A}$ is densely defined.

Let $\lambda$ be a Floquet multiplier. Notice that these operators can also be introduced on a space of $T$-antiperiodic functions. For the sake of simplicity, we will work from now on with $T$-periodic maps ($\lambda \in \mathbb{C} \setminus \mathbb{R}_{-}$) and remark at the end of this section on the $T$-antiperiodic setting ($\lambda \in \mathbb{R}_{-}$), see \cref{remark:Tantiperiodic}. The following result provides us the explicit computational formulas for the (generalized) eigenfunctions when $\lambda \in \mathbb{C} \setminus \mathbb{R}_{-}$. The autonomous analogue of the following result can be found in \cite[Theorem IV.5.5]{Diekmann1995}.

\begin{theorem} \label{thm:eigenfunctions}
Let $\lambda \in \mathbb{C} \setminus \mathbb{R}_{-}$ be a Floquet multiplier of algebraic multiplicity $m_\lambda$ with $\sigma$ its associated Floquet exponent. Then there exist $\varphi_i \in C_T^{k+1}(\mathbb{R},X)$ satisfying
\begin{equation} \label{eq:ODEeig2}
    ( \mathcal{A} - \sigma I)\varphi_i=
    \begin{cases}
    0, \quad &i = 0,\\
    \varphi_{i-1}, \quad &i=1,\dots,m_\lambda-1,
    \end{cases}
\end{equation}
such that the set of functions $\{ \varphi_0(\tau),\dots,\varphi_{m_\lambda - 1}(\tau) \}$ is an ordered basis of $E_\lambda(\tau)$.  If the functions $q_i \in C_T^{k+1}(\mathbb{R},\mathbb{C}^n)$ satisfy the periodic linear DDEs
\begin{equation} \label{eq:q_i}
    \sum_{l=0}^i \frac{1}{l!}\Delta^{(l)}(\sigma)  q_{i-l} = 0,
\end{equation}
then the (generalized) eigenfunctions $\varphi_i$ are given by
\begin{equation} \label{eq:varphi_i}
    \varphi_i(\tau)(\theta) = e^{\sigma \theta} \sum_{l=0}^{i} \frac{\theta^l}{l!} q_{i-l}(\tau + \theta).
\end{equation}
\end{theorem}

\begin{proof}
The claim on the existence and $C^{k+1}$-smoothness of the $T$-periodic (generalized) eigenfunctions $\varphi_i$ satisfying \eqref{eq:ODEeig2} is proven in \cite[Theorem 5]{Article2}. Let us now prove the explicit formulas \eqref{eq:q_i} and \eqref{eq:varphi_i}. Recall from the action of $\mathcal{A}$ that $\varphi_i$ must satisfy
\begin{equation*}
    \bigg(\frac{\partial}{\partial \theta} - \frac{\partial}{\partial \tau} - \sigma \bigg) \varphi_{i}(\tau)(\theta) = 
    \begin{dcases}
    0, \quad &i = 0, \\
    \varphi_{i-1}(\tau)(\theta), \quad &i=1,\dots,m_\lambda-1.
    \end{dcases}
\end{equation*}
This is a first-order linear (inhomogeneous) PDE that has the solution
\begin{equation} \label{eq:varphi_i_PDE}
    \varphi_i(\tau)(\theta) = e^{\sigma \theta}q_{i}(\tau+\theta) + \int_0^{\theta}e^{\sigma(\theta- s)} \varphi_{i-1}(\tau + \theta -s)(s) ds, \quad \forall \tau \in \mathbb{R}, \ \theta \in [-h,0],
\end{equation}
whenever an initial condition $q_i \in C_T^{k+1}(\mathbb{R},\mathbb{C}^n)$ for $i=0,\dots,m_\lambda-1$ is specified. Here we set $\varphi_{-1}$ to be the zero function for simplicity. Using the obtained recurrence relation, it is then straightforward to check that $\varphi_i$ is given as in \eqref{eq:varphi_i}. To find the functions $q_i$, recall that $\varphi_i(\tau) \in \mathcal{D}(A(\tau))$ for all $\tau \in \mathbb{R}$. The domain statement from \eqref{eq:D(As)} shows us that $\varphi_i$ must satisfy
\begin{equation} \label{eq:varphi_i_PDE_DF}
   \frac{\partial}{\partial \theta}\varphi_{i}(\tau)(\theta) \bigg |_{\theta = 0} = L(\tau)\varphi_{i}(\tau) = \int_0^h d_2\zeta(\tau,\theta)\varphi_{i}(\tau)(-\theta), \quad \forall \tau \in \mathbb{R}.
\end{equation}
The recurrence relation in \eqref{eq:varphi_i_PDE} show that $\frac{\partial}{\partial \theta}\varphi_{i}(\tau)(\theta) |_{\theta = 0} = \sigma q_i(\tau) + \dot{q}_i(\tau) + q_{i-1}(\tau)$, for $i=0,\dots,m_\lambda-1$ when we define $q_{-1}$ to be the zero function. Hence, \eqref{eq:varphi_i_PDE_DF} implies that $\dot{q}_{i}(\tau) = -\sigma q_i(\tau) - q_{i-1}(\tau) + \int_0^h d_2 \zeta(\tau,\theta) \varphi_{i}(\tau)(-\theta)$. Using \eqref{eq:varphi_i} again, the result as stated in \eqref{eq:q_i} follows.
\end{proof}

Our next goal is to obtain an explicit formula for the $T$-periodic adjoint (generalized) eigenfunctions in terms of the characteristic operator. In the autonomous setting, recall from \cite{Diekmann1995,Bosschaert2020,Janssens2010} that the adjoint (transpose) of the characteristic matrix $\hat{\Delta}(z) \in \mathbb{C}^{n \times n}$ is needed to derive the adjoint eigenfunctions. Moreover, the adjoint $\hat{\Delta}^\star(z) \coloneqq [\hat{\Delta}(z)]^\star$ is the unique linear map satisfying $\langle \hat{\Delta}^\star(z)p,q \rangle = \langle p, \hat{\Delta}(z)q \rangle$ for all $(p,q) \in \mathbb{C}^{n \star} \times \mathbb{C}^{n}$ due to the nondegeneracy of $\langle \cdot , \cdot \rangle$. Using this analogy, it is clear for our periodic setting that we need a unique adjoint of $\Delta(z)$ in terms of a specific nondegenerate bilinear map. To do this, let us introduce the \emph{pairing} $\langle \cdot, \cdot \rangle_T : C_T(\mathbb{R},\mathbb{C}^{n \star}) \times C_T(\mathbb{R},\mathbb{C}^n) \to \mathbb{C}$ by
\begin{equation} \label{eq: pairingTCn}
    \langle p,q \rangle_T \coloneqq \int_0^T  p(\tau)q(\tau) d \tau, \quad \forall (p,q) \in C_T(\mathbb{R},\mathbb{C}^{n \star}) \times C_T(\mathbb{R},\mathbb{C}^n),
\end{equation}
and note that $\langle \cdot, \cdot \rangle_T$ is not the natural dual pairing between $C_T(\mathbb{R},\mathbb{C}^{n \star})$ and $C_T(\mathbb{R},\mathbb{C}^n)$ as simply $C_T(\mathbb{R},\mathbb{C}^n)^\star \neq C_T(\mathbb{R},\mathbb{C}^{n \star})$. However, we can still consider the \emph{dual pair} $(C_T(\mathbb{R},\mathbb{C}^n),C_T(\mathbb{R},\mathbb{C}^{n \star}))$ with respect to the bilinear map $\langle \cdot , \cdot \rangle_T$, see \cite[Chapter 8]{Narici2010} for more information on duality theory. To guarantee uniqueness of the adjoint of $\Delta(z)$ with respect to $\langle \cdot , \cdot \rangle_T$, we need the following.

\begin{lemma} \label{lemma:nondegenerate}
The bilinear map $\langle \cdot, \cdot \rangle_T$ from \eqref{eq: pairingTCn} is nondegenerate.
\end{lemma}
\begin{proof}
We show that the linear maps $q \mapsto \langle \cdot,q \rangle_T$ and $p \mapsto \langle p, \cdot \rangle_T$ are injective. Suppose that $\langle \cdot,q \rangle_T = 0$ for some $q \in C_T(\mathbb{R},\mathbb{C}^n)$. Let $p_0 \in C_T(\mathbb{R},\mathbb{C}^{n \star})$ be such that $p_0(\tau)$ is the Hermitian adjoint of $q(\tau)$ for all $\tau \in \mathbb{R}$. Then $\langle p_0,q \rangle_T = \int_0^T |q(\tau)|^2 d\tau = 0$ and so $q=0$ by continuity. The proof on the injectivity of $p \mapsto \langle p, \cdot \rangle_T$ is similar and therefore omitted.
\end{proof}

The following result shows that we can define a unique adjoint $\Delta^\star(z) \coloneqq \Delta(z)^\star$ of $\Delta(z)$ in terms of the pairing \eqref{eq: pairingTCn}. This adjoint inherits the properties of $\Delta(z)$ from \cref{lemma:Delta(z)closed}.

\begin{lemma} \label{lemma:adjoint charac}
For any $z \in \mathbb{{C}}$, the operator $\Delta^\star(z) : \mathcal{D}(\Delta^\star(z)) \coloneqq C_T^1(\mathbb{R},\mathbb{C}^{n \star}) \subseteq C_T(\mathbb{R},\mathbb{C}^{n \star}) \to C_T(\mathbb{{R}},\mathbb{C}^{n \star})$ defined by
\begin{equation} \label{eq:adjoint charac operator}
    (\Delta^\star(z)p)(\tau) \coloneqq -\dot{p}(\tau) + zp(\tau) - \int_0^h   p(\tau + \theta) e^{-z \theta} d_2 \zeta(\tau + \theta,\theta), \quad \forall p \in C_T^1(\mathbb{R},\mathbb{C}^{n \star}), 
\end{equation}
is the unique linear operator satisfying
\begin{equation} \label{eq:pairing charac}
    \langle \Delta^\star(z)p,q \rangle_T = \langle p,\Delta(z)q \rangle_T, \quad \forall (p,q) \in C_T^1(\mathbb{R},\mathbb{C}^{n \star}) \times C_T^1(\mathbb{R},\mathbb{C}^n).
\end{equation}
Moreover, $\Delta^\star(z)$ is closed and $\rho(\Delta^\star(z)) \neq \emptyset$ for all $z \in \mathbb{C}$, and there holds
\begin{equation*} 
    \{ \sigma \in \mathbb{C} : \Delta^\star(\sigma) \mbox{ is not invertible} \} = \{ \sigma \in \mathbb{C} : \sigma \mbox{ is a characteristic value of } \Delta^\star \}. 
\end{equation*}
\end{lemma}
\begin{proof}
Let $z,p$ and $q$ be given as mentioned. It follows immediately that
\begin{equation*}
    \langle p, \Delta(z)q \rangle_T = \int_0^T  p(\tau) \dot{q}(\tau) d\tau +  \int_0^T  zp(\tau)q(\tau) d\tau - \int_0^T p(\tau) \int_0^h d_2 \zeta(\tau,\theta) e^{-z\theta}q(\tau-\theta) d\tau.
\end{equation*}
Applying integration by parts on the first term of the right-hand side yields $-\langle \dot{p},q \rangle_T$, due to the $T$-periodicity of $p$ and $q$. The second term equals $\langle zp,q \rangle_T$, while the third term requires a bit more work. Let $P_m = \{0 = \theta_0 < \theta_1 < \dots < \theta_m = h \}$ be a partition of $[0,h]$ and choose an arbitrary $c_i \in [\theta_{i},\theta_{i+1}]$ for all $i = 0,\dots,m-1$. Consider  $\int_0^T  p(\tau) \sum_{i=0}^{m-1} [ \zeta(\tau,\theta_{i+1}) - \zeta(\tau,\theta_i)] e^{-z c_i}q(\tau - c_i)  d\tau$. If we substitute $ \tau - c_i \to \tau$, the integral bounds become $-c_i$ and $T - c_i$. Because $p,q$ and $\tau \mapsto \zeta(\tau,\cdot)$ are $T$-periodic, the integrand is $T$-periodic, and so we can shift the bounds of integration back to $0$ and $T$. Hence,
\begin{align*}
&\int_0^T p(\tau) \sum_{i=0}^{m-1} [ \zeta(\tau,\theta_{i+1}) - \zeta(\tau,\theta_i)] e^{-z c_i}q(\tau - c_i)  d\tau \\
&= \int_0^T \sum_{i=0}^{m-1} p(\tau+c_i) e^{-z c_i}[ \zeta(\tau + c_i,\theta_{i+1}) - \zeta(\tau + c_i,\theta_i)] q(\tau) d\tau.  
\end{align*}
If we let $m \to \infty$ while $\max_{i=0,\dots,m-1} |\theta_{i+1} - \theta_i| \to 0$, then the summation on the left-hand side approaches $\int_0^h d_2 \zeta(\tau,\theta) e^{-z\theta}q(\tau-\theta)$ by definition of the Riemann-Stieltjes integral. Hence,
\begin{equation*}
\int_0^T p(\tau) \int_0^h d_2 \zeta(\tau,\theta) e^{-z\theta}q(\tau-\theta) d\tau = \int_0^T \int_0^h p(\tau + \theta) e^{-z\theta}  d_2\zeta(\tau+\theta,\theta)q(\tau) d\tau,
\end{equation*}
and so $\langle p,\Delta(z)q \rangle_T = \langle \Delta^\star(z)p, q \rangle_T$, which shows that $\Delta^\star(z)$ from \eqref{eq:adjoint charac operator} satisfies \eqref{eq:pairing charac}. 

It remains to show that $\Delta^{\star}(z)$ is the unique linear operator satisfying \eqref{eq:pairing charac}. Let $Q(z) : C_T^1(\mathbb{R},\mathbb{C}^{n \star}) \to C_T(\mathbb{R},\mathbb{C}^{n \star})$ be a linear operator satisfying $\langle (Q(z) - \Delta^\star(z))p,q \rangle_T = 0$, for all $z \in \mathbb{C}$ and $ (p,q) \in C_T^1(\mathbb{R},\mathbb{C}^{n \star}) \times  C_T^1(\mathbb{R},\mathbb{C}^n)$. Let $q \in C_T(\mathbb{R},\mathbb{C}^n)$ be given and let $(q_m)_m$ be a sequence in $C_T^1(\mathbb{R},\mathbb{C}^n)$ such that $q_m \to q$ in norm as $m \to \infty$. As $\langle (Q(z) - \Delta^\star(z))p,q_m \rangle_T = 0$ for all $m \in \mathbb{N}$, we obtain $|\langle (Q(z) - \Delta^\star(z))p,q \rangle_T| \leq T \| (Q(z) - \Delta^\star(z))p \|_\infty \| q_m - q \|_\infty \to 0$ as $m \to \infty$ and thus $\langle (Q(z) - \Delta^\star(z))p,q \rangle_T = 0$ for all $q \in C_T(\mathbb{R},\mathbb{C}^n)$. But due to \cref{lemma:nondegenerate}, the map $\langle \cdot , \cdot \rangle_T$ is nondegenerate and so $Q(z) = \Delta^\star(z)$ on $C_T^1(\mathbb{R},\mathbb{C}^{n \star})$, i.e. $\Delta^\star(z)$ is uniquely determined. The remaining properties of $\Delta^\star(z)$ can be proven as performed in \cref{lemma:Delta(z)closed}.
\end{proof}

If $w \in C_T(\mathbb{R},\mathbb{C}^{n \star})$ is given, the equation $\Delta^\star(z)u = w$ is a periodic linear (inhomogeneous) advance differential equation (ADEs) where $u \in C_T^1(\mathbb{R},\mathbb{C}^n)$ is the unknown. If $u$ is a unique solution of this equation, we denote $u$ by $\Delta^\star(z)^{-1} w$. The reason \eqref{eq:adjoint charac operator} induces differential equations with advanced arguments is due to the $(\tau+\theta)$-term inside the kernel $\zeta$ from \eqref{eq:adjoint charac operator}. For any integer $l \geq 0$, the $l$th order derivative of the map $z \mapsto \Delta^\star(z)$ will be denoted by $z \mapsto \Delta^{\star (l)} (z)$, and we use the same notation for $\Delta(z)$. If $l \geq 1$, it follows from \eqref{eq:Delta(z)q} that $\Delta^{(l)}(z)$ and $\Delta^{ \star (l)}(z)$ are bounded for all $z \in \mathbb{C}$ since the differential operator vanishes after differentiation with respect to $z$. Employing \eqref{eq: pairingTCn} and \eqref{eq:pairing charac}, it is clear that $\Delta^{\star (l)}(z)$ is the unique linear operator satisfying
\begin{equation} \label{eq:pairing charac deri}
    \langle \Delta^{\star (l)}(z)p,q \rangle_T = \langle p,\Delta^{(l)}(z)q \rangle_T, \quad \forall l \geq 1, \ z \in \mathbb{C}, \ (p,q) \in C_T(\mathbb{R},\mathbb{C}^{n \star}) \times  C_T(\mathbb{R},\mathbb{C}^n),
\end{equation}
since the pairing $\langle \cdot, \cdot \rangle_T$ is nondegenerate. To derive the explicit formulas for the adjoint (generalized) eigenfunctions, we proceed from \cite[Section 3.2]{Article2}. Therefore, let us introduce the linear operator $\mathcal{A}^\star : \mathcal{D}(\mathcal{A}^\star) \subseteq C_T(\mathbb{R},X^\star)  \to C_T(\mathbb{R},X^\star)$ by
\begin{align*}
    \mathcal{D}(\mathcal{A}^\star) &\coloneqq \{ \varphi^\star \in C_T^1(\mathbb{R},X^\star) : \varphi^\star(\tau) \in \mathcal{D}(A^\star(\tau)) \mbox{ for all $\tau \in \mathbb{R}$} \}, \\
    (\mathcal{A}^\star\varphi^\star)(\tau) &\coloneqq A^\star(\tau)\varphi^\star(\tau) + \dot{\varphi}^\star(\tau).
\end{align*}
and recall from \cite[Proposition 14]{Article1} that $\mathcal{A}^\star$ is not closed but closable. Since $\mathcal{D}(A^\star(\tau))$ is not norm-dense in $X^\star$ (\cref{sec:periodic center manifolds}), we can not expect, as for the operator $\mathcal{A}$, that $\mathcal{A}^\star$ is densely defined. The following result provides us the explicit computational formulas for the adjoint (generalized) eigenfunctions when a Floquet multiplier $\lambda \in \mathbb{C} \setminus \mathbb{R}_{-}$. Its autonomous analogue can be found in \cite[Theorem IV.5.9]{Diekmann1995}.

\begin{theorem} \label{thm:adjoint eigenfunctions2}
Let $\lambda \in \mathbb{C} \setminus \mathbb{R}_{-}$ be a Floquet multiplier of algebraic multiplicity $m_\lambda$ with $\sigma$ its associated Floquet exponent. Then there exist $\varphi_{i}^\odot \in C_T^{k+1}(\mathbb{R},X^\odot)$ satisfying
\begin{equation} \label{eq:ODE adjoint}
    (\mathcal{A}^\star - \sigma I)\varphi_i^\odot=
    \begin{cases}
    0, \quad &i = m_\lambda-1,\\
    \varphi_{i+1}^\odot, \quad &i=m_\lambda-2,\dots,0.
    \end{cases}
\end{equation}
such that the set of functions $\{\varphi_{m_\lambda - 1}^\odot(\tau),\dots,\varphi_{0}^\odot(\tau) \}$ is an ordered basis of $E_\lambda^\odot(\tau)$. If the functions  $p_i \in C_T^{k + 1}(\mathbb{R},\mathbb{C}^{n \star})$ satisfy the periodic linear ADEs
\begin{equation} \label{eq:p_i}
\sum_{l=0}^{m_\lambda - 1 - i} \frac{1}{l!}\Delta^{\star (l)} (\sigma)  p_{i+l} = 0,
\end{equation}
then the adjoint (generalized) eigenfunctions $\varphi_i^\odot$ are given by
\begin{equation} \label{eq:varphi_i^sun}
    \varphi_i^\odot(\tau) = \bigg( p_{i}(\tau), \theta \mapsto \sum_{l = 0}^{m_\lambda - 1 - i} \int_\theta^h e^{\sigma(\theta - s)} \frac{(\theta -s)^l}{l!} p_{i+l}(\tau+s-\theta) d_2\zeta(\tau+s-\theta,s) \bigg).
\end{equation}
\end{theorem}
\begin{proof}
The claim on the existence and $C^{k+1}$-smoothness of the $T$-periodic adjoint (generalized) eigenfunctions $\varphi_i^\odot$ satisfying \eqref{eq:ODE adjoint} is proven in \cite[Theorem 8]{Article1}. If we manage to prove that $\varphi_i \in \mathcal{D}(\mathcal{A}^\star)$, then $\varphi_i^\odot$ takes values in $X^\odot$ as $\mathcal{D}(A^\star(\tau)) \subseteq X^\odot$, see \cref{sec:periodic center manifolds}. Recall from \eqref{eq:D(Astars)} that $\varphi_i^{\odot}$ satisfies
\begin{equation} \label{eq:PDEadjoint}
    \bigg( \frac{\partial}{\partial \tau} + \frac{\partial}{\partial \theta} - \sigma \bigg) \varphi_i^{\odot}(\tau)(\theta) = 
    \begin{dcases}
    -p_{m_\lambda - 1}(\tau)\zeta(\tau,\theta), \quad &i = m_\lambda - 1, \\
    -p_{i}(\tau)\zeta(\tau,\theta) + \varphi_{i+1}^{\odot}(\tau)(\theta), \quad &i=m_\lambda - 2, \dots,0,
    \end{dcases}
\end{equation}
for all $\tau \in \mathbb{R}$ and $\theta \in (0,h]$, since at $\theta = 0$ a discontinuous jump is allowed when we define $\varphi_i^\odot(\tau)(0) \coloneqq 0$. In both cases, this is a first-order linear inhomogeneous PDE that has the solution
\begin{equation*}
    \varphi_i^\odot(\tau)(\theta) =
    e^{\sigma \theta}p_{i}(\tau-\theta) - \int_0^{\theta}e^{\sigma(\theta-s)}[p_i(\tau + s-\theta)\zeta(\tau+s-\theta,s) - \varphi_{i+1}^\odot(\tau+s-\theta)(s) ] ds,
\end{equation*}
for all $\tau \in \mathbb{R}$ and $\theta \in (0,h]$, whenever an initial condition $p_i \in C_T^{k+1}(\mathbb{R},\mathbb{C}^{n \star})$ for $i=m_\lambda-1,\dots,0$ is specified. Here, we set $\varphi_{m_\lambda}^\odot$ to be the zero function. The next goal is to determine $p_i$ and the derivative $g_i$ of $\varphi_i^\odot$ by recalling the fact that $\varphi_i^\odot(\tau) \in X^{\odot}$ for all $\tau \in \mathbb{R}$. Hence, we have to show that for every $\tau \in \mathbb{R}$ there exists a map $g_i(\tau) \in L^1([0,h], \mathbb{C}^{n \star})$ satisfying $g_i(\cdot)(h) = 0$ such that
\begin{equation*}
\varphi_i^{\odot}(\tau)(\theta) = p_i(\tau) + \int_0^\theta g_i(\tau)(s) ds, \quad \forall \tau \in \mathbb{R}, \ \theta \in (0,h].
\end{equation*}
A candidate for $g_i(\tau)(\theta)$ is $\frac{\partial}{\partial \theta}\varphi_i(\tau)(\theta)$ in a weak sense. The Leibniz integral rule shows us that
\begin{align*}
    g_i(\tau)(\theta) &= \sigma e^{\sigma \theta} - e^{\sigma \theta}\dot{p}_i(\tau - \theta)-p_i(\tau) \zeta(\tau,\theta) + \varphi_{i+1}^\odot(\tau)(\theta) \\
    &-\sigma\int_0^\theta e^{\sigma(\theta-s)}[p_i(\tau+s-\theta) \zeta(\tau+s-\theta,s) -\varphi_{i+1}^\odot(\tau+s-\theta)(s)]ds \\
    &+\int_0^\theta e^{\sigma(\theta-s)}[\dot{p}_i(\tau+s-\theta)\zeta(\tau+s-\theta,s) \\
    &+ p_i(\tau+s-\theta) D_1 \zeta(\tau+s-\theta,s) - D_1 \varphi_{i+1}^\odot(\tau+s-\theta)(s)] ds,
\end{align*}
If we use the formulas
\begin{align*}
    &\int_0^\theta e^{\sigma(\theta-s)}\dot{p}_i(\tau+s-\theta) \zeta(\tau+s-\theta,s) ds \\
    &= p_i(\tau)\zeta(\tau,\theta) -\int_0^\theta e^{\sigma(\theta-s)}p_i(\tau+s-\theta)[-\sigma \zeta(\tau+s-\theta,s) + D_1\zeta(\tau+s-\theta,s)] ds \\
    &- \int_0^\theta e^{\sigma(\theta-s)}p_i(\tau+s-\theta)d_2\zeta(\tau+s-\theta,s)
\end{align*}
and
\begin{align*}
    \int_0^\theta e^{\sigma(\theta-s)}\varphi_{i+1}^\odot(\tau+s-\theta)(s) ds &= - \frac{1}{\sigma}\bigg(\varphi_{i+1}^\odot(\tau)(\theta) - p_{i+1}(\tau-\theta)e^{-\sigma \theta}\\
    &- \int_0^\theta e^{\sigma(\theta-s)}[D_1 \varphi_{i+1}^\odot(\tau+s-\theta)(s) + g_{i+1}(\tau+s-\theta)(s)]ds\bigg)
\end{align*}
obtained via integration by parts (for Riemann-Stieltjes integrals), we get
\begin{align} \label{eq:g_i}
    g_i(\tau)(\theta) &= \sigma e^{\sigma  \theta}p_i(\tau-\theta) - e^{\sigma  \theta}\dot{p}_i(\tau - \theta) + e^{\sigma \theta}p_{i+1}(\tau-\theta) + \int_0^\theta e^{\sigma (\theta-s)}g_{i+1}(\tau+s-\theta)(s) ds \nonumber \\
    &-\int_0^\theta e^{\sigma(\theta-s)} p_i(\tau+s-\theta) d_2\zeta(\tau+s-\theta,s).
\end{align}
Recall that $g_i$ must satisfy in particular $g_i(\tau - \theta +h)(h) = 0$ and so we get
\begin{align} 
\begin{split} \label{eq:pidot}
    \dot{p}_i(\tau-\theta) &= \sigma p_i(\tau-\theta) +p_{i+1}(\tau-\theta) + \int_0^h e^{-\sigma s}g_{i+1}(\tau+s-\theta)(s) ds \\
    &-\int_0^h e^{-\sigma s}p_{i}(\tau+s-\theta) d_2 \zeta(\tau+s-\theta,s).    
\end{split}
\end{align}
Plugging this result back into \eqref{eq:g_i}, we obtain
\begin{equation*}
    g_i(\tau)(\theta) = \int_\theta^h e^{\sigma(\theta-s)} p_i(\tau+s-\theta) d_2 \zeta(\tau+s-\theta,s) - \int_\theta^h e^{\sigma(\theta-s)}g_{i+1}(\tau+s-\theta)(s) ds.
\end{equation*}
Using this recurrence relation in combination with Fubini's theorem we get
\begin{equation*}
    g_i(\tau)(\theta) = \sum_{l = 0}^{m_\lambda - 1 - i} \int_\theta^h e^{\sigma(\theta - s)} \frac{(\theta -s)^l}{l!} p_{i+l}(\tau+s-\theta) d_2\zeta(\tau+s-\theta,s),
\end{equation*}
which yields \eqref{eq:varphi_i^sun} and it is clear that $g_i(\tau) \in L^1([0,h], \mathbb{C}^{n \star})$. If we put $\theta = 0$ in \eqref{eq:pidot}, we obtain \eqref{eq:p_i} by using the recurrence relation derived above. Because $p_i$ is at least continuous, we have that the integrand of $g_i(\tau)$ has bounded variation on $[0,h]$. If we define $g_i(\cdot)(0) = 0$, it is clear that $g(\tau) \in \NBV([0,h],\mathbb{C}^{n\star})$ and so $\varphi_i^\odot(\tau) \in \mathcal{D}(A^\star(\tau))$ for all $\tau \in \mathbb{R}$.
\end{proof}

Notice that \cref{thm:adjoint eigenfunctions2} also holds for the linear operator $\mathcal{A}^\odot : \mathcal{D}(\mathcal{A}^\odot) \subseteq C_T(\mathbb{R},X^{\odot}) \to C_T(\mathbb{R},X^\odot)$ defined by
\begin{equation*}
    \mathcal{D}(\mathcal{A}^{\odot}) \coloneqq \{ \varphi^\odot \in C_T^1(\mathbb{R},X^{\odot}) : \varphi^\odot(\tau) \in \mathcal{D}(A^{\odot}(\tau)) \mbox{ for all $\tau \in \mathbb{R}$} \}, \quad \mathcal{A}^{\odot} \varphi^\odot \coloneqq \mathcal{A}^{\star} \varphi^\odot, 
\end{equation*}
since one can verify \cite[Theorem 13]{Article1} that the adjoint (generalized) eigenfunctions $\varphi_{m_\lambda - 1}^\odot,\dots,\varphi_0^\odot$ are in $\mathcal{D}(\mathcal{A}^\odot)$. Moreover, recall from \cite[Theorem 13]{Article1} that $\mathcal{A}^\odot$ is not closed but closable. However, as $C_T^1(\mathbb{R},X^{\odot})$ is norm dense in $C_T(\mathbb{R},X^{\odot})$ and $\mathcal{D}(A^\odot(\tau))$ is norm dense in $X^\odot$, it follows from \cite[Lemma 17]{Article1} that $\mathcal{A}^\odot$ is densely defined. To go full circle, let us introduce in addition the linear operator $\mathcal{A}^{\odot \star} : \mathcal{D}(\mathcal{A}^{\odot \star}) \subseteq C_T(\mathbb{R},X^{\odot \star}) \to C_T(\mathbb{R},X^{\odot \star})$ by
\begin{align*}
    \mathcal{D}(\mathcal{A}^{\odot \star}) &\coloneqq \{ \varphi^{\odot \star} \in C_T^1(\mathbb{R},X^{\odot \star}) : \varphi^{\odot \star}(\tau) \in \mathcal{D}(A^{\odot \star}(\tau)) \mbox{ for all $\tau \in \mathbb{R}$} \}, \\
    (\mathcal{A}^{\odot \star}\varphi^{\odot \star})(\tau) &\coloneqq A^{\odot \star}(\tau)\varphi^{\odot \star}(\tau) - \dot{\varphi}^{\odot \star}(\tau),
\end{align*}
and it follows from \cite[Proposition 6]{Article2} that $\mathcal{A}^{\odot \star}$ is not closed but closable. Again, as $\mathcal{D}(A^{\odot \star}(\tau))$ is not norm-dense in $X^{\odot \star}$ (\cref{sec:periodic center manifolds}), we can not expect that $\mathcal{A}^{\odot \star}$ is densely defined. To relate the Jordan chains of $\mathcal{A}$ from \cref{thm:eigenfunctions} with those of $\mathcal{A}^{\odot \star}$, let us first introduce the linear (canonical) embedding $\iota : C_T(\mathbb{R},X) \to C_T(\mathbb{R},X^{\odot \star})$ by $(\iota \varphi)(\tau) \coloneqq j \varphi(\tau)$. Note that $\iota \varphi$ takes values in $X^{\odot \odot}$ due to the $\odot$-reflexivity of $X$ with respect to $U$. Then, according to \cite[Theorem 5]{Article1}, there holds
\begin{equation} \label{eq:ODEeig}
    (\mathcal{A}^{\odot \star} - \sigma I)\iota\varphi_i=
    \begin{cases}
    0, \quad &i = 0,\\
    \iota\varphi_{i-1}, \quad &i=1,\dots,m_\lambda-1.
    \end{cases}
\end{equation}
Recall from \eqref{eq:pairing charac} that $\Delta^\star(z)$ is the unique \emph{adjoint} of $\Delta(z)$ with respect to the pairing $\langle \cdot, \cdot \rangle_T$ from \eqref{eq: pairingTCn}. We would like to prove a similar statement for on the one hand $\mathcal{A}$ and $\mathcal{A}^\star$, and on the other hand $\mathcal{A}^\odot$ and $\mathcal{A}^{\odot \star}$. To do this, let us introduce the (complexified) bilinear maps $\langle \cdot, \cdot \rangle_T : C_T(\mathbb{R},X^\star) \times C_T(\mathbb{R},X) \to \mathbb{C}$ and $\langle \cdot, \cdot \rangle_T : C_T(\mathbb{R},X^{ \odot \star}) \times C_T(\mathbb{R},X^\odot) \to \mathbb{C}$, also called \emph{pairings}, by
\begin{align} \label{eq:pairingT}
\begin{split}
    \langle \varphi^\star, \varphi \rangle_T &\coloneqq \int_0^T \langle \varphi^\star(\tau), \varphi(\tau) \rangle d\tau, \quad \forall (\varphi^\star,\varphi) \in  C_T(\mathbb{R},X^\star) \times  C_T(\mathbb{R},X), \\
    \langle \varphi^{\odot\star}, \varphi^\odot \rangle_T &\coloneqq \int_0^T \langle \varphi^{\odot \star}(\tau), \varphi^{\odot}(\tau) \rangle d\tau, \quad \forall (\varphi^{\odot\star},\varphi^\odot) \in  C_T(\mathbb{R},X^{\odot \star}) \times  C_T(\mathbb{R},X^\odot),
\end{split}
\end{align}
where the natural dual pairing $\langle \cdot, \cdot \rangle$ is between $X^\star$ and $X$ in the first integral, and between $X^{\odot \star}$ and $X^\odot$ in the second integral, recall \cref{sec:periodic center manifolds}. It follows from \cite[Lemma 19 and Proposition 20]{Article1} that these maps are nondegenerate and satisfy the adjoint relations
\begin{equation} \label{eq:curlyadjoint}
    \langle \mathcal{A}^\star \varphi^\star, \varphi \rangle_T = \langle \varphi^\star,\mathcal{A} \varphi \rangle_T, \quad \langle \mathcal{A}^{\odot \star} \varphi^{\odot \star}, \varphi^\odot \rangle_T = \langle \varphi^{\odot \star},\mathcal{A}^\odot \varphi^\odot \rangle_T,
\end{equation}
for all $\varphi \in \mathcal{D}(\mathcal{A}), \varphi^\star \in \mathcal{D}(\mathcal{A}^\star), \varphi^\odot \in \mathcal{D}(\mathcal{A}^\odot)$ and $\varphi^{\odot \star} \in \mathcal{D}(\mathcal{A}^{\odot \star})$. Moreover, $\mathcal{A}^\star$ and $\mathcal{A}^{\odot \star}$ are the unique linear operators satisfying \eqref{eq:curlyadjoint} since $\mathcal{A}$ and $\mathcal{A}^\odot$ are densely defined.

A vital tool in the theory of characteristic matrices is that the Jordan chains of the (adjoint) characteristic matrix induce (adjoint) (generalized) eigenfunctions of the generator, see \cite{Kaashoek1992,Diekmann1995,Bosschaert2020,Bosschaert2024a}. It turns out that a similar construction holds for characteristic operators since the ordered set of functions $\{q_0,\dots,q_{m_\lambda - 1} \}$ satisfying \eqref{eq:q_i} forms a \emph{Jordan chain} (of rank $m_\lambda$) for $\Delta(\sigma)$, and the ordered set of functions $\{p_{m_\lambda - 1},\dots,p_0 \}$ satisfying \eqref{eq:p_i} forms a \emph{Jordan chain} (of rank $m_\lambda$) for $\Delta^\star(\sigma)$, see \cite[Section IV.4]{Diekmann1995} and \cite[Section I.1.2]{Kaashoek1992}. Moreover, we call $q_0$ a \emph{null function} of $\Delta(\sigma)$ and $p_0$ a \emph{null function} of $\Delta^\star(\sigma)$.

\begin{remark} \label{remark:Tantiperiodic}
When a Floquet multiplier $\lambda \in \mathbb{R}_{-}$, for example in the case of the period-doubling bifurcation ($\lambda = -1$), it is conventional to work with $T$-antiperiodic functions \cite{Iooss1999,Kuznetsov2005,Article1}. The existence and $C^{k+1}$-smoothness of $T$-antiperiodic (adjoint) (generalized) eigenfunctions was proven in \cite[Proposition 10 and 16]{Article1}. When the elements of the Jordan chains $\{q_0,\dots,q_{m_\lambda - 1} \}$ and $\{p_{m_\lambda - 1},\dots,p_0\}$ are $T$-antiperiodic, it follows from \eqref{eq:varphi_i} and \eqref{eq:varphi_i^sun} that the (adjoint) (generalized) eigenfunctions $T$-antiperiodic. Furthermore, the (adjoint) characteristic operator and pairings can still be defined on a subspace of the function space consisting of $T$-antiperiodic $\mathbb{C}^{n}$- or $\mathbb{C}^{n \star}$-valued continuous functions. \hfill $\lozenge$
\end{remark}

Our next goal it to characterize the pairings between adjoint eigenfunctions and (generalized) eigenfunctions, as this will be important in \cref{sec:periodic normalization}. As we are interested in codimension one bifurcations of limit cycles (\cref{subsec:periodic normal forms}), we will restrict ourselves to two different cases regarding the algebraic multiplicity of the Floquet multiplier. First, when studying the period-doubling and Neimark-Sacker bifurcation, the multipliers are simple. Second, when studying the fold bifurcation, the trivial multiplier will have algebraic multiplicity two and geometric multiplicity one. We start by studying the first, more simple case. The autonomous analogue of the following result can be found in \cite[Corollary IV.5.12]{Diekmann1995}.
\begin{proposition} \label{prop:pairing}
Let $\lambda$ be a simple Floquet multiplier with $\sigma$ its associated Floquet exponent. If $q$ is a null function of $\Delta(\sigma)$, then $\varphi$ defined by
\begin{equation} \label{eq:simple eig}
    \varphi(\tau)(\theta) = e^{\sigma\theta} q(\tau + \theta), \quad \forall \tau \in \mathbb{R}, \ \theta \in [-h,0],
\end{equation}
is an eigenfunction of $\mathcal{A}$ corresponding to $\sigma$. Furthermore, if $p$ is a null function of $\Delta^\star(\sigma)$, then $\varphi^\odot$ defined by
\begin{equation} \label{eq:adjointsimple}
    \varphi^\odot (\tau) = \bigg(p(\tau), \theta \mapsto \int_\theta^h e^{\sigma(\theta -s)}p(\tau + s - \theta) d_2 \zeta(\tau + s - \theta,s)\bigg), \quad \forall \tau \in \mathbb{R},
\end{equation}
is an adjoint eigenfunction of $\mathcal{A}^\star$ corresponding to $\sigma$. Moreover, the following identities hold:
\begin{equation} \label{eq:pairingsimple}
    \langle \varphi^\odot, \varphi \rangle_T = \langle p, \Delta'(\sigma) q \rangle_T \neq 0, \quad \langle \varphi^\odot, \tau \mapsto A(\tau)\varphi(\tau) \rangle_T = \langle p, \Delta'(\sigma)[\sigma q + \dot{q} ]\rangle_T.
\end{equation}
\begin{proof}
The formulas for the eigenfunction and adjoint eigenfunction follow directly from \cref{thm:eigenfunctions} and \cref{thm:adjoint eigenfunctions2}. To prove the pairing identities \eqref{eq:pairingsimple}, recall from the duality pairing \eqref{eq:Xsunpairing} that
\begin{equation} \label{eq:pairingcomp}
    \langle \varphi^\odot, \varphi  \rangle_T = \int_0^T p(\tau)q(\tau) d\tau + \int_0^T \int_0^h \int_\theta^h e^{- \sigma s} p(\tau + s - \theta) d_2 \zeta(\tau+s-\theta,s)q(\tau-\theta) d\theta d\tau.
\end{equation}
Perform in the triple integral the change of variables: $\tau = u+v$ and $\theta = v$ and use the fact that $u \mapsto \int_v^h e^{-\sigma s} p(u+s) d_2 \zeta(u+s,s)q(u)$ is $T$-periodic due to the $T$-periodicity of $p,q$ and $\tau \mapsto \zeta(\tau,\cdot)$ since then the triple integral equals
\begin{equation*}
    \int_0^T \int_0^h \int_v^h e^{-\sigma s} p(u+s) d_2 \zeta(u+s,s) dv q(u) du = \int_0^T \int_0^h e^{-\sigma s} s p(u+s) d_2 \zeta(u+s,s) q(u) du,
\end{equation*}
where we used twice Fubini's theorem, but on different integrals. Hence,
\begin{equation*}
    \langle \varphi^\odot, \varphi  \rangle_T = \int_0^T p(\tau)q(\tau) + \int_0^h e^{-\sigma s} s p(\tau+s) d_2 \zeta(\tau+s,s) q(\tau) d\tau = \langle \Delta^{\star \prime }(\sigma)p, q \rangle_T = \langle p, \Delta'(\sigma)q \rangle_T,
\end{equation*}
where the last equality follows from \eqref{eq:pairing charac deri}. It remains to show that this pairing is non-vanishing. Let us first assume that $\lambda \in \mathbb{C} \setminus \mathbb{R}_{-}$. Then, recalling \cite[Corollary 25]{Article1} yields
\begin{equation} \label{eq:directsum1}
    C_T(\mathbb{R},X) = \mathcal{N}(\sigma I - \mathcal{A}) \oplus \overline{\mathcal{N}(\sigma I -\mathcal{A}^\star)^{\perp}},
\end{equation}
where the annihilator $\perp$ is defined in terms of the first pairing from \eqref{eq:pairingT}. If $\langle \varphi^\odot, \varphi \rangle_T = 0$, then $\varphi \in \mathcal{N}(\sigma I -\mathcal{A}^\star)^{\perp}$ but there also holds $\varphi \in \mathcal{N}(\sigma I - \mathcal{A})$ and so the eigenfunction $\varphi$ must be by \eqref{eq:directsum1}, a contradiction. When $\lambda \in \mathbb{R}_{-}$, the result follows from \cite[Remark 25]{Article1}.

Let us now prove the second equation of \eqref{eq:pairingsimple}. It follows from \eqref{eq:D(As)} and \eqref{eq:simple eig} that
\begin{equation*}
    \langle \varphi^\odot, \tau \mapsto A(\tau)\varphi(\tau) \rangle_T = \int_0^T \langle \varphi^\odot(\tau), \theta \mapsto e^{\sigma \theta}[\sigma q(\tau+\theta) + \dot{q}(\tau + \theta)] \rangle d\tau.
\end{equation*}
By linearity, the right-hand side can be split into two parts for which the first term equals $\sigma \langle p, \Delta'(\sigma) q \rangle_T$ due to \eqref{eq:pairingsimple}, and the second term equals $\int_0^T \langle \varphi^\odot(\tau), \theta \mapsto e^{\sigma \theta} \dot{q}(\tau + \theta) \rangle d\tau$. This expression equals the right hand-side of \eqref{eq:pairingcomp} where $q$ must be substituted by $\dot{q}$. The exact same proof can be followed to conclude that this term equals $\langle p,\Delta'(\sigma) \dot{q} \rangle_T$, which proves the second equation of \eqref{eq:pairingsimple}.
\end{proof}
\end{proposition}

As the pairing in \eqref{eq:pairingsimple} is nonzero, it is clear that $\langle \varphi^\odot, \varphi \rangle_T$ can be normalized to $1$. This will be done in \cref{sec:periodic normalization} to select in a consistent and systematic way a unique (adjoint) eigenfunction for the period-doubling and Neimark-Sacker bifurcation. For the fold bifurcation, we need the following result, for which its autonomous analogue can be found in \cite[Lemma 2.7]{Janssens2010}.
\begin{proposition} \label{prop:pairing2}
Let $\lambda$ be a Floquet multiplier of algebraic multiplicity two and geometric multiplicity one with $\sigma$ it associated Floquet exponent. If $\{q_0,q_1\}$ is a Jordan chain for $\Delta(\sigma)$, then $\varphi_{0,1}$ defined by
\begin{equation*}
    \varphi_0(\tau)(\theta) = e^{\sigma \theta} q_0(\tau + \theta), \quad \varphi_1(\tau)(\theta) = e^{\sigma \theta} (\theta q_0(\tau + \theta) + q_1(\tau + \theta)), \quad \forall \tau \in \mathbb{R}, \ \theta \in [-h,0],
\end{equation*}
are an eigenfunction and a generalized eigenfunction of $\mathcal{A}$ corresponding to $\sigma$. Furthermore, if $p$ is a null function of $\Delta^\star(\sigma)$, then $\varphi^\odot$ defined as in \eqref{eq:adjointsimple} is an adjoint eigenfunction of $\mathcal{A}^\star$ corresponding to $\sigma$. Moreover, the following identities hold:
\begin{gather*}
    \langle \varphi^\odot,\varphi_0 \rangle_T = \langle p, \Delta'(\sigma)q_0 \rangle_T = 0, \quad \langle \varphi^\odot,\varphi_1 \rangle_T = \langle p, \Delta'(\sigma)q_1 \rangle_T + \frac{1}{2} \langle p, \Delta''(\sigma)q_0 \rangle_T \neq 0,\\
    \langle \varphi^\odot,\tau \mapsto A(\tau)\varphi_1(\tau) \rangle_T = \langle p, \Delta'(\sigma)[\sigma q_1 + \dot{q}_1]\rangle_T + \frac{1}{2} \langle p, \Delta''(\sigma)[\sigma q_0 + \dot{q}_0] \rangle_T.
\end{gather*}
\begin{proof}
The formulas for the (generalized) eigenfunction and adjoint eigenfunction follow directly from \cref{thm:eigenfunctions} and \cref{thm:adjoint eigenfunctions2}. That $\langle \varphi^\odot,\varphi_0 \rangle_T = \langle p, \Delta'(\sigma)q_0 \rangle_T$ is just the same proof as \cref{prop:pairing}. To show that this pairing vanishes, note that $\langle \varphi^\odot,\varphi_0 \rangle_T = \langle \varphi^\odot, (\mathcal{A}-\sigma I)\varphi_1 \rangle_T = \langle (\mathcal{A}^\star - \sigma I)\varphi^\odot, \varphi_1 \rangle_T = 0,$ where we used \eqref{eq:curlyadjoint} in the second equality, and the fact that $\varphi^\odot$ is an adjoint eigenfunction of $\mathcal{A}^\star$.

To show the second pairing identity, one uses similar techniques as in the proof of \cref{prop:pairing}, see also \cite[Lemma 2.7]{Janssens2010} for a proof in the setting of characteristic matrices. To prove that this pairing is non-vanishing, let us first assume that $\lambda \in \mathbb{C} \setminus \mathbb{R}_{-}$. According to \cite[Corollary 25]{Article1}, there holds
\begin{equation} \label{eq:directsum2}
    C_T(\mathbb{R},X) = \mathcal{N}((\sigma I - \mathcal{A})^2) \oplus \overline{\mathcal{N}((\sigma I -\mathcal{A}^\star)^2)^{\perp}},
\end{equation}
where the annihilator $\perp$ is defined in terms of the pairing from \eqref{eq:pairingT}. Suppose that the pairing vanishes, then $\varphi \in \mathcal{N}((\sigma I - \mathcal{A}^\star)^2)^\perp$, but we also have that $\varphi \in \mathcal{N}((\sigma I - \mathcal{A})^2)$ and so the eigenfunction $\varphi$ must be zero by \eqref{eq:directsum2}, a contradiction. When $\lambda \in \mathbb{R}_{-}$, the result also holds due to \cite[Remark 25]{Article1}.

To prove the last identity, notice that
\begin{equation*}
    \langle \varphi^\odot,\tau \mapsto A(\tau)\varphi_1(\tau) \rangle_T = \int_0^T \langle \varphi^\odot(\tau), \theta \mapsto \sigma \varphi_1(\tau)(\theta) + \varphi_0(\tau)(\theta) + e^{\sigma \theta}(\theta \dot{q}_0(\tau + \theta) + \dot{q}_1(\tau + \theta)) \rangle d\tau.
\end{equation*}
where we used \eqref{eq:D(As)}. Using the previous derived identities together with a similar argument as used in the proof of \cref{prop:pairing}, the result follows.
\end{proof}
\end{proposition}

\subsection{Solvability of periodic linear operator equations} \label{subsec:solvability}
When computing the critical normal form coefficients in \cref{sec:periodic normalization} using \eqref{eq:homologicalequation}, we will encounter periodic linear operator equations of the form
\begin{equation} \label{eq:solvability}
    (zI - \mathcal{A}^{\odot \star})(v_0,v) = (w_0,w),
\end{equation}
where $z \in \mathbb{C}$, $(w_0,w) \in C_T(\mathbb{R},X^{\odot \star})$ is given, and $(v_0,v) \in \mathcal{D}(\mathcal{A}^{\odot \star})$ is the unknown. We only discuss here the solvability of $T$-periodic linear operator equations, as the solvability for $T$-antiperiodic linear operator equations is treated analogously. In general, both $z$ and the right-hand side of \eqref{eq:solvability} will have a nontrivial imaginary part and so it is necessary to regard \eqref{eq:solvability} as the complexification of the original periodic operator equations, see \cref{remark:complexification}. Before we discuss solvability of \eqref{eq:solvability}, we need a representation of the resolvent operator of $\mathcal{A}^{\odot \star}$. The autonomous analogue of the following result can be found in \cite[Corollary IV.5.4]{Diekmann1995} or \cite[Lemma 3.3]{Janssens2010}.

\begin{proposition} \label{prop:solvability}
If $z \in \mathbb{C}$ is such that $\Delta(z)$ is invertible, then the resolvent of $\mathcal{A}^{\odot \star}$ at $z$ has the following explicit representation
\begin{equation} \label{eq: v0,v}
    (v_0,v) = (zI - \mathcal{A}^{\odot \star})^{-1} (w_0,w),
\end{equation}
where
\begin{equation} \label{eq:solvabilityv}
    v(\tau)(\theta) = e^{z \theta} v_0(\tau + \theta) + \int_\theta^0 e^{z(\theta -s)} w(\tau + \theta - s)(s) ds, \quad \forall \tau \in \mathbb{R}, \ \theta \in [-h,0],
\end{equation}
and $v_0$ is the unique solution of the periodic linear (inhomogeneous) DDE
\begin{equation} \label{eq:solvabilityu}
    v_0 = \Delta(z)^{-1}\bigg( \tau \mapsto w_0(\tau) + \int_0^h d_2 \zeta(\tau,\theta)\int_{-\theta}^{0} e^{-z(\theta + s)} w(\tau -\theta - s)(s) ds \bigg).
\end{equation}
\begin{proof}
Let us first assume that $(w_0,w) \in C_T^1(\mathbb{R},X^{\odot \star})$. The second component of \eqref{eq: v0,v} is equivalent to solving the first order linear inhomogeneous PDE
\begin{equation*}
\bigg( \frac{\partial}{\partial \tau} - \frac{\partial}{\partial \theta} +z \bigg) v(\tau)(\theta) = w(\tau)(\theta),    
\end{equation*}
which has the solution specified in \eqref{eq:solvabilityv}. The first component of \eqref{eq: v0,v} is equivalent to
\begin{equation*}
    \dot{v}_0(\tau) - \int_{0}^h d_2 \zeta (\tau,\theta)v(\tau)(-\theta) + zv(\tau)(0) = w_0(\tau).
\end{equation*}
Substituting \eqref{eq:solvabilityv} into the equation above yields
\begin{equation} \label{eq:Deltazv0}
    (\Delta(z)v_0)(\tau) = w_0(\tau) + \int_0^h d_2 \zeta(\tau,\theta)\int_{-\theta}^{0} e^{-z(\theta + s)} w(\tau -\theta - s)(s) ds,
\end{equation}
which proves \eqref{eq:solvabilityu}. Let us now show that $(v_0,v) \in \mathcal{D}(\mathcal{A}^{\odot \star})$. Since $(w_0,w) \in C_T^1(\mathbb{R},X^{\odot \star})$, the right-hand side of \eqref{eq:Deltazv0} is $C^1$-smooth and so $v_0$, as a solution of the periodic linear (inhomogeneous) DDE \eqref{eq:Deltazv0}, is at least in $C_T^1(\mathbb{R},\mathbb{C}^n)$. Since $w$ is $C^1$-smooth in the first component, it follows immediately from \eqref{eq:solvabilityv} that $(v_0,v) \in C_T^1(\mathbb{R},X^{\odot \star})$. Let us now prove that $(v_0(\tau),v(\tau)) \in \mathcal{D}(A^{\odot \star}(\tau))$ for all $\tau \in \mathbb{R}$. Notice from \eqref{eq:solvability} that $v(\tau)(0) = v_0(\tau)$ and $v(\tau)$ is clearly in $\Lip([-h,0],\mathbb{R}^n)$ since $(w_0,w) \in C_T^1(\mathbb{R},X^{\odot \star})$ by assumption. Hence, $(v_0,v) \in \mathcal{D}(\mathcal{A}^{\odot \star})$ whenever $(w_0,w) \in C_T^1(\mathbb{R},\mathbb{C}^n)$ and so $C_T^1(\mathbb{R},X^{\odot \star}) \subseteq \mathcal{R}(zI - \mathcal{A}^{\odot \star}) \subseteq C_T(\mathbb{R},X^{\odot \star})$, which proves that $zI - \mathcal{A}^{\odot \star}$ has dense range and $\mathcal{D}(\mathcal{A}^{\odot \star}) \subseteq \mathcal{R}(zI - \mathcal{A}^{\odot \star})$. To show that the resolvent of $ \mathcal{A}^{\odot \star}$ at $z$ is bounded, let $(w_0,w) \in \mathcal{R}(zI-\mathcal{A}^{\odot \star})$ be given and note that $\| v_0 \|_\infty \leq ( {1+ M_z \mbox{TV}_\zeta } ) \| \Delta(z)^{-1} \|  \| w \|_\infty$ with $M_z \coloneqq (1-e^{-h \Re(z)})/\Re(z)$, where we recall that the operator norm $\| \Delta(z)^{-1} \| < \infty$ as $\Delta(z)^{-1}$ is a bounded linear operator on $C_T(\mathbb{R},\mathbb{C}^n)$ since $\Delta(z)$ is closed, recall \cref{lemma:Delta(z)closed}. Then,
\begin{equation*}
    \| (zI-\mathcal{A}^{\odot \star})^{-1} (w_0,w) \|_\infty \leq [\max\{1,e^{-h \Re(z)} \} \| \Delta(z)^{-1} \| ( {1+ M_z \mbox{TV}_\zeta }) + M_z ] \| (w_0,w) \|_\infty,
\end{equation*}
which shows that the resolvent of $\mathcal{A}^{\odot \star}$ at $z$ is a well-defined bounded linear operator.
\end{proof}
\end{proposition}

To identify the spectrum of $\mathcal{A}^{\odot \star}$, let us first recall from \cite[Proposition 21]{Article1} that
\begin{equation} \label{eq:pointspectracurly}
    \sigma_p(\mathcal{A}) = \sigma_p(\mathcal{A}^\star) = \sigma_p(\mathcal{A}^\odot) = \sigma_p(\mathcal{A}^{\odot \star}) =\{ \sigma \in \mathbb{C} : \sigma \mbox{ is a Floquet exponent} \}.
\end{equation}

\begin{corollary} \label{cor:equivalencefloquet}
The spectrum of $\mathcal{A}^{\odot \star}$ consists of point spectrum only and is given by
\begin{equation} \label{eq:spectraequivalent}    
\sigma(\mathcal{A}^{\odot \star}) = \{ z \in \mathbb{C} : \Delta(z) \mbox{ is not invertible} \} =\{ \sigma \in \mathbb{C} : \sigma \mbox{ is a Floquet exponent} \}.
\end{equation}
\end{corollary}
\begin{proof}
Note that $\sigma(\mathcal{A}^{\odot \star}) \subseteq \{ z \in \mathbb{C} : \Delta(z) \mbox{ is not invertible} \} = \{\sigma \in \mathbb{C} : \sigma \mbox{ is a Floquet exponent} \} = \sigma_p(\mathcal{A}^{\odot \star})$, where the first inclusion follows from \cref{prop:solvability}, the first equality from \cref{lemma:Delta(z)closed} and the second equality from \eqref{eq:pointspectracurly}. However, as $\sigma_p(\mathcal{A}^{\odot \star}) \subseteq \sigma(\mathcal{A}^{\odot \star}),$ the proof is complete. 
\end{proof}

\begin{corollary} \label{cor:spectra}
The spectra of $\mathcal{A}, \mathcal{A}^\star, \mathcal{A}^\odot$ and $\mathcal{A}^{\odot \star}$ consist of point spectrum only
\begin{equation*} \label{eq:pointspectra}
    \sigma(\mathcal{A}) = \sigma(\mathcal{A}^\star) = \sigma(\mathcal{A}^\odot) = \sigma(\mathcal{A}^{\odot \star}) = \{ \sigma \in \mathbb{C} : \sigma \mbox{ is a Floquet exponent} \},
\end{equation*}
which relates to the (adjoint) characteristic operator as
\begin{alignat*}{2}
    \sigma(\mathcal{A}) &= \{ \sigma \in \mathbb{C} : \Delta(\sigma) \mbox{ is not invertible} \}  &&= \{ \sigma \in \mathbb{C} : \sigma \mbox{ is a characteristic value of } \Delta \} \\
    &= \{ \sigma \in \mathbb{C} : \Delta^\star(\sigma) \mbox{ is not invertible} \}  &&= \{ \sigma \in \mathbb{C} : \sigma \mbox{ is a characteristic value of } \Delta^\star \}.
\end{alignat*}
\end{corollary}
\begin{proof}
Analogously to \cref{prop:solvability}, one can prove that $v = (zI-\mathcal{A})^{-1}w$, where $v$ is defined in \eqref{eq:solvabilityv} with $v_0 = v(0)$ and $w_0 = w(0)$. Hence, it follows from \eqref{eq:pointspectracurly} and a similar proof of \cref{cor:equivalencefloquet} that $\sigma(\mathcal{A}) = \sigma(\mathcal{A}^{\odot \star})$ consists of point spectrum (Floquet exponents) only. We claim that
\begin{equation} \label{eq:invertibleDeltaDeltastar}
    \{ \sigma \in \mathbb{C} : \Delta(\sigma) \mbox{ is not invertible} \} = \{ \sigma \in \mathbb{C} : \Delta^\star(\sigma) \mbox{ is not invertible} \}.
\end{equation}
Let $\sigma \in \mathbb{C}$ be such that $\Delta(\sigma)$ is not invertible. According to \cref{lemma:Delta(z)closed}, $\sigma$ is a Floquet exponent and thus $\sigma \in \sigma_p(\mathcal{A}^\star)$ by \eqref{eq:pointspectracurly}. Hence, $\Delta^\star(\sigma)p=0$ for some $p \in C_T^1(\mathbb{R},\mathbb{C}^{n \star})$ by \cref{thm:adjoint eigenfunctions2} and so $\Delta^\star(\sigma)$ is not injective and thus not invertible. To prove the converse, one uses \cref{lemma:adjoint charac} instead of \cref{lemma:Delta(z)closed} and \cref{thm:eigenfunctions} instead of \cref{thm:adjoint eigenfunctions2}. Furthermore, 
\begin{equation} \label{eq:spectraAstar}
    \sigma(\mathcal{A}^{\star}) \subseteq \{ \sigma \in \mathbb{C} : \Delta^\star(z) \mbox{ is not invertible} \} = \{\sigma \in \mathbb{C} : \sigma \mbox{ is a Floquet exponent} \} = \sigma_p(\mathcal{A}^{\star})
\end{equation}
where the first inclusion follows from a dual version of \cref{prop:solvability}, see also \cref{prop:resolventadjoint} for a full proof. The first equality is obtained from \eqref{eq:not inv is floquet} and \eqref{eq:invertibleDeltaDeltastar}, and the second equality follows from \eqref{eq:pointspectracurly}. The resolvent operator of $\mathcal{A}^\odot$ equals that of $\mathcal{A}^\star$, but is only defined on a smaller space. Therefore, we obtain \eqref{eq:spectraAstar} with $\mathcal{A}^\star$ replaced by $\mathcal{A}^\odot$.
\end{proof}

Solvability of \eqref{eq:solvability} shows us that there are two situations to consider depending on whether or not $z$ is a Floquet exponent. If $z$ is not a Floquet exponent then \cref{prop:solvability} tells us that $zI - \mathcal{A}^{\odot \star}$ has a densely defined bounded inverse, and \eqref{eq:solvability} admits a unique solution $(v_0,v)$ which can be computed explicitly by \cref{prop:solvability}. It turns out in \cref{sec:periodic normalization} that we have to solve \eqref{eq:solvability} for a particular right-hand side $(w_0,w)$ specified in \cref{cor:inverse}. The autonomous analogue of the following result can be found in \cite[Corollary 3.4]{Janssens2010}.
\begin{corollary} \label{cor:inverse}
Suppose that $z$ is not a Floquet exponent. If $(w_0,w) = ( \tau \mapsto w_0(\tau)r^{\odot \star})$, then the unique solution $(v_0,v)$ of \eqref{eq:solvability} can be represented by
    \begin{equation*}
    v(\tau)(\theta) = e^{z \theta} v_0(\tau + \theta), \quad \forall \tau \in \mathbb{R}, \ \theta \in [-h,0], \quad v_0 = \Delta^{-1}(z)w_0.
    \end{equation*}
\end{corollary}

Let us now turn our attention to the solvability of \eqref{eq:solvability} when $z = \sigma$ is a Floquet exponent. Then \eqref{eq:solvability} need not to have a solution. To find a solution, we will use a Fredholm alternative that is suited for periodic linear operator equations of the form \eqref{eq:solvability}.

\begin{proposition}[Fredholm solvability condition] \label{prop:FSC}
Suppose that $z = \sigma$ is a Floquet exponent. If $(v_0,v)$ is a solution of \eqref{eq:solvability} then
\begin{equation} \label{eq:FSC} \tag{FSC}
    \langle (w_0,w), \varphi^\odot \rangle_T = 0,
\end{equation}
where $\varphi^\odot$ is an eigenfunction of $\mathcal{A}^\star$ corresponding to $\sigma$.
\end{proposition}
\begin{proof}
It follows immediately from \eqref{eq:pairingT} and \eqref{eq:solvability} that $\langle (w_0,w),\varphi^\odot \rangle_T = \langle (\sigma I-\mathcal{A}^{\odot \star})(v_0,v), \varphi^\odot \rangle_T = \langle (v_0,v), (\sigma I - \mathcal{A}^{\odot}) \varphi^\odot \rangle_T = 0$, since the adjoint eigenfunction $\varphi^\odot$ satisfies $\mathcal{A}^\odot \varphi^\odot = \sigma \varphi^\odot$, recall the observation below \cref{thm:adjoint eigenfunctions2}.
\end{proof}

If $z = \sigma$ is a Floquet exponent and \eqref{eq:solvability} is consistent, then any solution that does exist is not unique. Indeed, we may add to it any arbitrary linear combination of eigenfunctions of $\mathcal{A}^{\odot \star}$. The \emph{bordered operator inverse} $(\sigma I - \mathcal{A}^{\odot \star})^{\inv} : \mathcal{R}(\sigma I - \mathcal{A}^{\odot \star}) \to \mathcal{D}(\mathcal{A}^{\odot \star})$ is used to select a particular solution in a systematic and convenient way. For the case that $\sigma$ is a simple Floquet exponent, it assigns a solution to the extended linear system
\begin{equation} \label{eq:borderedinverse}
    (\sigma I - \mathcal{A}^{\odot \star}) (v_0,v) = (w_0,w), \quad \langle (v_0,v), \varphi^\odot \rangle_T = 0,
\end{equation}
for every $(w_0,w)$ for which \eqref{eq:solvability} is consistent, where $\varphi^\odot$ is an adjoint eigenfunction of $\mathcal{A}^\star$ at $\sigma$. The pairing in \eqref{eq:borderedinverse} may be evaluated explicitly in concrete cases using \eqref{eq:Xsunstarpairing}. This will be done many times in \cref{sec:periodic normalization} when we apply \eqref{eq:FSC} to specific periodic linear operator equations of the form presented in \cref{prop:borderedinversemu}. The autonomous analogue of the following result can be found in \cite[Proposition 3.6]{Janssens2010}.

\begin{proposition} \label{prop:borderedinversemu}
Suppose that $z = \sigma$ is a simple Floquet exponent and assume that \eqref{eq:solvability} is consistent for a given $(w_0,w)$. Let $q$ be a null function of $\Delta(\sigma)$ with associated eigenfunction $\varphi$ and let $p$ be a null function of $\Delta^\star(\sigma)$ with associated adjoint eigenfunction $\varphi^\odot$ normalized to $\langle \varphi^\odot,\varphi \rangle_T = 1$. Then a solution $(v_0,v)$ of \eqref{eq:solvability} is given by
\begin{equation} \label{eq:solvabilityv2}
    v(\tau)(\theta) = e^{\sigma \theta} v_0(\tau + \theta) + \int_\theta^0 e^{\sigma(\theta -s)} w(\tau + \theta - s)(s) ds, \quad \forall \tau \in \mathbb{R}, \ \theta \in [-h,0],
\end{equation}
where $v_0 = \xi + \mu q$ with $\xi \in C_T^1(\mathbb{R},\mathbb{C}^n)$ and $\mu \in \mathbb{C}$ are given by
\begin{align*}
    \xi &= \Delta(\sigma)^{\inv}\bigg( \tau \mapsto w_0(\tau) + \int_0^h d_2 \zeta(\tau,\theta)\int_{0}^{\theta} e^{-\sigma s} w(\tau - s)(s-\theta) ds \bigg), \\
    \mu &= - \langle p,\Delta'(\sigma)\xi \rangle_T \\
    &- \int_0^T \int_0^h \bigg(\int_\theta^h e^{-\sigma s}p(s+\tau-\theta) d_2\zeta(\tau + s - \theta,s) \bigg) \bigg(\int_{- \theta}^0 e^{-\sigma r} w(\tau - \theta - r)(r) dr \bigg) d\theta d\tau.
\end{align*}

\begin{proof}
Regarding the bordered operator inverse, we have to find a solution of \eqref{eq:borderedinverse} and so the representation for $v$, specified in \eqref{eq:solvabilityv2}, follows. Since $\sigma$ is simple, the null space of $\Delta(\sigma)$ is spanned by $q$, see \eqref{eq:q_i}, and therefore $v_0 = \xi + \mu q$ for some $\mu \in \mathbb{C}$. The value of $\mu$ is determined by the second requirement of \eqref{eq:borderedinverse}. According to \cref{thm:adjoint eigenfunctions2} and the formula for the duality pairing \eqref{eq:Xsunstarpairing}, we have
\begin{align*}
    \langle (v_0,v), \varphi^\odot \rangle_T &= \int_0^T p(\tau) v_0(\tau)  + \int_0^h \int_{\theta}^h e^{\sigma(\theta -s)}p(s+ \tau - \theta) d_2 \zeta(\tau  + s - \theta,s) v_0(\tau - \theta) d\theta d \tau \\
    &\hspace{-1.2cm}+ \int_0^T \int_0^h \bigg( \int_\theta^h e^{-\sigma s} p(\tau + s - \theta) d_2 \zeta(\tau  + s - \theta,s) \bigg) \bigg( \int_{-\theta}^0 e^{- \sigma r} w(\tau - \theta -r)(r) dr \bigg) d\theta d \tau.
\end{align*}
Let us denote the third term on the right-hand side by $I$. Using similar arguments as in the proof of \cref{prop:pairing}, we obtain for the first two terms
\begin{equation*}
    \int_0^T p(\tau) v_0(\tau)  + \int_0^h \int_{\theta}^h e^{\sigma(\theta -s)}p(s+ \tau - \theta) d_2 \zeta(\tau  + s - \theta,s) v_0(\tau - \theta) d\theta d \tau = \langle p, \Delta '(\sigma) v_0 \rangle_T.
\end{equation*}
Requiring $\langle (v_0,v), \varphi^\odot \rangle_T = 0$ implies that $\int_0^T p(\tau) \Delta '(\sigma)(\xi + \mu q) d\tau + I = 0$. Since $ \langle p, \Delta '(\sigma)q \rangle_T = \langle \varphi^\odot, \varphi \rangle_T = 1$ by assumption, it follows that $\mu = - \langle p, \Delta'(\sigma)\xi \rangle_T - I$ which is the desired result.
\end{proof}
\end{proposition}
We observe that the expression for $\xi$ in \cref{prop:borderedinversemu} involves a bordered operator inverse $\Delta(\sigma)^{\inv} : \mathcal{R}(\Delta(\sigma)) \to C_T^1(\mathbb{R},\mathbb{C}^n)$ which assigns a unique solution to the extended linear system
\begin{equation*}
    \Delta(\sigma)v = w, \quad \langle p,v\rangle_T = 0,
\end{equation*}
to every $w \in C_T^1(\mathbb{R},\mathbb{C}^n)$ for which $\Delta(\sigma)v = w$ is consistent, where $p$ is a null function of $\Delta^\star(\sigma)$. This notation was also used in \cite{Bosschaert2020,Bosschaert2024a,Janssens2010} for the bordered matrix inverse of the characteristic matrix. The properties of
bordered linear systems and their role in numerical bifurcation analysis are discussed more extensively in \cite{Keller1987,Kuznetsov2023a} and \cite[Chapter 3]{Govaerts2000}.

We shall encounter in \cref{sec:periodic normalization} the following special case, for which its autonomous analogue can be found in \cite[Corollary 3.7]{Janssens2010} or \cite[Lemma 4]{Bosschaert2020}.

\begin{corollary} \label{cor:borderedinversemu}
Suppose in addition to \cref{prop:borderedinversemu} that $(w_0,w) = (\eta,0) + \kappa(q,\varphi)$, where $\kappa \in \mathbb{C}$ and $\eta \in C_T(\mathbb{R},\mathbb{C}^n)$. Then,
\begin{equation*}
    v(\tau)(\theta) = e^{\sigma \theta}(v_0(\tau + \theta) - \kappa \theta q(\tau+\theta)), \quad \forall \tau \in \mathbb{R}, \ \theta \in [-h,0], 
\end{equation*}
where $v_0 = \xi + \mu q$ with $\xi = \Delta(\sigma)^{\inv} [ \eta + \kappa \Delta '(\sigma) q ]$ and $ \mu = -\langle p, \Delta'(\sigma)\xi \rangle_T - \frac{\kappa}{2} \langle p, \Delta''(\sigma)q \rangle_T$.
\begin{proof}
We note that $\eta$ and $\kappa$ are uniquely determined for a given right-hand side $(w_0,w)$. The result follows now directly from \cref{prop:borderedinversemu}. The formula for $\mu$ is obtained by applying Fubini's theorem as performed in the proof of \cref{prop:pairing}.
\end{proof}
\end{corollary}
We will use in \cref{sec:periodic normalization} the shorthand notation $v = B_{\sigma}^{\inv}(\eta,\kappa)$ for a solution $(v_0,v)$ of \eqref{eq:solvability} as described in \cref{cor:borderedinversemu}.

\begin{remark} \label{remark:ODE}
In the setting of finite-dimensional ODEs, the above construction was used extensively in \cite{Kuznetsov2005,Witte2013,Witte2014} to study codimension one and two bifurcations of limit cycles of
\begin{equation} \label{eq:ODEremark}
    \dot{x}(t) = f(x(t)),
\end{equation}
where $x(t) \in \mathbb{R}^n$ and the vector field $f : \mathbb{R}^n \to \mathbb{R}^n$ is sufficiently smooth. Since $\mathbb{R}^n$ is reflexive and finite-dimensional, the sun-star calculus construction becomes trivial \cite[Section 1]{Neerven1992} and thus $\mathcal{A} = \mathcal{A}^{\odot \star}$ and $\mathcal{A}^{\star} = \mathcal{A}^\odot$, where the (complexified) densely defined unbounded linear operators $\mathcal{A} : \mathcal{D}(\mathcal{A}) \subseteq  C_T(\mathbb{R},\mathbb{C}^n) \to C_T(\mathbb{R},\mathbb{C}^n)$ and $\mathcal{A}^\star : \mathcal{D}(\mathcal{A}^\star) \subseteq C_T(\mathbb{R},\mathbb{C}^{n \star}) \to C_T(\mathbb{R},\mathbb{C}^{n \star})$ now take the form $(\mathcal{A}\varphi)(\tau) = A(\tau)\varphi(\tau) - \dot{\varphi}(\tau)$ and $(\mathcal{A}^\star\varphi^\star)(\tau) = \varphi^\star(\tau) A^\star(\tau) + \dot{\varphi}^\star(\tau)$, with domains $\mathcal{D}(\mathcal{A}) = C_T^1(\mathbb{R},\mathbb{C}^n)$ and $\mathcal{D}(\mathcal{A}^\star) = C_T^1(\mathbb{R},\mathbb{C}^{n \star})$, where $A(\tau) = A^\star(\tau) = Df(\gamma(\tau)) \in \mathbb{R}^{n \times n}$ and $\gamma$ denotes a (nonhyperbolic) $T$-periodic solution of \eqref{eq:ODEremark}. Recall from \cite[Remark 8 and 15]{Article1} that the unbounded linear operators $\mathcal{A} = \mathcal{A}^{\odot \star}$ and $\mathcal{A}^{\star} = \mathcal{A}^\odot$ are all closed in the finite-dimensional ODE-setting, while all these operators are not closed but closable in the classical DDE-setting. It is more traditional \cite{Iooss1988,Iooss1999,Kuznetsov2005,Kuznetsov2023a,Witte2013,Witte2014} to write the operators $\mathcal{A}$ and $\mathcal{A}^\star$ as $-\frac{d}{d\tau}+A(\tau)$ and $\frac{d}{d\tau}+A^\star(\tau)$, respectively. Strictly speaking, this notation requires a special interpretation as $\frac{d}{d\tau}$ acts on functions in $C_T^1(\mathbb{R},X) \subseteq C_T(\mathbb{R},X)$ and $A(\tau)$ acts on functions in $\mathcal{D}(A(\tau)) \subseteq X$. We are confident that this notation would not create confusion for the reader. Moreover, the periodic linear operator equation \eqref{eq:solvability} reduces to $(\frac{d}{d\tau} - A(\tau) + z) v(\tau) = w(\tau)$, where $w \in C_T(\mathbb{R},\mathbb{C}^n)$ is given and $v \in C_T^1(\mathbb{R},\mathbb{C}^n)$ is the unknown. Recall from \eqref{eq:Xsunpairing} and \eqref{eq:Xsunstarpairing} that the nondegenerate pairings are bilinear maps since the natural duality pairings are bilinear. In the finite-dimensional ODE-setting, one can rely on the Hilbert space structure of $\mathbb{C}^n$ and therefore use sesquilinear maps. As a consequence, $-\sigma$ is an eigenvalue of the operator $\frac{d}{d\tau} + A^\star(\tau)$ instead of $\sigma$, see \cite[Equation (2.34)]{Kuznetsov2005} together with the comment below \cite[Equation (3.27)]{Witte2013} for more information. \hfill $\lozenge$
\end{remark}

\section{Periodic normalization on the center manifold} \label{sec:periodic normalization}
In this section, we derive the explicit formulas for the critical normal form coefficients of all codimension one bifurcations of limit cycles in classical DDEs via the periodic normalization method. The recipe of this method will be discussed in \cref{subsec: periodic normal forms on the center manifold}. A key ingredient here is the \emph{homological equation} \eqref{eq:homologicalequation} which describes the dynamics of \eqref{eq:DDE} on $\mathcal{W}_{\loc}^c(\Gamma)$ in terms of one of the periodic normal forms from \cref{subsec:periodic normal forms}. The preparatory work from \cref{subsec: periodic normal forms on the center manifold} provides us in \cref{subsec:fold}, \cref{subsec:PD} and \cref{subsec:NS} rather compact and explicit computational formulas for the critical normal form coefficients of the fold \eqref{eq:normalformcoeff fold}, period-doubling \eqref{eq:normalformcoeff pd} and Neimark-Sacker bifurcation \eqref{eq:normalformcoeff NS}.

\subsection{Computation of critical coefficients} \label{subsec: periodic normal forms on the center manifold}
Suppose that there are $n_0 + 1$ Floquet multipliers (counted with algebraic multiplicity) on the unit circle. We can assume that a parameterization of the $(n_0 +1)$-dimensional center manifold $\mathcal{W}_{\loc}^c(\Gamma)$ has been selected so that the
restriction of \eqref{eq:DDE} to this invariant manifold has one of the periodic normal forms from \cref{subsec:periodic normal forms}. All these normal forms are of the form
\begin{equation} \label{eq:periodic normal form}
    \begin{dcases}
    \dot{\tau} = 1 + p(\xi) + r(\tau,\xi), \\
    \dot{\xi} = P(\xi) + R(\tau,\xi).
    \end{dcases}
\end{equation}
The maps $p$ and $P$ are polynomials in $\xi \in \mathbb{R}^{n_0}$ of some degree $ 1 \leq N \leq k$ without constant terms, while the functions $r$ and $R$ are $C^k$-smooth and $T$- or $2T$-periodic in $\tau$, see \cite[Section 2]{Witte2013}. Let $I \subseteq \mathbb{R}$ be an interval. If $u: I \to X$ with $u(t) \coloneqq x_t \in \mathcal{W}_{\loc}^{c}(\Gamma)$ for $t \in I$ is a solution of \eqref{eq:DDE}, then it follows from \cite[Theorem VII.6.1]{Diekmann1995} and \cite[Theorem 3.6]{Clement1988} that $u$ is differentiable on $I$ and satisfies the abstract ODE
\begin{equation} \label{eq:hom abstract ODE}
    j\dot{u}(t) = A_0^{\odot \star}ju(t) + G(u(t)), \quad \forall t \in I,
\end{equation}
where $G \in C^{k+1}(X,X^{\odot \star})$ is defined by $G(\varphi) \coloneqq [F(\varphi)]r^{\odot \star}$ for all $\varphi \in X$. According to \cref{thm:eigenfunctions}, we are able to choose a $T$- or $2T$-periodic $C^{k+1}$-smooth ordered basis $\{\dot{\gamma}_\tau,\varphi_1(\tau),\dots,\varphi_{n_0}(\tau) \}$ of the $(n_0+1)$-dimensional center subspace $X_0(\tau)$ for all $\tau \in \mathbb{R}$. It follows from \cite[Theorem 10, 11 and 12]{Article1} that there exists a locally defined $C^k$-smooth parametrization $\mathcal{H} : \mathbb{R} \times \mathbb{R}^{n_0} \to X$ of $\mathcal{W}_{\loc}^c(\Gamma)$ in $(\tau,\xi)$-coordinates, such that the solution $u$ of \eqref{eq:hom abstract ODE} on $\mathcal{W}_{\loc}^c(\Gamma)$ can be written as
\begin{equation} \label{eq:u is H}
    u(t) = \mathcal{H}(\tau(t),\xi(t)), \quad \forall t \in I,
\end{equation}
and the map $\mathcal{H}$ has in $(\tau,\xi)$-coordinates the representation
\begin{equation} \label{eq:Hparametrization}
    \mathcal{H}(\tau,\xi) = \gamma_\tau + \sum_{i=1}^{n_0} \xi_i \varphi_i(\tau) + H(\tau,\xi),
\end{equation}
where the nonlinear map $H$ is $C^k$-smooth. Combining \eqref{eq:hom abstract ODE} with \eqref{eq:u is H} yields the \emph{homological equation}
\begin{equation} \label{eq:homologicalequation} \tag{HOM}
    j\bigg(\frac{\partial \mathcal{H}(\tau,\xi)}{\partial \tau} \dot{\tau} + \frac{\partial \mathcal{H}(\tau,\xi)}{\partial \xi} \dot{\xi} \bigg) = A_{0}^{\odot \star}j(\mathcal{H}(\tau,\xi)) + G(\mathcal{H}(\tau,\xi)),
\end{equation}
where $\dot{\tau}$ and $\dot{\xi}$ are given by the right-hand side of the periodic normal form \eqref{eq:periodic normal form}. The unknowns in \eqref{eq:homologicalequation} are the map $H$, in particular its Taylor coefficients, and the coefficients of the polynomials $p$ and $P$ from \eqref{eq:periodic normal form}. These coefficients are determined up to a certain (finite) order of interest, depending on the type of bifurcation. We expand the nonlinearity $G$ around the cycle $\Gamma$ as
\begin{equation} \label{eq:Gexpand}
    G(\gamma_\tau + \psi(\tau)) = G(\gamma_\tau) + B(\tau)\psi(\tau) + \sum_{j = 2}^{k} \frac{1}{j!}D^{j}F(\gamma_\tau)(\psi(\tau)^{(j)})r^{\odot \star} + \mathcal{O}(\|\psi(\tau)\|_{\infty}^{k+1}),
\end{equation}
for all $\psi \in C_T(\mathbb{R},X)$. The bounded linear perturbation $B$ is specified in \eqref{eq:perturbationB} and the $j$-linear form $D^jF(\gamma_\tau) : X^j \to \mathbb{R}^n$ is the $j$th order Fr\'echet derivative of $F$ evaluated at $\gamma_\tau$ for $j=2,\dots,k$. The mapping $\mathcal{H}$ from \eqref{eq:Hparametrization} can then be expanded as
\begin{equation} \label{eq:expandH}
    \mathcal{H}(\tau,\xi) = \gamma_\tau + \sum_{i=1}^{n_0} \xi_i \varphi_i(\tau) + \sum_{|\mu| = 2}^{k} \frac{1}{\mu!} h_\mu(\tau) \xi^\mu + \mathcal{O}(|\xi|^{k+1}),
\end{equation}
where the last summation represents a Taylor expansion of the map $H$ up to order $k$. Notice that the maps $\tau \mapsto h_\mu(\tau)$ are $T$- or $2T$-periodic, see also \cite[Theorem 10, 11 and 12]{Article1} for more information. Substituting \eqref{eq:Gexpand} and \eqref{eq:expandH} into \eqref{eq:homologicalequation}, collecting coefficients of the terms $\xi^\mu$ from lower to higher order, and solving the resulting periodic linear operator equations, one can solve recursively for the unknown coefficients of the polynomials $p$ and $P$ and the center manifold expansions $h_\mu$ by applying the Fredholm solvability condition \eqref{eq:FSC}, as explained in \cref{subsec:solvability}. Furthermore, to determine the (adjoint) (generalized) eigenfunctions $\varphi_i$ and $\varphi_i^\odot$, and center manifold expansions $h_\mu$ uniquely, we will set up a $T$-(anti)periodic boundary-value problem on $[0,T]$ with an additional normalization or phase condition. The solutions of these functions on $\mathbb{R}$ are then obtained by a trivial $T$-(anti)periodic extension.

This method will be used in the following three subsections to derive the critical normal form coefficients of the fold (\cref{subsec:fold}), period-doubling (\cref{subsec:PD}) and Neimark-Sacker bifurcation (\cref{subsec:NS}). It is sufficient to expand the nonlinearity $G$ around the cycle $\Gamma$, as done in \eqref{eq:Gexpand}, up to the second-order coefficient for the fold bifurcation and to the third-order coefficient for the period-doubling and Neimark-Sacker bifurcation. For simplicity of notation, we will keep the tradition with the literature \cite{Kuznetsov2005,Witte2013,Witte2014} on periodic normalization and simply write
\begin{gather*}
    B(\tau;\psi(\tau),\psi(\tau)) \coloneqq D^2F(\gamma_\tau)(\psi(\tau)^{(2)}),  \quad C(\tau;\psi(\tau),\psi(\tau),\psi(\tau)) \coloneqq D^3F(\gamma_\tau)(\psi(\tau)^{(3)}), \\
    B(\psi_1,\psi_2) \coloneqq (\tau \mapsto B(\tau;\psi_1(\tau),\psi_2(\tau)), \quad C(\psi_1,\psi_2,\psi_3) \coloneqq (\tau \mapsto C(\psi_1(\tau),\psi_2(\tau),\psi_3(\tau))).
\end{gather*}
It is sufficient to take $k + 1 \geq 2$ for the fold bifurcation and $k + 1 \geq 3$ for the period-doubling and Neimark-Sacker bifurcation. This condition on $k$ will be assumed from now on. To keep tradition with the mentioned literature, we will work from now on with the notation from \cref{remark:ODE}.

\subsection{Fold bifurcation} \label{subsec:fold}
Suppose that \eqref{eq:DDE} has a $T$-periodic solution $\gamma$ with $\Lambda_0 = \{ 1 \},$ where the trivial Floquet multiplier on the unit circle has algebraic multiplicity two and geometric multiplicity one. Hence, there exist $T$-periodic (generalized) eigenfunctions $\varphi_{0,1}$ and an adjoint eigenfunction $\varphi^\odot$ satisfying
\begin{equation*} 
    \bigg( \frac{d}{d\tau} - A(\tau)\bigg)\varphi_1(\tau)=
    -\varphi_0(\tau), \quad \bigg( \frac{d}{d\tau} + A^\star(\tau)\bigg)\varphi^\odot(\tau)=0, \quad \langle \varphi^\odot, \varphi_0 \rangle_T = 0, \quad \langle \varphi^\odot, \varphi_1 \rangle_T = 1,
\end{equation*}
where $\varphi_0(\tau) = \dot{\gamma}_\tau = q_0(\tau+\cdot)$. The (adjoint) (generalized) eigenfunctions are given as in \cref{prop:pairing2} where $\{q_0,q_1\}$ is a $T$-periodic Jordan chain of $\Delta(0)$ and $p$ is a $T$-periodic null function of $\Delta^\star(0)$. To determine $\varphi_1$ and $\varphi^\odot$ uniquely, we also impose that $\langle q_1^\star,q_0 \rangle_T = 0$, where $q_1^\star$ is the Hermitian adjoint of $q_1$. The 2-dimensional center manifold $\mathcal{W}_{\loc}^{c}(\Gamma)$ at the fold bifurcation can be parametrized locally near $\Gamma$ in $(\tau,\xi)$-coordinates as
\begin{equation*}
    \mathcal{H}(\tau,\xi) = \gamma_\tau + \xi \varphi_1(\tau) + \frac{1}{2}h_2(\tau) \xi^2 + \mathcal{O}(\xi^3), \quad \tau \in \mathbb{R}, \ \xi \in \mathbb{R},
\end{equation*}
where $h_2$ is $T$-periodic. As every ingredient in \eqref{eq:homologicalequation} is introduced, we can fill everything into this equation and start comparing coefficients in $\xi^j$ for $j \geq 0$ sufficiently large. When this procedure is carried out, we find for the $\xi^0$- and $\xi^1$-terms in \eqref{eq:homologicalequation} the identities
\begin{equation*}
    \frac{d}{d\tau} j({\gamma}_\tau) = A_0^{\odot \star} j(\gamma_\tau) + G(\gamma_\tau), \quad \bigg(\frac{d}{d \tau} - A^{\odot \star}(\tau) \bigg)j\varphi_1(\tau) = -j\dot{\gamma}_\tau,
\end{equation*}
which was already known. Collecting the $\xi^2$-terms yield
\begin{equation*}
    \bigg(\frac{d}{d \tau} - A^{\odot \star}(\tau) \bigg)jh_2(\tau) = B(\tau;\varphi_1(\tau),\varphi_1(\tau))r^{\odot \star} - 2aj\dot{\gamma}_\tau - 2j\dot{\varphi}_1(\tau) - 2bj\varphi_1(\tau).
\end{equation*}
Since $0$ is a Floquet exponent, the Fredholm solvability condition implies that
\begin{equation*}
    \int_0^T \langle B(\tau;\varphi_1(\tau),\varphi_1(\tau))r^{\odot \star} - 2a j\dot{\gamma}_\tau - 2j\dot{\varphi}_1(\tau) - 2b j\varphi_1(\tau), \varphi^{\odot}(\tau) \rangle d\tau = 0.
\end{equation*}
Using the normalization identities, we obtain
\begin{equation*}
    b = \frac{1}{2} \int_0^T \langle B(\tau;\varphi_1(\tau),\varphi_1(\tau))r^{\odot \star} - 2A^{\odot \star}(\tau)j\varphi_1(\tau), \varphi^{\odot}(\tau) \rangle d\tau.
\end{equation*}
Therefore, the critical coefficient $b$ in the periodic normal form for the fold bifurcation \eqref{eq:normal form fold} has been computed, and it follows from \cref{prop:pairing2} that
\begin{equation}  \label{eq:normalformcoeff fold}
    b = \frac{1}{2} \langle p, B(\varphi_1,\varphi_1) - 2 \Delta'(0)\dot{q}_1 - \Delta''(0)\dot{q}_0 \rangle_T.
\end{equation}

\subsection{Period-doubling bifurcation} \label{subsec:PD}
Suppose that \eqref{eq:DDE} has a $T$-periodic solution $\gamma$ with $\Lambda_0 = \{ \pm 1 \},$ where all Floquet multipliers on the unit circle are simple. Hence, there exist $T$-antiperiodic (adjoint) eigenfunctions $\varphi$ and $\varphi^\odot$ satisfying
\begin{equation*} 
    \bigg( \frac{d}{d\tau} - A(\tau)\bigg)\varphi(\tau)=
    0, \quad \bigg( \frac{d}{d\tau} + A^\star(\tau)\bigg)\varphi^\odot(\tau)=0, \quad \langle \varphi^\odot, \varphi \rangle_T = 1.
\end{equation*}
The (adjoint) eigenfunctions are given as in \cref{prop:pairing}, where $q$ is a $T$-antiperiodic null function of $\Delta(0)$ and $p$ is a $T$-antiperiodic null function of $\Delta^\star(0)$. To determine  $\varphi$ and $\varphi^\odot$ uniquely, we also impose that $\langle q^{\star},q \rangle_T = 1$, where $q^\star$ is the Hermitian adjoint of $q$. The two-dimensional critical center manifold $\mathcal{W}_{\loc}^{c}(\Gamma)$ at the period-doubling bifurcation can be parametrized locally near $\Gamma$ in $(\tau,\xi)$-coordinates as
\begin{equation} \label{eq:H of PD}
    \mathcal{H}(\tau,\xi) = \gamma_\tau + \xi \varphi(\tau) + \frac{1}{2}h_2(\tau) \xi^2 + \frac{1}{6}h_3(\tau) \xi^3 + \mathcal{O}(\xi^4), \quad \tau \in \mathbb{R}, \ \xi \in \mathbb{R},
\end{equation}
where $h_j$ is $2T$-periodic. Since we are at the period-doubling point, it follows from \cite[Theorem 12]{Article1} that $\mathcal{H}(\tau,\xi) = \mathcal{H}(\tau+T,-\xi)$, meaning that $h_2$ is $T$-periodic and $h_3$ is $T$-antiperiodic. Collecting the $\xi^0$- and $\xi^1$-terms in \eqref{eq:homologicalequation} yield
\begin{equation*}
    \frac{d}{d\tau} j(\gamma_\tau) = A_0^{\odot \star} j(\gamma_\tau) + G(\gamma_\tau), \quad \bigg(\frac{d}{d \tau} - A^{\odot \star}(\tau) \bigg)j\varphi(\tau) = 0,
\end{equation*}
which was already known. Collecting the $\xi^2$-terms yield 
\begin{equation} \label{eq:PDxi^2terms}
    \bigg(\frac{d}{d \tau} - A^{\odot \star}(\tau) \bigg) jh_2(\tau) = B(\tau;\varphi(\tau),\varphi(\tau))r^{\odot \star} - 2aj\dot{\gamma}_\tau.
\end{equation}
Since $0$ is a Floquet exponent in the $T$-periodic setting, the Fredholm solvability condition tells us that
\begin{equation*}
    \int_0^{T} \langle B(\tau; \varphi(\tau),\varphi(\tau))r^{\odot \star} - 2a j\dot{\gamma}_\tau, \psi^{\odot}(\tau) \rangle  d\tau = 0, \quad \bigg(\frac{d}{d \tau} + A^\star(\tau) \bigg)\psi^{\odot}(\tau) = 0, \quad \langle \psi^\odot, \varphi_0 \rangle_T = 1,
\end{equation*}
where $\varphi_0(\tau) = \dot{\gamma}_\tau = q_0(\tau + \cdot)$. The $T$-periodic adjoint eigenfunction $\psi^\odot = (r,g)$ is given as in \cref{prop:pairing}, where $r$ is a $T$-periodic null function of $\Delta^\star(0)$ and satisfies $\langle \psi^\odot,\varphi_0 \rangle_T = \langle r, \Delta'(0)q_0 \rangle_T = 1$, which is possible as $\langle \psi^{\odot},\varphi_0 \rangle_T \neq 0$ due to \cref{prop:pairing}. This leads to the expression
\begin{equation} \label{eq:pdcoeffa}
    a = \frac{1}{2} \langle r, B(\varphi,\varphi) \rangle_T.
\end{equation}
As we have an explicit formula for the coefficient $a$, we can solve \eqref{eq:PDxi^2terms} uniquely for the unknown coefficient $h_2$. Indeed, by an application of \cref{cor:borderedinversemu}, we obtain the unique solution satisfying $\langle \psi^\odot, h_2 \rangle_T = 0$ as
\begin{equation*}
h_2 = B_{0}^{\inv}(B(\varphi,\varphi), -2a).
\end{equation*}
Collecting the $\xi^3$-terms, we get
\begin{equation*} 
    \bigg (\frac{d}{d \tau} - A^{\odot \star}(\tau) \bigg) jh_3(\tau) = [C(\tau;\varphi(\tau),\varphi(\tau),\varphi(\tau)) + 3 B(\tau;\varphi(\tau),h_2(\tau))]r^{\odot \star} -6aj\dot{\varphi}(\tau) - 6c j\varphi(\tau).
\end{equation*}
Since $0$ is a Floquet exponent in the $T$-antiperiodic setting, the Fredholm solvability condition implies that
\begin{equation*} 
    c = \frac{1}{3} \int_0^{T} \langle [C(\tau;\varphi(\tau),\varphi(\tau),\varphi(\tau)) + 3 B(\tau;\varphi(\tau),h_2(\tau))]r^{\odot \star} - 6aA^{\odot \star}(\tau)j\varphi(\tau), \varphi^{\odot}(\tau) \rangle d\tau.
\end{equation*}
Thus, the critical coefficient $c$ in the critical normal form for the period-doubling bifurcation has been computed, and it follows from \cref{prop:pairing} that 
\begin{equation} \label{eq:normalformcoeff pd}
    c = \frac{1}{3} \langle p, C(\varphi,\varphi,\varphi) + 3 B(\varphi,h_2) -6a \Delta'(0)\dot{q} \rangle_T.
\end{equation}

\subsection{Neimark-Sacker bifurcation} \label{subsec:NS}
Suppose that \eqref{eq:DDE} has a $T$-periodic solution $\gamma$ with $\Lambda_0 = \{ 1, e^{\pm i\omega T} \},$ where all Floquet multipliers on the unit circle are simple and are in absence of strong resonances \eqref{eq:resonances}. Hence, there exist $T$-periodic (adjoint) eigenfunctions $\varphi$ and $\varphi^\odot$ satisfying
\begin{equation*}
    \bigg(\frac{d}{d \tau} - A(\tau) + i \omega \bigg)\varphi(\tau) = 0, \quad \bigg(\frac{d}{d \tau} + A^\star(\tau) - i \omega \bigg)\varphi^\odot(\tau) = 0, \quad \langle \varphi^\odot, \varphi \rangle_T = 1.
\end{equation*}
The (adjoint) eigenfunctions are given as in \cref{prop:pairing}, where $q$ is a null function of $\Delta(i \omega)$ and $p$ is a null function of $\Delta^\star(i \omega).$ To determine the eigenfunction $\varphi$ uniquely, we also impose that $\langle q^\star, q \rangle_T = 1$, where $q^\star$ is the Hermitian adjoint of $q$. The corresponding 3-dimensional center manifold $\mathcal{W}_{\loc}^{c}(\Gamma)$ at the Neimark-Sacker bifurcation can be parametrized locally near $\Gamma$ in $(\tau,\xi)$-coordinates as
\begin{align*}
    \mathcal{H}(\tau,\xi,\bar{\xi}) &= \gamma_\tau + \xi \varphi(\tau) + \bar{\xi}\bar{\varphi}(\tau)
     + \frac{1}{2}h_{20}(\tau)\xi^2 + h_{11}(\tau) \xi \bar{\xi} + \frac{1}{2} h_{02}(\tau)\bar{\xi}^2 \\
     & +\frac{1}{6}h_{30}(\tau)\xi^3 + \frac{1}{2} h_{21}(\tau) \xi^2 \bar{\xi} +  \frac{1}{2}h_{12}(\tau) \xi \bar{\xi}^2  + \frac{1}{6} h_{03}(\tau) \bar{\xi}^3 + \mathcal{O}(|\xi|^4), \quad \tau \in \mathbb{R}, \ \xi \in \mathbb{C},
\end{align*}
with $h_{ij}$ $T$-periodic and $h_{ij} = \bar{h}_{ji}$ so that $h_{11}$ is real. To derive the normal form coefficients $a$ and $c$, the homological equation reads
\begin{equation*}
        j\bigg(\frac{\partial \mathcal{H}(\tau,\xi,\bar{\xi})}{\partial \tau} \dot{\tau} + \frac{\partial \mathcal{H}(\tau,\xi,\bar{\xi})}{\partial \xi} \dot{\xi} + \frac{\partial \mathcal{H}(\tau,\xi,\bar{\xi})}{\partial \bar{\xi}} \dot{\bar{\xi}} \bigg) = A_{0}^{\odot \star}j(\mathcal{H}(\tau,\xi,\bar{\xi})) + G(\mathcal{H}(\tau,\xi,\bar{\xi})).
\end{equation*}
Collecting the $\xi^0$- and $\xi^1$- or $\bar{\xi}^1$-terms  in \eqref{eq:homologicalequation} yield
\begin{equation*}
    \frac{d}{d\tau} j({\gamma}_\tau) = A_0^{\odot \star} j(\gamma_\tau) + G(\gamma_\tau), \quad \bigg(\frac{d}{d \tau} - A^{\odot \star}(\tau) + i \omega \bigg)j\varphi(\tau) = 0,
\end{equation*}
or its complex-conjugate, which was already known. Collecting the $\xi^2$- or $\bar{\xi}^2$-terms lead to
\begin{equation*}
    \bigg(\frac{d}{d \tau} - A^{\odot \star}(\tau) + 2i\omega \bigg)jh_{20}(\tau) = B(\tau;\varphi(\tau),\varphi(\tau))r^{\odot \star}
\end{equation*}
or its complex-conjugate. This equation has a unique solution $h_{20}$ since $e^{2 i \omega T}$ is not a Floquet exponent because we are in the absence of strong resonances. Hence, it follows from \cref{cor:inverse} that
\begin{equation*}
    h_{20}(\tau)(\theta) = e^{2 i \omega\theta} \Delta^{-1}(2 i \omega)[B(\varphi,\varphi)](\tau + \theta).
\end{equation*}
The $|\xi|^2$-terms give
\begin{equation*}
    \bigg( \frac{d}{d \tau} - A^{\odot \star}(\tau) \bigg)j h_{11}(\tau) = B(\tau;\varphi(\tau),\bar{\varphi}(\tau))r^{\odot \star} - a j\dot{\gamma}_\tau.
\end{equation*}
Since $0$ is a Floquet exponent, the Fredholm solvability condition shows that 
\begin{equation*}
    \int_0^T \langle B(\tau;\varphi(\tau),\bar{\varphi}(\tau))r^{\odot \star} - aj \dot{\gamma}_\tau, \psi^{\odot}(\tau) \rangle d\tau = 0, \quad  \bigg(\frac{d}{d \tau} + A^\star(\tau) \bigg)\psi^{\odot}(\tau) = 0, \quad \langle \psi^\odot, \varphi \rangle_T = 1,
\end{equation*}
where $\varphi(\tau) = \dot{\gamma}_\tau = q_0(\tau + \cdot)$. The adjoint eigenfunction $\psi^\odot = (r,g)$ is given as in \cref{prop:pairing}, where $r$ is a $T$-periodic null function of $\Delta^\star(0)$ and satisfies $\langle \psi^\odot,\varphi \rangle_T = \langle r, \Delta'(0)q_0 \rangle_T = 1$. Note that $\langle \psi^{\odot},\varphi \rangle_T \neq 0$ due to \cref{prop:pairing}. This leads to the expression
\begin{equation*}
    a = \langle r,B(\varphi,\varphi) \rangle_T.
\end{equation*}
With $a$ defined in this way, let $h_{11}$ be the unique solution of \eqref{eq:PDxi^2terms} under the extra condition $\langle \varphi^{\odot}, h_{11} \rangle_T$, arising from the Fredholm solvability condition. It follows from \cref{cor:borderedinversemu} that this unique solution can be represented as
\begin{equation*}
    h_{11} = B_{0}^{\inv}(B(\varphi,\bar{\varphi}), -a).
\end{equation*}
Finally, the coefficients of the $\xi^2 \bar{\xi}$-terms provide the singular equation
\begin{align*}
    \bigg(\frac{d}{d \tau} - A^{\odot \star}(\tau) + {i \omega} \bigg)j h_{21}(\tau)  &= [B(\tau;h_{20}(\tau),\bar{\varphi}(\tau)) + 2B(\tau;h_{11}(\tau),\varphi(\tau)) \\
    & + C(\tau;\varphi(\tau),\varphi(\tau),\bar{\varphi}(\tau))]r^{\odot \star} - 2 a j\dot{\varphi}(\tau) -2d j\varphi(\tau),
\end{align*}
and so the Fredholm solvability condition implies that
\begin{align*} 
    d &= \frac{1}{2}\int_0^T \langle [B(\tau;h_{20}(\tau),\bar{\varphi}(\tau)) + 2B(\tau;h_{11}(\tau),\varphi(\tau))  + C(\tau;\varphi(\tau),\varphi(\tau),\bar{\varphi}(\tau))]r^{\odot \star} \nonumber\\ 
    & - 2aA^{\odot \star}(\tau)j\varphi(\tau),  \varphi^\odot(\tau) \rangle  d\tau + i a \omega.
\end{align*}
Thus, the critical coefficient $d$ in the critical normal form for the Neimark-Sacker bifurcation has been computed, and it follows from \cref{prop:pairing} that
\begin{equation} \label{eq:normalformcoeff NS}
    d = \frac{1}{2} \langle p, B(h_{20},\bar{\varphi}) + 2B(h_{11},\varphi)  + C(\varphi,\varphi,\bar{\varphi}) -2a\Delta'(i \omega)[i\omega q + \dot{q}]  \rangle_T + ia\omega.
\end{equation}

\section{Implementation issues} \label{sec:implementation}
The goal of this section is to discretize the boundary-value problems obtained in \cref{sec:periodic normalization} and solve them numerically to compute the critical normal form coefficients \eqref{eq:normalformcoeff fold}, \eqref{eq:normalformcoeff pd} and \eqref{eq:normalformcoeff NS}. Notice that the mentioned normal form coefficients depend on the higher order Fr\'echet derivatives of the operator $F \in C^{k+1}(X,\mathbb{R}^n)$, which is the right-hand side of \eqref{eq:DDE} for some sufficiently large $k \geq 1$. Unfortunately, it is in general not possible to provide a method for the computation of these higher order derivatives. Therefore, we will restrict our attention in this section towards a simple, but still rich subclass of classical DDEs as they are frequently used in applied literature. 
\subsection{The special case of discrete DDEs} \label{subsec: discrete DDEs}
We say that \eqref{eq:DDE} is a \emph{discrete} DDE or a DDE of \emph{point type} if there exist (discrete) delays $0 \eqqcolon \tau_0 < \tau_1 < \tau_2 < \dots < \tau_m \coloneqq h < \infty$ for some $m \in \mathbb{N}$ and a function $f : \mathbb{R}^{n \times (m+1)} \to \mathbb{R}^n$ such that
\begin{equation*}
    F(\varphi) = f(\varphi(-\tau_0),\dots,\varphi(-\tau_m)), \quad \forall \varphi \in X.
\end{equation*}
From this specification, it is clear that \eqref{eq:DDE} takes the form
\begin{equation} \label{eq: discrete DDE}
    \dot{x}(t) = f(x(t),x(t-\tau_1),\dots,x(t-\tau_m)), \quad t \geq 0.
\end{equation}
It is convenient to introduce the linear evaluation operator $\Xi : X \to \mathbb{R}^{n \times (m+1)}$ by $\Xi \varphi \coloneqq (\varphi(-\tau_0),\dots,\varphi(-\tau_m))$ since then $F = f \circ \Xi$. As already announced, the higher order Fr\'echet derivatives of the operator $F$ evaluated at  $\gamma_\tau$ are needed to compute the critical normal form coefficients. The map $L(\tau) = DF(\gamma_\tau) \in \mathcal{L}(X,\mathbb{R}^n)$ can be represented as
\begin{equation} \label{eq:derivative}
    DF(\gamma_\tau)\varphi = \sum_{j=0}^{m} M_j(\tau) \varphi(-\tau_j), \quad \forall \varphi \in X,
\end{equation}
where $M_j(\tau) \coloneqq D_{j+1} f(\gamma(\tau),\gamma(t-\tau_1),\dots,\gamma(t-\tau_m)) \in \mathbb{R}^{n \times n}$ is the $(j+1)$th partial derivative of $f$ evaluated at $\Xi \gamma_\tau$. Recall from \eqref{eq:L(t)varphi} that there is a unique $T$-periodic family $\{\zeta(\tau,\cdot)\}_{\tau \in \mathbb{R}}$ such that $\langle \zeta(\tau,\cdot), \varphi \rangle = \sum_{j=0}^{m} M_j(\tau) \varphi(-\tau_j)$ for all $\varphi \in X$. Hence, $\zeta(\tau,\cdot)$ has jump discontinuities $M_j(\tau)$ at the points $\tau_j$ for $j = 0,\dots,m$ and is constant otherwise. The characteristic operator \eqref{eq:Delta(z)q} and adjoint characteristic operator \eqref{eq:adjoint charac operator} take in this setting the form
\begin{align*}
    (\Delta(z)q)(\tau) &= \dot{q}(\tau) + z q(\tau) - \sum_{j=0}^m e^{-z \tau_j} M_j(\tau)  q(\tau - \tau_j), \quad \forall q \in C_T^1(\mathbb{R},\mathbb{C}^{n}), \\
    (\Delta^\star(z)p)(\tau) &= - \dot{p}(\tau) + z p(\tau) - \sum_{j=0}^m p(\tau + \tau_j)e^{-z \tau_j} M_j(\tau + \tau_j), \quad \forall p \in C_T^1(\mathbb{R},\mathbb{C}^{n \star}).
\end{align*}
To compute the critical normal form coefficients, recall that we need to compute the second and third order derivative from $F$. These higher order Fr\'echet derivatives of $F$ can be expressed in terms of the higher order derivatives of the map $f$. Indeed, if we set $\Phi \coloneqq \Xi \varphi$ and take $r \in \{1,\dots,k+1\}$, then 
\begin{equation} \label{D^r F}
    D^r F(\gamma_\tau)(\varphi_1,\dots,\varphi_r) = \sum_{j_1,\dots,j_r = 1}^{n} \sum_{k_1,\dots,k_r = 0}^m
    D_{{j_1k_1},\dots,{j_r k_r}}^r f_\gamma(\tau) \Phi_{1,j_1k_1}\cdots \Phi_{r,j_r k_r}.
\end{equation}

\subsection{Implementation of the normal form coefficients} \label{subsec:implementationnormalforms}
Numerical implementation of the formulas derived in the preceding sections require the evaluation of integrals of functions over $[0,T]$ and the solution to nonsingular linear BVPs with integral constraints. Such tasks can be carried out in (standard) continuation software for (discrete) DDEs such as \verb|DDE-BifTool| \cite{Engelborghs2002,Sieber2014}, \verb|Knuth| (\verb|PDDE-Cont|) \cite{Szalai2013}, \verb|BifurcationKit| \cite{Veltz2020} and \verb|PeriodicNormalizationDDEs| \cite{Bosschaert2024c}, see also \cite{Roose2007,Krauskopf2022} and the references therein for an overview of the techniques used in numerical bifurcation analysis for (discrete) DDEs. In these software packages, periodic solutions for discrete DDEs are computed via \emph{orthogonal collocation} with piecewise polynomials applied to properly formulated BVPs \cite{Engelborghs2001,Engelborghs2002a}. Numerical continuation of periodic solutions in one or two parameter families of discrete DDEs can also be handled by orthogonal collocation and leads to (large and sparse) linear systems, which can be solved by standard numerical methods. To detect codim 1 bifurcations of limit cycles, one can specify test functions that are based on computing Floquet multipliers \cite{Luzyanina2002,Breda2022a} or characteristic matrices \cite{Sieber2011,Szalai2006}. Numerical continuation of codim 1 bifurcations in two parameter families of discrete DDEs can be done by combining the mentioned methods. For the implementation of our formulas of the critical normal form coefficients, we only encounter $T$-(anti)periodic linear (inhomogeneous) DDEs/ADEs with integral constraints of the form
\begin{equation} \label{eq:BVPs} \tag{BVP}
    \begin{cases}
    \begin{aligned}
        \dot{y}(t) - g(y(t),y(t \pm \tau_1),\dots,y(t \pm \tau_m)) &= h(t,t \pm \tau_1,\dots,t \pm \tau_m), &&t \in [0,T], \\
        y(T + \theta) \pm y(\theta) &= 0, &&\theta \in [-\tau_m,0], \\
        \textstyle\int_0^T \langle s(t), y(t) \rangle dt  &= 0,
    \end{aligned}
    \end{cases}
\end{equation}
where $h,g$ and $s$ are given sufficiently smooth and $T$-(anti)periodic functions, and $g$ is linear. Numerical approximate solutions to \eqref{eq:BVPs} can also be computed using orthogonal collocation, see \hyperlink{mysupplement}{Supplement}.

Using this procedure, all terms inside the $\langle \cdot, \cdot \rangle_T$-pairing of our formulas have been computed numerically up to the desired accuracy. To numerically evaluate the critical normal form coefficients, one can use Gauss-Legendre quadrature to discretize the remaining integral over $[0,T]$, since this technique is also frequently used to numerically approximate \eqref{eq:BVPs}, see \hyperlink{mysupplement}{Supplement} for more information.

\begin{remark} \label{remark:efficient}
Numerical calculations can be simplified by avoiding the construction of adjoint functions through the adjoint characteristic operator and bypassing the need for quadrature formulas. This is achieved by applying the Fredholm alternative to the \emph{finite-dimensional} discretized systems derived from the linear systems generated by the homological equation. This more efficient numerical approach has been implemented in the Julia package, and the resulting formulas are provided in \hyperlink{mysupplement}{Supplement}. \hfill $\lozenge$
\end{remark}

\section{Examples} \label{sec:examples}
In this section, we validate the correctness and effectiveness of the derived and implemented normal form coefficients in two different models. We do this twofold. First, we compare our method with alternative indicators from normal form theory and numerical bifurcation analysis of equilibria. Second, we assess our method by comparing it to a pseudospectral approximation of DDEs, which leads to a (large) system of ODEs. A numerical bifurcation analysis is then conducted using the compatible \verb|Julia| package \verb|BifurcationKit|. Based on these analyses, we demonstrate that our implemented method is more efficient and reliable than the pseudospectral approach, see \cref{subsec:Chenciner}. The source code of the examples has been included into the software \verb|Julia| package \verb|PeriodicNormalizationDDEs| \cite{Bosschaert2024c}. It hopefully provides a good starting point when considering other models.

\subsection{LPC-connected GPD bifurcations in a lumped model of neocortex}
In \cite{Visser2012}, a two-node Hopfield lumped model of neural activity in neocortex with two delays
\begin{equation} \label{eq:neocortex}
\begin{dcases}
    \dot{x}_1(t) \hspace{-7pt}&= - x_1(t) - a  g(b  x_{1}(t-\tau_1)) + c  g(d  x_{2}(t - \tau_2)), \\
  \dot{x}_2(t) \hspace{-7pt}&= - x_2(t) - a  g(b  x_{2}(t- \tau_1)) + c  g(d  x_{1}(t- \tau_2)),
\end{dcases}
\end{equation}
is considered. Here, $x_i$ denotes the activity of node $i \in \{1,2\}$, $\tau_1$ is the time lag of feedback inhibition, $\tau_2$ is the delay of feedforward excitation, and the monotonically increasing function $g : \mathbb{R} \to \mathbb{R}$ represents the combined inhibitory and excitatory synaptic activation. The parameter $a > 0$ measures the strength of inhibitory feedback, while $c > 0$ measures the strength of the excitatory effect. The parameters $b > 0$ and $d > 0$ are saturation rates. To illustrate our results, we take $g(x) = (\tanh(x - 1) + \tanh(1)) \cosh(1)^2$, and fix $b = 2.0, d = 1.2, \tau_1 = 11.6$ and $ \tau_2 = 20.3$, so that $a$ and $c$ are the unfolding parameters. 

\begin{figure}[ht] \label{fig:LPCGPD}
    \includegraphics{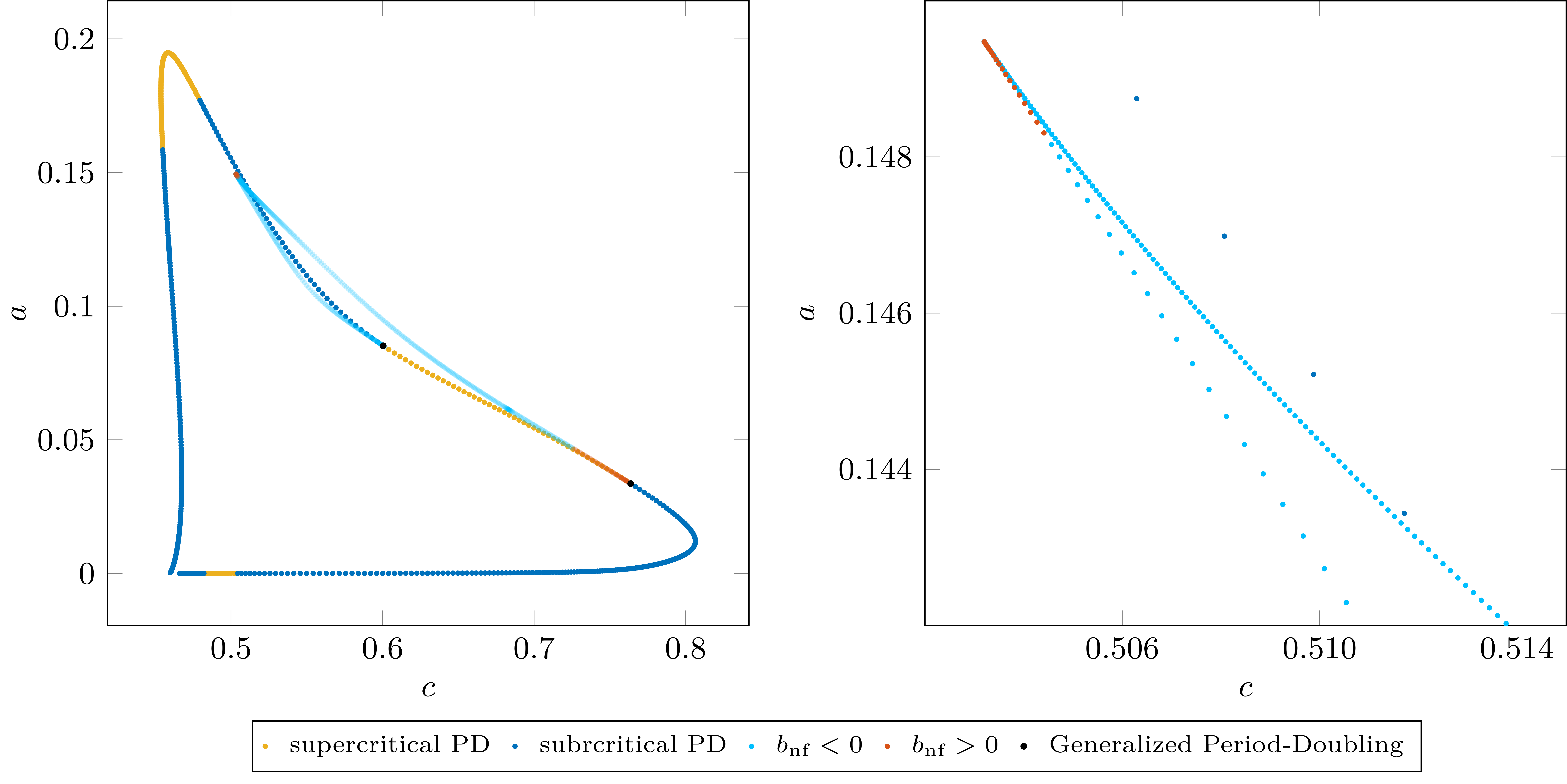}
    \caption{Bifurcation diagrams of \eqref{eq:neocortex} with unfolding parameters $(a,c)$. Left, sub- and supercritical period-doubling (PD) branches are shown, with two identified generalized period-doubling points connected by a limit point of cycles branch. Right, a close-up near the cusp bifurcation located on the LPC-branch.}
\end{figure}

In this setting, a cycle that undergoes a period-doubling bifurcation can be detected and continued in the $(a,c)$-plane, see the online \hyperlink{mysupplement}{Supplement} for its derivation. Computing the period doubling normal form coefficient along this branch yields two generalized period-doubling (GPD) bifurcations, while additional sign changes of this coefficient occur at $\pm \infty$. The validity of these GPD-points is confirmed by the tangential connection to a limit point of cycles (LPC) branch, see \cref{fig:LPCGPD}. Along this curve, a cusp bifurcation (CPC) is both visually identified and numerically confirmed, as the fold normal coefficient $b_{\text{nf}}$ vanishes at this point, while two additional sign changes occur at $\pm \infty$.

\subsection{Chenciner bifurcations in an active control system} \label{subsec:Chenciner}
In \cite{Peng2013}, an active control system to control the response of structures to internal or external excitation with time delay of the following form
\begin{equation} \label{eq:activecontrol}
    \begin{dcases}
        \dot{x}(t) = \tau y(t), \\
        \dot{y}(t) = \tau(-x(t) - g_u x(t-1) -2\zeta y(t) - g_vy(t-1) + \beta x^3(t-1)),
    \end{dcases}
\end{equation}
is considered. Here, $x$ is the displacement of the controlled system, $\tau > 0$ is the time delay represented in the relative displacement and velocity feedback loop, $g_u > 0$ and $g_v > 0$ are rescaled feedback strengths, $\zeta \in \mathbb{R}$ is a rescaled damping while $\beta > 0$ denotes the strength of the external excitation. To illustrate our results, we fix $g_u = 0.1, g_v = 0.52$ and $\beta = 0.1$, so that $\zeta$ and $\tau$ are the unfolding parameters. 

\begin{figure}[ht] \label{fig:activecontrol}
    \includegraphics{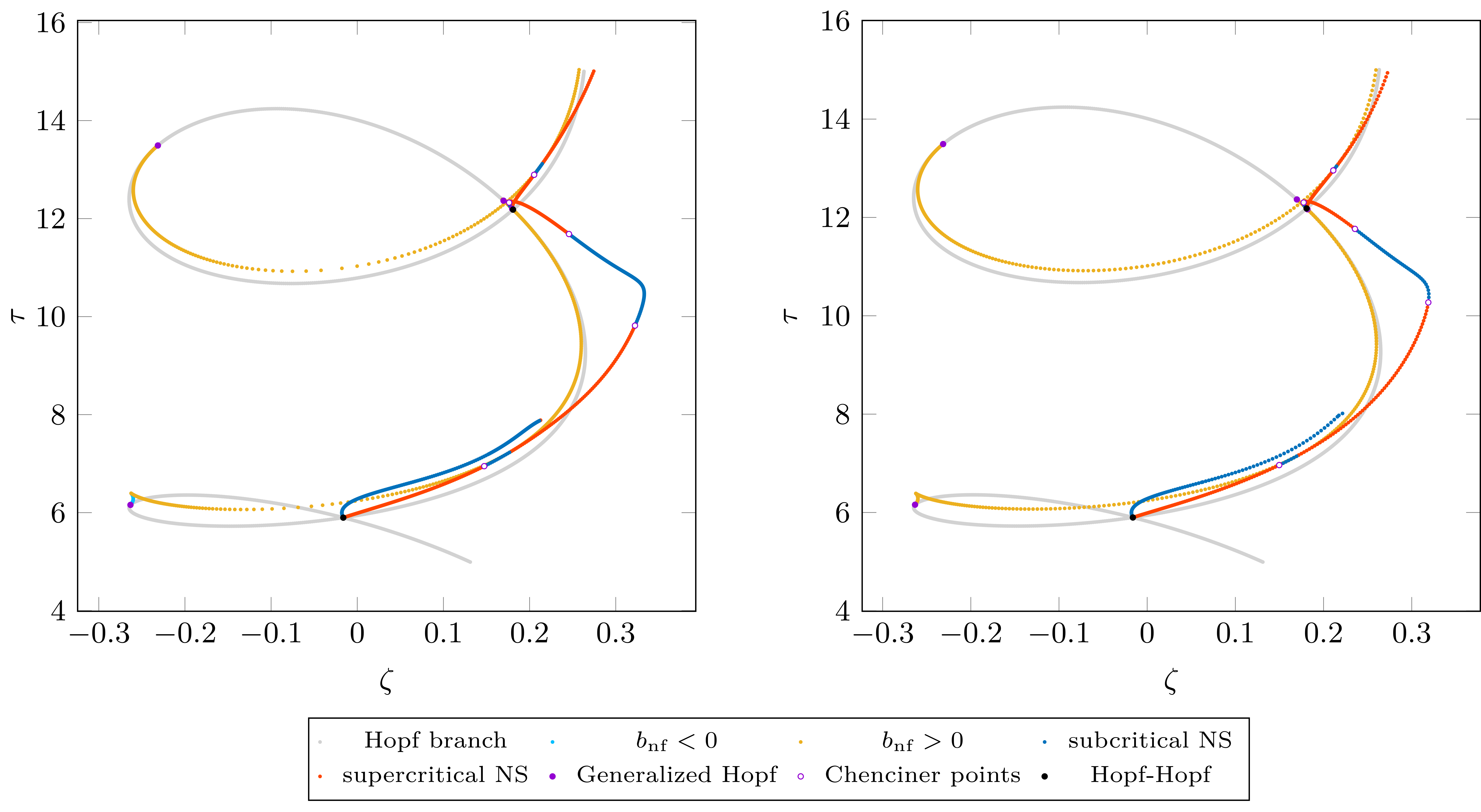}
    \caption{Bifurcation diagrams of \eqref{eq:activecontrol} with unfolding parameters $(\zeta,\tau)$. The left figure is obtained via periodic normalization in DDEs while right figure is obtained via a pseudospectral approximation.}
\end{figure}

In this setting, an extensive numerical bifurcation analysis for the equilibria of this system was performed in \cite[Section 8.3]{Bosschaert2020}. Specifically, a Hopf-Hopf bifurcation of type VI was identified at $\tau \approx 5.90$ from which the associated Neimark-Sacker and Hopf bifurcation curves were continued until they intersected with another Hopf-Hopf point at $\tau \approx 12.18$ of type I. A normal form analysis of the Hopf-Hopf points reveals that the Neimark-Sacker branches can be either sub- or supercritical, a result that aligns precisely with our findings. Indeed, we observe that the sign of the Neimark-Sacker normal form coefficient remains consistent along these branches for parameter values $(\zeta,\tau)$ sufficiently close to the identified Hopf-Hopf points. Computing the Neimark-Sacker normal form coefficient along these branches yields several Chenciner points. To further validate our findings, we constructed a pseudospectral approximation of \eqref{eq:activecontrol}, enabling the application of available numerical bifurcation analysis ODE-tools, see \cref{fig:activecontrol}. 

The pseudospectral approximation method has proven to be a powerful and effective tool for analyzing bifurcations of equilibria (see \cite{Breda2016, Ando2022, Wolff2021} for details). However, when studying bifurcations of limit cycles in DDEs, we found that our normalization-based approach, implemented in \verb|Julia| using continuation algorithms similar to those in \verb|DDE-BifTool| \cite{Engelborghs2002,Sieber2014}, offers better efficiency, greater reliability, and significant performance advantages. For example, obtaining the Neimark-Sacker curve connecting the two Hopf-Hopf points on the left-hand side of \cref{fig:activecontrol} required approximately 58 seconds. In contrast, computing the same curve using the pseudospectral method with 12 discretization points in \verb|BifurcationKit| \cite{Veltz2020} (right panel of \cref{fig:activecontrol}) required 71 minutes, roughly 74 times longer. These performance measurements were obtained on an HP ZBook Firefly 15.6 inch G8 Mobile Workstation PC.

Although the Neimark-Sacker curves in the two plots of \cref{fig:activecontrol} appear qualitatively similar, their precise parameter values differ, and the locations of the Chenciner points do not coincide. Nevertheless, depending on the application, such quantitative discrepancies may be of secondary importance. In these cases, a qualitatively accurate global representation of the bifurcation structure is sufficient for the intended analysis.

\section{Conclusion and outlook}
We derived and implemented explicit computational formulas for the critical normal form coefficients of all codim 1 bifurcations of limit cycles in (classical) DDEs. These coefficients can be used to verify if a codim 1 bifurcation is nondegenerate. In particular, the coefficients for the period-doubling and Neimark-Sacker bifurcation allow one to distinguish between sub- and supercritical bifurcations. The coefficients serve as test functions for detecting some codim 2 bifurcations of limit cycles. The effectiveness of our methods has been demonstrated through applications on various models. We chose to implement our formulas in the high-performance and open-source \verb|Julia| programming language.

The introduction of the characteristic operator $\Delta(z)$ and its adjoint $\Delta^\star(z)$ turned out to be very fruitful for the derivation and implementation of our formulas. This  operator allowed us to reduce spectral problems in $C_T(\mathbb{R},X)$ towards linear (inhomogeneous) differential equations in $C_T(\mathbb{R},\mathbb{C}^n)$. Therefore, we expect that there exists, similar to the results arising from the theory of characteristic matrices \cite[Theorem II.1.2]{Kaashoek1992} and \cite[Theorem IV.5.1]{Diekmann1995}, an equivalence (after extension) between $\Delta(z)$ and $zI - \mathcal{A}^{\odot \star}$. Such an equivalence would enable us to prove that a multiplicity theorem holds for periodic linear DDEs, see \cite[Theorem I.2.1]{Kaashoek1992} for its autonomous version. Moreover, as the characteristic matrix and the mentioned above equivalence can be introduced for a rather large class of closed linear operators generated by autonomous linear evolution equations \cite[Section I.3]{Kaashoek1992}, we expect that this construction can be generalized for a broad class of closable linear operators generated by periodic linear evolution equations.

The underlying normalization technique can also be used to derive critical normal form coefficients for all codim 2 bifurcations of limit cycles by lifting the results from \cite{Witte2014,Witte2013} towards the setting of classical DDEs. When this task is accomplished, the next obvious challenge is the problem of switching to secondary cycle bifurcations rooted at codim 2 points of bifurcations of cycles for ODEs and DDEs. In the finite-dimensional ODE-setting, a generalization of the periodic center manifold theorem \cite{LentjesCMODE}, the periodic normal forms \cite{Iooss1988,Iooss1999}, and the periodic normalization method \cite{Kuznetsov2005,Witte2013,Witte2014} to parameter-dependent systems in the spirit of \cite{Govaerts2007,Kuznetsov2008,Bosschaert2024} is required. When this task has been completed for ODEs, these results could be lifted towards (classical) DDEs by using the sun-star calculus framework in combination with the results from this paper, see also \cref{remark:parameterdependentCMT}. Switching to codim 1 curves rooted at codim 2 points has been accomplished recently for bifurcations of equilibria, see \cite{Bosschaert2020} for switching towards codim 1 curves of nonhyperbolic cycles rooted at certain codim 2 points and \cite{Bosschaert2024a} for switching to the unique homoclinic curve rooted at the codim 2 Bogdanov-Takens bifurcation point.

We restricted ourselves in this paper to the setting of classical DDEs. However, as our results are written in the general context of sun-star calculus, we believe that our method extends to a broader class of delay equations, such as for example renewal equations \cite{Breda2020,Diekmann2008,Diekmann1995}. If one is interested in bifurcations of limit cycles for systems consisting of infinite delay \cite{Diekmann2012,Hino1991,Liu2020} or abstract DDEs \cite{Janssens2019,Janssens2020,Spek2022a,Spek2020,Spek2024}, that describe for example neural fields, it is in known that $\odot$-reflexivity is in general lost \cite[Theorem 12]{Spek2020}. As a consequence, the center manifold theorem for nonhyperbolic cycles from \cite{Article1} does not directly apply. However, we believe that this technical difficulty can be resolved by
employing similar techniques as in \cite{Janssens2020}. We believe that these methods can also be applied to demonstrate the existence of periodic normal forms from \cite{Article2}, and subsequently, the described normalization technique can be used to determine the (critical) normal form coefficients of bifurcating cycles.

\section*{Acknowledgments}
The authors would like to thank Prof. Odo Diekmann (Utrecht University), Dr. Hil Meijer (University of Twente), Prof. Peter de Maesschalck (UHasselt) and Stein Meereboer (Radboud University) for helpful discussions and suggestions.

\appendix
\section{Resolvent representation of $\mathcal{A}^\star$}

In the proof of \cref{cor:spectra}, we used the following result, which is a dual version of \cref{prop:solvability}.
\begin{proposition} \label{prop:resolventadjoint}
If $z \in \mathbb{C}$ is such that $\Delta^\star(z)$ is invertible, then the resolvent of $\mathcal{A}^\star$ at $z$ is given by $(c,g) = (zI- \mathcal{A}^\star)^{-1}\psi^\star$, where
\begin{align}
\begin{split} \label{eq:gresolvent}
    c &= \Delta^\star(z)^{-1} \bigg( \tau \mapsto \psi^\star(\tau)(0^+) + \int_0^h e^{-z s} k(\tau+s)(s) ds \bigg), \\
    g(\tau)(\theta) &= \int_\theta^h e^{\sigma(\theta-s)}c(\tau+s-\theta) d_2 \zeta(\tau+s-\theta,s)  + \int_\theta^h e^{\sigma(\theta-s)} k(\tau+s-\theta)(s) ds.
\end{split}
\end{align}
\end{proposition}
\begin{proof}
Let us first assume that $\psi^\star =(d,k) \in C_T^1(\mathbb{R},X^\star)$ and set $\varphi^\star = (c,g)$. Then, $(zI - \mathcal{A}^\star)\varphi^\star = \psi^{\star}$ is equivalent to solving the first order linear inhomogeneous PDE
\begin{equation*}
    \bigg( \frac{\partial}{\partial \tau} + \frac{\partial}{\partial \theta} - z \bigg) \varphi^\star(\tau)(\theta) = - c(\tau) \zeta(\tau,\theta) - \psi^\star(\tau)(\theta).
\end{equation*}
Along the same lines of the proof of \cref{thm:adjoint eigenfunctions2}, we obtain $g$ as in \eqref{eq:gresolvent}. Similar to \eqref{eq:pidot} at $\theta = 0$, we get
\begin{equation*}
    \dot{c}(\tau) = z c(\tau) - \psi^{\star}(\tau)(0^+) - \int_0^h e^{-z s}k(\tau+s)(s) ds - \int_0^h e^{-zs} c(\tau+s) d_2 \zeta(\tau+s,s),
\end{equation*}
and recalling \eqref{eq:adjoint charac operator} yields \eqref{eq:gresolvent}. Note that the solution $c$ to the periodic linear advance differential equation \eqref{eq:gresolvent} is at least in $C_T^1(\mathbb{R},\mathbb{C}^n)$. As $\psi^\star(\cdot)(0^+)$ and $\zeta$ are (at least) $C^1$-smooth in the first component, it follows from \eqref{eq:gresolvent} that $g$ is $C^1$-smooth in the first component and thus $(c,g) \in C_T^1(\mathbb{R},X^\star)$. Clearly, $g(\tau) \in \NBV([0,h],\mathbb{C}^{n \star})$ because $\zeta(\tau,\cdot)$ and $k(\tau)$ are of bounded variation. Hence, $(c,g) \in \mathcal{D}(\mathcal{A}^\star)$ as $g(\tau)(h) = 0$ and so $zI - \mathcal{A}^{\star}$ has dense range.

It remains to show that the resolvent of $\mathcal{A}^\star$ at $z$ is bounded. Let $\psi^\star \in \mathcal{R}(zI - \mathcal{A}^\star)$ be given and note that $\| c \|_\infty \leq G_z \| \Delta^\star(z)^{-1} \| \|\psi^\star\|_\infty$ with $G_z \coloneqq 1 + h \max \{1,e^{-h \Re(z)} \}$, where $\| \Delta^\star(z)^{-1} \| < \infty$ as $\Delta^\star(z)^{-1}$ is bounded since $\Delta^\star(z)$ is closed, recall \cref{lemma:adjoint charac}. Then, $\|(zI - \mathcal{A}^\star)^{-1}\psi^\star \|_\infty \leq G_z[\| \Delta^\star(z)^{-1} \| + h (TV_\zeta G_z + h^2)] \|\psi^\star\|_\infty$, which proves the result.
\end{proof}
\end{sloppypar}

\phantomsection\label{mysupplement}
\begin{mytitle}
  \title{\textsc{supplementary materials for:}\\ \TheTitle}
  \maketitle
\end{mytitle}
\thispagestyle{plain}

\ResetCounters

\begin{sloppypar}
In this supplement, we provide an overview of the (efficient) numerical implementation of the derived normal form coefficients. First, we present the collocation method for numerically solving periodic linear delay differential equations as this technique will be frequently used in our implementation. Second, we provide a detailed explanation of the numerical implementation for each bifurcation individually. The actual software implementation can be found in the freely available software \verb|Julia| package \verb|PeriodicNormalizationDDEs| \cite{Bosschaert2024c}. In the same reference, a step-by-step interactive walkthrough of the examples from \cref{sec:examples} is provided. Additionally, our code is applied to various other models, showcasing its effectiveness, robustness and broad applicability.

\section{Numerical implementation of the normal form coefficients}
\subsection{Collocation method for periodic linear DDEs} \label{supp:collocationsubsec}
In this section, we summarize the collocation method for numerically approximating solutions to \eqref{eq:BVPs}, see \cite{Engelborghs2001,Engelborghs2002a} for more information on this topic. Let $\Pi$ be a \emph{mesh}, i.e a collection of \emph{mesh points} $0 = t_0 < t_1 < \dots < t_L = T$. This mesh partitions the domain $[0,T]$ into $L$ subintervals of the form $[t_i,t_{i+1}]$ for $i=0,\dots,L-1$. Let $\mathbb{P}_M$ denote the vector space of all $\mathbb{R}^n$-valued polynomials of degree at most $M$. The aim is to construct an approximate solution $y$ to \eqref{eq:BVPs} on $[0,T]$ by an element of the following $n(LM + 1)$-dimensional subspace
\begin{equation*}
    S_M(\Pi) \coloneqq \{ u \in C([0,T],\mathbb{R}^n) : u|_{[t_i,t_{i+1}]} \in \mathbb{P}_M \mbox{ for all } i=0,\dots,L-1 \} \subseteq C([0,T],\mathbb{R}^n).
\end{equation*}
Note that $u \in S_M(\Pi)$ is piecewise $C^1$-smooth since at the mesh points we only require continuity instead of differentiability. Within each subinterval $[t_i,t_{i+1}]$, we introduce $M-1$ internal \emph{basis points} $t_{i,j} \coloneqq t_i + \frac{j}{M}(t_{i+1}-t_i)$ for $j=1,\dots,M-1$ and $M$ \emph{collocation points} $c_{i,j} \coloneqq t_i + c_j(t_{i+1}-t_i)$ for $j=1,\dots,M$, where $0\leq c_1 < \dots < c_M \leq 1$ are $M$th degree roots of the Gauss-Legendre polynomial translated to $[0,1]$.

The idea of the collocation method for approximating a solution to \eqref{eq:BVPs} is to find a so-called \emph{collocation solution} $y : [-\tau_m,T] \to \mathbb{R}$ so that $y|_{[0,T]} \in S_M(\Pi)$, $y|_{[t_i,t_{i+1}]}$ satisfies \eqref{eq:BVPs} on the collocation points and $y$ satisfies the additional boundary and phase condition. The collocation solution $y$ can then be represented on $[0,T]$ as
\begin{equation} \label{eq:ycollocation}
    y(t) = \sum_{j=0}^M y_{i,j} \ell_{i,j}(t), \quad \forall t \in [t_{i},t_{i+1}], \ i=0,\dots,L-1,
\end{equation}
where $\ell_{i,j}$ are the Lagrange basis polynomials corresponding to $t_{i,0},\dots,t_{i,M}$ and $y_{i,j}$ are the unknowns. To determine $y_{i,j}$, recall that $y$ must satisfy the $n(LM+1)+1$ equations
\begin{equation} \label{eq:yij}
\begin{cases}
\begin{aligned}
    \dot{y}(c_{i, j})- g(y(c_{i, j}), \ldots, y((c_{i, j} \pm \tau_m) \bmod T)) &= h(c_{i, j}, \ldots, (c_{i, j} \pm \tau_m) \bmod T), \\
    y_{0,0} \pm y_{L,0} &=0, \\
    \textstyle\sum_{i=0}^{L-1} \sum_{j=0}^{M-1} \sigma_{i,j} \langle y_{i,j},s_{i,j}\rangle + \sigma_{L,0} \langle y_{L,0},s_{L,0} \rangle &= 0,
\end{aligned}
\end{cases}
\end{equation}
with the convection $y_{L,0} \coloneqq y_{L-1,M}$ and $s_{i,j} \coloneqq s(t_{i,j})$. If $s$ is not defined on the basis points but computed by collocation, we use the interpolation from \eqref{eq:ycollocation}. Here, in the first system of equations the function $y$ must be substituted by \eqref{eq:ycollocation}, the second equation guarantees $T$-periodicity and the last equation reflects on Gauss-Legendre quadrature approximation of the phase condition. The integration weight $\sigma_{i,j}$ at $t_{i,j}$ is given by 
\begin{equation*}
    \sigma_{i,j} \coloneqq 
    \begin{dcases}
        w_1 t_1, \quad & i=0,\dots,L-1, \ j=0, \\
        w_{j+1}t_{i+1}, \quad & i=0,\dots,L-1, \ j=1,\dots,M-1, \\
        w_{M+1} t_{i+1} + w_1 t_{i+2}, \ & i=0,\dots,L-2, \ j=M, \\
        w_{M+1} t_L, \ & i=L-1, \quad \quad \quad \hspace{0.5pt} j=M,
    \end{dcases}
\end{equation*}
where $w_{j+1}$ denotes the Lagrange quadrature coefficient. Moreover, we assume that the function $s$ is known at the basis points and set. Moreover, the same Gauss-Legendre quadrature approximation formula from \eqref{eq:yij} can also be used to evaluate numerically the formula for the normal form coefficients.

Let $y$ be a solution of \eqref{eq:BVPs} and let $y_{LM}$ be a solution of \eqref{eq:yij}. With this choice of collocation points, the convergence of $y_{LM}$ towards $y$ is guaranteed of order $M$, i.e. $\|y - y_{LM} \|_\infty = \mathcal{O}(H^{M})$, where $H \coloneqq \max_{i=0,\dots,L-1} |t_{i+1}-t_i|$. Moreover, if $\Pi$ is a so-called \emph{constrained mesh} in the sense of \cite[Section 4]{Engelborghs2001}, then \emph{superconvergence} at the mesh points occurs, i.e. $\max_{i=0,\dots,L}|y(t_i) - y_{LM}(t_i)| = \mathcal{O}(H^{2M})$. 

The linear system \eqref{eq:yij} is represented by a (large and sparse) matrix equation, where this matrix has band structure; further details can be found in \cite[Section 5.2]{Engelborghs2001}. Since \eqref{eq:yij} is a linear problem, solving it does not require initial guesses or the computation of Jacobian matrices, as would be necessary for a Newton-Raphson-like method. However, such nonlinear solvers are essential for tasks like computing an approximation of the $T$-periodic solution $\gamma$ of \eqref{eq: discrete DDE}. This can be done by performing a collocation-like discretization on the (nonlinear) system
\begin{equation} \label{eq:collocationgamma}
    \begin{cases}
        \begin{aligned}
            \dot{\gamma}(t) - f(\gamma(t),\gamma(t-\tau_1),\dots,\gamma(t-\tau_m)) &= 0, && t \in [0,T],\\
            \gamma(T + \theta) - \gamma(\theta) &= 0, && \theta \in [-\tau_m,0], \\
            p(\gamma,T)  &= 0,
        \end{aligned}
    \end{cases}
\end{equation}
where $p(\gamma,T)$ is a phase condition. For example, a well-known classical integral phase condition \cite{DOEDEL1991} is 
\begin{equation*}
    p(\gamma,T) = \int_0^T \langle \dot{\gamma}_0(s), \dot{\gamma}_0(s) - \gamma(s) \rangle ds,
\end{equation*}
where $\gamma_0$ is a reference solution.

\subsection{Efficient implementation of the normal form coefficients}
As discussed in \cref{subsec:implementationnormalforms}, the normal form coefficients can be computed using Gauss-Legendre quadrature, where such an integral discretization is illustrated in the last equation of \eqref{eq:yij}. However, as noted in \cref{remark:efficient}, a more efficient method exists for calculating these coefficients. We provide a detailed explanation of this implementation for each bifurcation below. We assume that a $T$-periodic solution $\gamma$ to \eqref{eq: discrete DDE} has been approximated through collocation by solving \eqref{eq:collocationgamma}. The approximate $T$-periodic solution of $\gamma$ can then be evaluated at $t \in [0,T]$ using Lagrange polynomials as illustrated in \eqref{eq:ycollocation}.

\subsubsection{Fold bifurcation}
Recall from \cref{subsec:fold} that the linear term is given by
\begin{equation*}
    \bigg(\frac{d}{d \tau} - A^{\odot \star}(\tau) \bigg)j\varphi_1(\tau) = -j\dot{\gamma}_\tau.
\end{equation*}
According to \cref{prop:pairing2}, the solution of this equation is given by $\varphi_1(\tau)(\theta) = \theta q_0(\tau+\theta) + q_1(\tau+\theta)$ for all $\tau \in \mathbb{R}$ and $\theta \in [-\tau_m,0]$, where $\{q_0,q_1\}$ is a $T$-periodic Jordan chain of $\Delta(0)$. Let us recall that $q_0$ is already computed since $q_0(\tau) = \dot{\gamma}(\tau)$ for all $\tau \in \mathbb{R}$ and $q_1$ is a solution to the BVP
\begin{equation*}
    \begin{cases}
    \begin{aligned}
    \dot{q}_1(\tau) - \textstyle\sum_{j=0}^m M_j(\tau)q_1(\tau-\tau_j) &= -[q_0(\tau) + \textstyle\sum_{j=1}^m \tau_j M_j(\tau)q_0(\tau-\tau_j)], &&\tau \in [0,T], \\
    q_1(T + \theta) - q_1(\theta) &= 0, &&\theta \in [-\tau_m,0], \\
    \end{aligned}
    \end{cases}
\end{equation*}
for which an approximate solution can be obtained by orthogonal collocation. In contrast to the method provided in \cref{subsec:fold}, we do not require to uniquely define $\varphi_1$ by normalizing $q_1$ against $q_0$ since we are only interested in a solution. This is not a problem because the sign of the normal form coefficient is independent of the choice of (generalized) eigenfunction.

Recall that the second order term yields
\begin{equation} \label{supp:fold2}
    \bigg(\frac{d}{d \tau} - A^{\odot \star}(\tau) \bigg)jh_2(\tau) = B(\tau;\varphi_1(\tau),\varphi_1(\tau))r^{\odot \star} - 2aj\dot{\gamma}_\tau - 2j\dot{\varphi}_1(\tau) - 2bj\varphi_1(\tau).
\end{equation}
Rather than relying on \eqref{eq:FSC}, we take a different approach to bypass the need for calculating the adjoint eigenfunction. Recall from \eqref{eq:expandH} that $h_2$ exists and therefore \eqref{supp:fold2} has a solution, meaning that \eqref{eq:solvability} is consistent. Since $0$ is not a simple Floquet exponent (it has algebraic multiplicity $2$), we can not use \cref{prop:borderedinversemu} (or \cref{cor:borderedinversemu}) directly. Instead, if we write out \eqref{supp:fold2} in two components, the second component yields, as in the proof of \cref{prop:solvability}, a first order linear inhomogeneous PDE that has the solution
\begin{align*}
    h_2(\tau)(\theta) &= h_{20}(\tau+\theta) - 2 \int_{\theta}^0 a \dot{\gamma}(\tau+\theta) + \dot{\varphi}_1(\tau+\theta-s)(s) + b \varphi_1(\tau+\theta-s)(s) ds \\
    &= h_{20}(\tau+\theta) + 2a\theta q_0(\tau+\theta) + \theta^2 \dot{q}_0(\tau + \theta) + 2\theta\dot{q}_1(\tau + \theta) + b(\theta^2 q_0(\tau + \theta) + 2\theta q_1(\tau + \theta)).
\end{align*}
Using this representation, the first component of \eqref{supp:fold2} yields the BVP
\begin{equation*}
   \begin{cases}
    \begin{aligned}
    \dot{h}_{20}(\tau) - \textstyle\sum_{j=0}^m M_j(\tau)h_{20}(\tau-\tau_j) &= B(\tau;\varphi_1(\tau),\varphi_1(\tau)) - 2a[q_0(\tau) + \textstyle\sum_{j=1}^m \tau_j M_j(\tau) q_0(\tau-\tau_j)]\\
    & -[2\dot{q}_1(\tau) + \textstyle\sum_{j=1}^m \tau_j M_j(\tau)(2\dot{q}_1(\tau-\tau_j)-\tau_j \dot{q}_0(\tau-\tau_j))] \\
    & -b[2{q}_1(\tau) + \textstyle\sum_{j=1}^m \tau_j M_j(\tau) (2{q}_1(\tau-\tau_j)-\tau_j {q}_0(\tau-\tau_j))],\\
    h_{20}(T + \theta) - h_{20}(\theta) &= 0,
    \end{aligned}    
    \end{cases}
\end{equation*}
where $\tau \in [0,T]$ and $\theta \in [-\tau_m,0]$. This system is of the form \eqref{eq:BVPs} and therefore the collocation method from \cref{supp:collocationsubsec} applies. Hence, we obtain, as in \eqref{eq:yij}, a system of $n(LM+1)$ linear equations, but without phase condition. Let $\mathbf{p} \in \mathbb{R}^{nLM}$ be a null-vector of discretization of the first $nLM$ equations of this discretization. The Fredholm alternative for finite-dimensional linear systems tells us that
\begin{equation*}
    b_{LM} \coloneqq \frac{\langle \mathbf{p}, B(c_{i,j};\varphi_1(c_{i,j}),\varphi_1(c_{i,j})) -  2\dot{q}_1(c_{i,j}) \hspace{-1pt} - \hspace{-1pt} \sum_{l=1}^m \tau_l M_l(c_{i,j}) (2\dot{q}_1(c_{i,j}-\tau_l)-\tau_l \dot{q}_0(c_{i,j}-\tau_l)) \rangle}{ \langle \mathbf{p}, 2{q}_1(c_{i,j}) + \sum_{l=1}^m \tau_l M_l(c_{i,j})(2{q}_1(c_{i,j}-\tau_l)-\tau_l {q}_0(c_{i,j}-\tau_l) \rangle}
\end{equation*}
is an approximation of the normal form coefficient $b$ from \eqref{eq:normalformcoeff fold}. Here, the second component of the variables appearing in the inner product $\langle \cdot , \cdot \rangle$ on $\mathbb{R}^{nLM}$ must have the same indexing scheme as $\mathbf{p}$.

Recall from \cref{supp:collocationsubsec} that convergence of (order $M$) is guaranteed for the orthogonal collocation method applied to the system above. Therefore, convergence of $b_{LM}$ towards $b$ is also guaranteed.

\subsubsection{Period-doubling bifurcation}
Recall from \cref{subsec:PD} that the linear term is given by
\begin{equation*}
    \bigg(\frac{d}{d \tau} - A^{\odot \star}(\tau) \bigg)j\varphi(\tau) = 0.
\end{equation*}
According to \cref{prop:pairing}, the solution reads $\varphi(\tau)(\theta) = q(\tau+\theta)$ for all $\tau \in \mathbb{R}$ and $\theta \in [-\tau_m,0]$, where $q$ is a $T$-antiperiodic null function of $\Delta(0)$:
\begin{equation*}
    \begin{cases}
    \begin{aligned}
    \dot{q}(\tau) - \textstyle\sum_{j=0}^m M_j(\tau)q(\tau-\tau_j) &= 0, &&\tau \in [0,T], \\
    q(T + \theta) + q(\theta) &= 0, &&\theta \in [-\tau_m,0].
    \end{aligned}
    \end{cases}
\end{equation*}
Again, an approximate solution of this system can be obtained by the orthogonal collocation method. 

Recall that the second order term is given by
\begin{equation} \label{eq:pd1}
    \bigg(\frac{d}{d \tau} - A^{\odot \star}(\tau) \bigg) jh_2(\tau) = B(\tau;\varphi(\tau),\varphi(\tau))r^{\odot \star} - 2aj\dot{\gamma}_\tau.
\end{equation}
We use the same strategy as performed on \eqref{supp:fold2} and therefore avoid using \cref{prop:FSC} and \cref{cor:borderedinversemu}. Note that the solution of the second component of \eqref{eq:pd1} is given by
\begin{equation} \label{eq:h2supppd}
    h_2(\tau)(\theta) = h_{20}(\tau+\theta) + 2 a \theta q_0(\tau+\theta),
\end{equation}
where we recall that $q_0 = \dot{\gamma}$. Using this representation, the first component of \eqref{eq:pd1} yields the BVP
\begin{equation*}
    \begin{cases}
    \begin{aligned}
        \dot{h}_{20}(\tau) - \textstyle\sum_{j=0}^{m} M_j(\tau) h_{20}(\tau-\tau_j) &= B(\tau; \varphi(\tau),\varphi(\tau)) \\
        &-2a[q_0(\tau) + \textstyle\sum_{j=1}^m \tau_j M_j(\tau) q_0(\tau-\tau_j)], &&\tau \in [0,T], \\
        h_{20}(T + \theta) - h_{20}(\theta) &= 0, &&\theta \in [-\tau_m,0].
    \end{aligned}
    \end{cases}
\end{equation*}
Consider the $n(LM+1)$ linear equations obtained from the orthogonal collocation method applied on this BVP. Let $\mathbf{r} \in \mathbb{R}^{nLM}$ be a null-vector of the first $nLM$ equations of this discretization. The Fredholm alternative for finite-dimensional linear systems tells us that
\begin{equation} \label{eq:aLM}
    a_{LM} \coloneqq \frac{\langle \mathbf{r}, B(c_{i,j}; \varphi(c_{i,j}), \varphi(c_{i,j}) \rangle}{ 2\langle \mathbf{r}, q_0(c_{i,j}) + \sum_{l=1}^m \tau_l M_l(c_{i,j}) q_0(c_{i,j}-\tau_l) \rangle}
\end{equation}
is an approximation of the normal form coefficient $a$ from \eqref{eq:pdcoeffa}. Here, the second component of the variables appearing in the inner product $\langle \cdot , \cdot \rangle$ on $\mathbb{R}^{nLM}$ must have the same indexing scheme as $\mathbf{r}$. With $a_{LM}$ defined in this way, we can find an approximate solution for $h_{20}$ by solving the above BVP and thus an approximate solution $h_2$ from \eqref{eq:pd1} is computed.

Collecting the third order terms, we obtain
\begin{equation} \label{eq:h3suppd}
    \bigg (\frac{d}{d \tau} - A^{\odot \star}(\tau) \bigg) jh_3(\tau) = [C(\tau;\varphi(\tau),\varphi(\tau),\varphi(\tau)) + 3 B(\tau;\varphi(\tau),h_2(\tau))]r^{\odot \star} -6aj\dot{\varphi}(\tau) - 6c j\varphi(\tau).
\end{equation}
To avoid the computation of the adjoint eigenfunction $\varphi^\odot$ from \cref{subsec:PD}, we proceed as described above. The solution of the second component of \eqref{eq:h3suppd} is given by
\begin{equation*}
    h_{3}(\tau)(\theta) = h_{30}(\tau+\theta) + 6 \theta(a \dot{q}(\tau+\theta) + c q(\tau+\theta)).
\end{equation*}
Using this representation, the first component of \eqref{eq:h3suppd} yields the BVP
\begin{equation*}
    \begin{cases}
    \begin{aligned}
        \dot{h}_{30}(\tau) - \textstyle\sum_{j=0}^{m} M_j(\tau) h_{30}(\tau-\tau_j) &= C(\tau;\varphi(\tau), \varphi(\tau), \varphi(\tau)) + 3B(\tau; \varphi(\tau),\varphi(\tau)) \\
        &-6a[\dot{q}(\tau) + \textstyle\sum_{j=1}^m \tau_j M_j(\tau) \dot{q}(\tau-\tau_j)] \\
        &-6c[q(\tau) + \textstyle\sum_{j=1}^m \tau_j M_j(\tau) q(\tau-\tau_j)], &&\tau \in [0,T], \\
        h_{30}(T + \theta) - h_{30}(\theta) &= 0, &&\theta \in [-\tau_m,0].
    \end{aligned}
    \end{cases}
\end{equation*}
Consider the $n(LM+1)$ linear equations derived from applying the orthogonal collocation method to this BVP. Let $\mathbf{p} \in \mathbb{R}^{nLM}$ be a null-vector of the first $nLM$ equations of this discretization. The Fredholm alternative for finite-dimensional linear systems tells us that
\begin{align*}
    c_{LM} &\coloneqq \frac{\langle \mathbf{p}, C(c_{i,j};\varphi(c_{i,j}), \varphi(c_{i,j}), \varphi(c_{i,j})) + 3B(c_{i,j}; \varphi(c_{i,j}),\varphi(c_{i,j})) \rangle}{ 6\langle \mathbf{p}, q(c_{i,j}) + \sum_{l=1}^m \tau_l M_l(c_{i,j}) q(c_{i,j}-\tau_l) \rangle} \\
    & -a_{LM} \frac{\langle \mathbf{p}, \dot{q}(c_{i,j}) + \sum_{l=1}^m \tau_l M_l(c_{i,j}) \dot{q}(c_{i,j} - \tau_l)\rangle}{\langle \mathbf{p}, q(c_{i,j}) + \sum_{l=1}^m \tau_l M_l(c_{i,j}) q(c_{i,j}-\tau_l) \rangle}
\end{align*}
is an approximation of the normal form coefficient $c$ from \eqref{eq:normalformcoeff pd}. Here, the second component of the variables appearing in the inner product $\langle \cdot , \cdot \rangle$ on $\mathbb{R}^{nLM}$ must have the same indexing scheme as $\mathbf{p}$.

Recall from \cref{supp:collocationsubsec} that convergence of (order $M$) is guaranteed for the orthogonal collocation method applied to the system above. Therefore, convergence of $c_{LM}$ towards $c$ is also guaranteed.

\subsubsection{Neimark-Sacker bifurcation}
Recall from \cref{subsec:NS} that the linear term is given by
\begin{equation*}
    \bigg(\frac{d}{d \tau} - A^{\odot \star}(\tau) + i \omega \bigg)j\varphi(\tau) = 0.
\end{equation*}
According to \cref{prop:pairing}, the solution reads $\varphi(\tau)(\theta) = e^{i \omega \theta}q(\tau+\theta)$ for all $\tau \in \mathbb{R}$ and $\theta \in [-\tau_m,0]$, where $q$ is a $T$-periodic null function of $\Delta(i \omega)$:
\begin{equation*}
    \begin{cases}
    \begin{aligned}
    \dot{q}(\tau) + i \omega q(\tau) - \textstyle\sum_{j=0}^m e^{- i \omega \tau_j} M_j(\tau)q(\tau-\tau_j) &= 0, \quad \tau \in [0,T], \\
    q(T + \theta) + q(\theta) &= 0, \quad  \theta \in [-\tau_m,0]. \\
    \end{aligned}
    \end{cases}
\end{equation*}
Considering separately the real and imaginary part of this BVP, an approximate solution of this system can be obtained by the orthogonal collocation method. The coefficient $h_{20}$ can be obtained as described in \cref{subsec:NS}. Note that $h_{11}$ satisfies \eqref{eq:pd1} and therefore is given as in \eqref{eq:h2supppd}, where an approximation of the normal form coefficient $a$ is given by $a_{LM}$ specified in \eqref{eq:aLM}.

Collecting the $\xi^2 \bar{\xi}$-terms, we obtain
\begin{align}
\begin{split} \label{eq:h21supp}
    \bigg(\frac{d}{d \tau} - A^{\odot \star}(\tau) + {i \omega} \bigg)j h_{21}(\tau)  &= [B(\tau;h_{20}(\tau),\bar{\varphi}(\tau)) + 2B(\tau;h_{11}(\tau),\varphi(\tau)) \\
    & + C(\tau;\varphi(\tau),\varphi(\tau),\bar{\varphi}(\tau))]r^{\odot \star} - 2 a j\dot{\varphi}(\tau) -2d j\varphi(\tau),
\end{split}
\end{align}
Again, we avoid the computation of the adjoint eigenfunction $\varphi^\odot$ from \cref{subsec:NS}. The solution of the second component of \eqref{eq:h21supp} is given by
\begin{equation*}
    h_{21}(\tau)(\theta) = h_{210}(\tau+\theta) + \frac{2(e^{i \omega \theta} - 1)}{i \omega} (a \dot{q}(\tau+\theta) + d q(\tau+\theta)).
\end{equation*}
Using this representation, the first component of \eqref{eq:h21supp} yields the BVP
\begin{equation*}
    \begin{cases}
    \begin{aligned}
        \dot{h}_{210}(\tau) - \textstyle\sum_{j=0}^{m} M_j(\tau) h_{210}(\tau-\tau_j) + i \omega h_{210}(\tau) &= B(\tau;h_{20}(\tau),\bar{\varphi}(\tau)) 
        + 2B(\tau;h_{11}(\tau),\varphi(\tau)) \\
        &+C(\tau;\varphi(\tau),\varphi(\tau),\bar{\varphi}(\tau)) \\
        & -2a[\dot{q}(\tau) + \textstyle\sum_{j=1}^m \frac{(1-e^{-i\omega \tau_j})}{i \omega} M_j(\tau) \dot{q}(\tau-\tau_j)] \\
        &-2d[q(\tau) + \textstyle\sum_{j=1}^m \frac{(1-e^{-i\omega \tau_j})}{i \omega} M_j(\tau) q(\tau-\tau_j)], \\
        h_{210}(T + \theta) - h_{210}(\theta) &= 0,
    \end{aligned}
    \end{cases}
\end{equation*}
where $\tau \in [0,T]$ and $\theta \in [-\tau_m,0]$. Consider the $2n(LM+1)$ linear equations derived from applying the orthogonal collocation method to this BVP. The factor of $2$ arises from accounting for both the real and imaginary parts of the BVP. Let $\mathbf{p} \in \mathbb{R}^{2nLM}$ be a null-vector of the first $2nLM$ equations of this discretization. The Fredholm alternative for finite-dimensional linear systems tells us that
\begin{align*}
    d_{LM} \hspace{-2pt} &\coloneqq \hspace{-2pt} \frac{\langle \mathbf{p}, B(c_{i,j};h_{20}(c_{i,j}), \bar{\varphi}(c_{i,j})) + 2B(c_{i,j}; h_{11}(c_{i,j}),\varphi(c_{i,j})) + C(c_{i,j},\varphi(c_{i,j}),\varphi(c_{i,j}),\bar{\varphi}(c_{i,j})) \rangle}{2\langle \mathbf{p}, q(c_{i,j}) + \sum_{l=1}^m \frac{(1-e^{-i\omega \tau_l})}{i\omega} M_l(c_{i,j}) q(c_{i,j}-\tau_l) \rangle} \\
    & -a_{LM} \frac{\langle \mathbf{p}, \dot{q}(c_{i,j}) + \sum_{l=1}^m \frac{(1-e^{-i\omega \tau_l})}{i\omega} M_l(c_{i,j}) \dot{q}(c_{i,j} - \tau_l)\rangle}{\langle \mathbf{p}, q(c_{i,j}) + \sum_{l=1}^m \frac{(1-e^{-i\omega \tau_l})}{i\omega} M_l(c_{i,j}) q(c_{i,j}-\tau_l) \rangle}
\end{align*}
is an approximation of the normal form coefficient $d$ from \eqref{eq:normalformcoeff NS}. Here, the second component of the variables appearing in the inner product $\langle \cdot , \cdot \rangle$ on $\mathbb{R}^{2nLM}$ must have the same indexing scheme as $\mathbf{p}$. Recall from \cref{supp:collocationsubsec} that convergence of (order $M$) is guaranteed for the orthogonal collocation method applied to the system above. Therefore, convergence of $d_{LM}$ towards $d$ is also guaranteed.
\end{sloppypar}

\clearpage
\bibliographystyle{siamplain}
\bibliography{references}

\end{document}